\documentclass[10pt]{amsart}
\usepackage{amssymb,latexsym,amsmath}
\input{xypic}
\xyoption{all}

\begin{document}

\title{Notes on the arithmetic of Hilbert modular forms}

\author{\bf A. Raghuram \ \ \and \ \ Naomi Tanabe}
\date{\today}

\maketitle


\def\q{\mathfrak{q}}
\def\p{\mathfrak{p}}
\def\l{\mathfrak{l}}
\def\u{\mathfrak{u}}
\def\a{\mathfrak{a}}
\def\m{\mathfrak{m}}
\def\n{\mathfrak{n}}
\def\g{\mathfrak{g}}
\def\k{\mathfrak{k}}
\def\z{\mathfrak{z}}
\def\h{{\mathfrak h}}
\def\gl{\mathfrak{gl}}
\def\sl{\mathfrak{sl}}

\def\Ext{{\rm Ext}}
\def\Hom{{\rm Hom}}
\def\Ind{{\rm Ind}}

\def\GL{{\rm GL}}
\def\SL{{\rm SL}}
\def\SO{{\rm SO}}
\def\O{{\rm O}}

\def\R{\mathbb{R}}
\def\C{\mathbb{C}}
\def\Z{\mathbb{Z}}
\def\Q{\mathbb{Q}}
\def\A{\mathbb{A}}

\def\w{\wedge}

\def\Cat{\mathcal{C}}
\def\HC{{\rm HC}}
\def\HCat{\Cat^\HC}
\def\proj{{\rm proj}}

\def\to{\rightarrow}
\def\To{\longrightarrow}

\def\1{1\!\!1}
\def\dim{{\rm dim}}

\def\th{^{\rm th}}
\def\isom{\approx}

\def\CE{\mathcal{C}\mathcal{E}}
\def\E{\mathcal{E}}

\def\dis{\displaystyle}
\def\f{{\bf f}}                 
\def\T{{\rm T}}              
\def\omegatil{\tilde{\omega}}  
\def\H{\mathcal{H}}            
\def\Dif{\mathfrak{D}}      
\def\W{W^{\circ}}           
\def\Whit{\mathcal{W}}      
\def\ringO{\mathcal{O}}     
\def\S{\mathcal{S}}      
\def\M{\mathcal{M}}      
\def\K{{\rm K}}          
\def\h{\mathfrak{h}}     
\def\norm{{\rm N}}       


\numberwithin{equation}{section}
\newtheorem{thm}[equation]{Theorem}
\newtheorem{cor}[equation]{Corollary}
\newtheorem{lemma}[equation]{Lemma}
\newtheorem{prop}[equation]{Proposition}
\newtheorem{con}[equation]{Conjecture}
\newtheorem{ass}[equation]{Assumption}
\newtheorem{defn}[equation]{Definition}
\newtheorem{rem}[equation]{Remark}
\newtheorem{exer}[equation]{Exercise}
\newtheorem{exam}[equation]{Example}


\section{Introduction}
\label{sec:intro}

Shimura proved the following fundamental result (see \cite[Theorem 4.3]{shimura-duke}) on the critical values of the standard $L$-function attached to a holomorphic Hilbert modular form. 

\begin{thm}[Shimura]
\label{thm:shimura}
Let ${\bf f}$ be a primitive holomorphic Hilbert modular cusp form of type $(k, \psi)$ over a totally real number field $F$ of degree $n$. Assume that  
the weight $k = (k_1,\dots,k_n)$ satisfies the parity condition
$$
k_1 \equiv \cdots \equiv k_n \pmod{2}. 
$$
Let $k^0 = {\rm min}(k_1,\dots,k_n)$ and $k_0 = {\rm max}(k_1,\dots,k_n)$. 
There exist nonzero complex numbers $u(r, {\bf f}^{\sigma})$ defined for $r \in \Z^n/2\Z^n$ and $\sigma \in {\rm Aut}(\C)$ such that 
for any Hecke character $\chi$ of $\A_F^{\times}$ of finite order, for any integer 
$m$ with 
$$
(k_0 - k^0)/2 < m < (k_0+k^0)/2, 
$$
and for any $\sigma \in {\rm Aut}(\C)$, we have 
$$
\sigma\left(\frac{L_f(m, {\bf f}, \chi)} {(2\pi i)^{mn} \, \tau(\chi) \, u(\epsilon, {\bf f})}\right) \ = \ 
\frac{L_f(m, {\bf f}^{\sigma}, \chi^{\sigma})} {(2\pi i)^{mn} \, \tau(\chi^{\sigma}) \, u(\epsilon, {\bf f^{\sigma}})}, 
$$
where $\epsilon$ is prescribed by: $\chi(a) = {\rm sgn}[a^{\epsilon} N(a)^m]$; the quantity $\tau(\chi)$ is the Gauss sum attached to $\chi$, and 
$L_f(s, {\bf f}, \chi)$ is the (finite part of the) standard $L$-function attached to ${\bf f}$, twisted by $\chi$. 
\end{thm}

The purpose of this semi-expository article is to give another proof of the above theorem, which is rather different from Shimura's proof. However, before proceeding any further, let us mention that our proof is contained in the  union of these papers: Harder \cite{harder-general}, Hida \cite{hida-duke}; see also Dou \cite{dou}. What is different from these papers is an organizational principle based on the period relations proved in 
Raghuram-Shahidi \cite{raghuram-shahidi-imrn} while working in the context of regular algebraic cuspidal automorphic representations. 
The point of view taken in this article is that one need only prove an algebraicity  theorem for the most interesting $L$-value, namely, the central critical value for the standard $L$-function of a sufficiently general type of a cuspidal automorphic representation; this theorem says that the central critical value divided by a suitable power of $2\pi i$ and also divided by a suitable period attached to the representation lies in the rationality field of the representation. 
The period relations mentioned above tell us what happens to these periods upon twisting the underlying representation by algebraic Hecke characters, which then gives us a result for {\it all} critical values of any such twisted $L$-function. This is analogous to Hida \cite[Theorem 8.1]{hida-duke} as well as 
Harder \cite[Theorem on p.86]{harder-general}.

Next, to see that this theorem indeed gives Shimura's theorem `on the nose', one needs the dictionary between holomorphic Hilbert modular forms and automorphic representations of ${\rm GL}_2$ over a totally real number field $F$. The dictionary is certainly well-known to experts, however, it is difficult to find all the details in the literature. Most treatments assume at some point that the narrow class number of $F$ is one, while Shimura makes no such restrictions on $F$. Besides, we could not find anywhere the answer to the question: is the dictionary ${\rm Aut}(\C)$-equivariant? Some of the standard books on Hilbert modular forms like Freitag \cite{freitag}, Garrett \cite{garrett} or van der Geer \cite{geer} do not have what we want; although Garrett's book has a definitive treatment of the action of ${\rm Aut}(\C)$ on spaces of Hilbert modular forms--which is called the `arithmetic structure theorem' in his book. This is our second motivation to write this article; we write down such a dictionary, give enough details to make the presentation self-contained, and also analyze its arithmetic properties. (We refer the reader to 
Blasius-Rogawski \cite{blasius-rogawski} and Harris \cite{harris} for some intimately related arithmetic issues about Hilbert modular forms.) 

\smallskip

We now describe the theorems proved in this paper in greater detail, toward which we need some notation. 
Let $F$ be a totally real field of degree $n$, and let $\A_F$ be the ad\`ele ring. Let $G = {\rm Res}_{F/\Q}(\GL_2)$. 
Given a regular algebraic cuspidal automorphic representation $\Pi$ of 
$G(\A_{\Q}) = {\rm GL}_2(\A_F)$, one knows (from Clozel \cite{clozel}) that there is a pure dominant integral weight 
$\mu$ such that $\Pi$ has a nontrivial contribution to the cohomology of some locally symmetric space of $G$   
with coefficients coming from the dual of the finite dimensional representation with highest weight $\mu$. 
We denote this as $\Pi \in {\rm Coh}(G, \mu^{\sf v})$, for $\mu \in X_0^+(T)$, where $T = {\rm Res}_{F/\Q}(T_2)$ and $T_2$ is the diagonal torus of 
${\rm GL}_2$. Under this assumption on $\Pi$, one knows that its rationality field
${\mathbb Q}(\Pi)$ is a number field, that $\Pi_f$ is defined over this number field, and for any $\sigma \in {\rm Aut}(\C)$ the representation 
${}^{\sigma}\Pi$ is also a regular algebraic cuspidal automorphic representation of ${\rm GL}_2(\A_F)$. 
It is further known that the Whittaker model of $\Pi_f$ carries a ${\mathbb Q}(\Pi)$-structure. Similarly, for every 
$\epsilon = (\epsilon_1,\dots,\epsilon_n) \in \Z^n/2\Z^n = \{\pm\}^n$, a suitable cohomological realization of $\Pi_f$ depending on $\epsilon$  
also carries a rational structure. One defines a period $p^{\epsilon}(\Pi)$ by comparing these rational structures. Such periods came up originally in the works of Eichler and Shimura; in our context, these periods are the same as in the papers of Harder and Hida alluded to above. The periods are defined simultaneously for ${}^{\sigma}\Pi$ for any $\sigma \in {\rm Aut}(\C)$. The first main theorem proved in this article is:

\begin{thm}[Central critical value]
\label{thm:central}
Let $\Pi$ be a regular algebraic cuspidal automorphic representation of $\GL_2(\A_F)$, for a totally real number field $F$ of degree $n = n_F = [F:\Q]$. 
Assume that $s = 1/2$ is critical for the standard $L$-function $L(s,\Pi)$ attached to $\Pi$.  
Then for any $\sigma \in {\rm Aut}(\C)$ we have 
$$
\sigma \left( \frac{L_f(\frac12, \Pi)} {(2\pi i)^{d_{\infty}} p^{(+,\dots, +)}(\Pi)} \right) \ = \ 
\frac{L_f(\frac12, {}^{\sigma}\Pi)} {(2\pi i)^{d_{\infty}} p^{(+,\dots, +)}({}^{\sigma}\Pi)}, 
$$
where $d_{\infty} = d(\Pi_{\infty}) = d({}^{\sigma}\Pi_{\infty})$ is an integer determined by the representation at infinity; see 
Proposition~\ref{prop:archimedean-mess}.

In particular, 
$$
L_f(1/2, \Pi) \ \sim_{\Q(\Pi)} \  (2\pi i)^{d_{\infty}} p^{(+,\dots, +)}(\Pi), 
$$
where, by $\sim_{\Q(\Pi)}$, we mean up to an element of the number field $\Q(\Pi)$.
\end{thm}

Section~\ref{sec:l-value} is devoted to giving a reasonably self-contained proof of the above theorem. 
The following result on all critical values for twisted $L$-functions follows from the period relations proved in \cite{raghuram-shahidi-imrn}.

\begin{cor}[All critical values]
\label{cor:all}
Let $\Pi$ be a regular algebraic cuspidal automorphic representation of $\GL_2(\A_F)$, for a totally real number field $F$ of degree $n = n_F = [F:\Q]$. 
Let $\eta_1,\dots,\eta_n$ be all the infinite places of $F$. 
Assume that $s = \frac12 + m \in \frac12 + \Z$ is critical for the standard $L$-function $L(s,\Pi)$ attached to $\Pi$.  
Then, for any finite order character $\chi$ of $F^{\times}\backslash \A_F^{\times}$, and for any $\sigma \in {\rm Aut}(\C)$ we have 
$$
\sigma \left( \frac{L_f(\frac12 + m, \Pi \otimes \chi)} {(2\pi i)^{d_{\infty}+nm} \, p^{((-1)^m\epsilon_{\chi})}(\Pi) \, \mathcal{G}(\chi)} \right) \ = \ 
\frac{L_f(\frac12 +m, {}^{\sigma}\Pi \otimes {}^{\sigma}\chi)} {(2\pi i)^{d_{\infty}+nm} \, p^{((-1)^m\epsilon_{ {}^{\sigma}\!\chi})}({}^{\sigma}\Pi) 
\, \mathcal{G}({}^{\sigma}\chi)}, 
$$
where $\epsilon_{\chi} = (\chi_{\eta_1}(-1),\dots,\chi_{\eta_n}(-1))$ is the `parity' of $\chi$ determined completely by 
$\chi_{\infty} = \otimes_{j=1}^n \chi_{\eta_j}$; and $\mathcal{G}(\chi)$ is the Gauss sum of $\chi$.

In particular, 
$$
L_f(1/2+m, \Pi \otimes \chi ) \ \sim_{\Q(\Pi, \chi)} \  (2\pi i)^{d_{\infty}+nm} p^{((-1)^m\epsilon_{\chi})}(\Pi) \mathcal{G}(\chi), 
$$
where, by $\sim_{\Q(\Pi, \chi)}$, we mean up to an element of the  compositum of $\Q(\Pi)$ and $\Q(\chi)$.
\end{cor}

To see that the corollary exactly corresponds to Shimura's Theorem~\ref{thm:shimura} above, we need to know the dictionary between primitive holomorphic Hilbert modular forms for $F$ and regular algebraic cuspidal automorphic representations of $\GL_2$ over $F$. The latter part of this paper 
is devoted to this dictionary and its arithmetic properties. The main statements are summarized in the theorem below. 

\begin{thm}[The dictionary]
\label{thm:dictionary}
There is a bijection ${\bf f} \leftrightarrow \Pi$ between 
\begin{itemize}
\item ${\bf f} \in S_k(\mathfrak{n}, \omega)_{\rm prim}$, that is  ${\bf f}$ is a primitive holomorphic Hilbert modular form of weight $k = (k_1,\dots,k_n)$, 
of level $\mathfrak{n}$ and nebentypus character $\omega$ which is a character of $(\mathcal{O}_F/\mathfrak{n})^{\times}$; by primitive we mean it is an 
eigenform for all Hecke operators $T_{\p}$, a newform, and it is normalized as ${\rm c}(\mathcal{O}_F, {\bf f}) = 1$.
\item $\Pi$ is a cuspidal automorphic representation of $\GL_2(\A_F)$ whose representation at infinity 
$\Pi_{\infty} = \otimes_j D_{k_j-1}$, of conductor $\mathfrak{n}$, and central character $\omega_{\Pi} = \tilde{\omega}$--the adelization of $\omega$; 
here $D_l$ is the discrete series representation of $\GL_2(\R)$ with lowest non-negative $K$-type being the character 
$\left(\begin{smallmatrix} \cos{\theta} & -\sin{\theta} \\ 
\sin{\theta}& \cos{\theta} \end{smallmatrix}\right)  \mapsto  e^{-i(l+1) \theta}$, and central character $a \mapsto {\rm sgn}(a)^{l+1}$.
\end{itemize}

This dictionary has the following arithmetic properties: 
\begin{enumerate}
\item ($L$-functions)
For any finite order character $\chi$ of $F^{\times}\backslash \A_F^{\times}$ we have an equality of (completed) $L$-functions:
$$
L(s, \Pi({\bf f}) \otimes \chi) = L(s + (k_0-1)/2, {\bf f}, \chi), 
$$
where the left hand side is the standard $L$-function defined as in Jacquet and Langlands \cite{jacquet-langlands}, and the right hand side is 
defined  via a Dirichlet series as in Shimura \cite{shimura-duke}. 
\item (Algebraicity)
  \begin{enumerate}
  \item if $k_1 \equiv \cdots \equiv k_n \equiv 0 \pmod{2}$ then $\Pi({\bf f})$ is algebraic; 
  \item if $k_1 \equiv \cdots \equiv k_n \equiv 1 \pmod{2}$ then $\Pi({\bf f}) \otimes |\ |^{1/2}$ is algebraic; 
  \item if $k_i \not\equiv k_j  \pmod{2}$ for some $i$ and $j$ then no twist of $\Pi({\bf f})$ is algebraic. 
  \end{enumerate}
Note that (a), (b) and (c) can all be put-together as 
$$k_1 \equiv \cdots \equiv k_n \pmod{2} \iff 
\Pi({\bf f}) \otimes |\ |^{k_0/2} \mbox{is algebraic}. 
$$
\item (Regularity) Suppose now that $k_1 \equiv \cdots \equiv k_n \pmod{2}$. Then $\Pi({\bf f}) \otimes |\ |^{k_0/2}$ is regular exactly when 
each $k_j \geq 2$. 

\item (Galois equivariance)
Let $k_1 \equiv \cdots \equiv k_n \pmod{2}$ with $k_j \geq 2$ for all $j$. Then, for any $\sigma \in {\rm Aut}(\C)$ we have: 
$$
{}^{\sigma}(\Pi({\bf f}) \otimes |\ |^{k_0/2}) = \Pi({\bf f}^{\sigma}) \otimes |\ |^{k_0/2}, 
$$
where the action of $\sigma$ on representations is as in Clozel \cite{clozel} or Waldspurger \cite{waldspurger}, and on Hilbert modular forms 
is as in Shimura \cite{shimura-duke}. 

\item (Rationality field)
Let $\Q({\bf f})$ be the field generated by the Fourier coefficients of ${\bf f}$, and let $\Q(\Pi({\bf f}))$ be the subfield of complex numbers fixed by the set of 
all $\sigma \in {\rm Aut}(\C)$ such that ${}^{\sigma}\Pi({\bf f})_v =  \Pi({\bf f})_v$ for all finite places $v$. Then 
$\Q({\bf f}) = \Q(\Pi({\bf f})).$
\end{enumerate}
\end{thm}

We start by setting up our notations and conventions in Section \ref{sec:prelims}.  In \ref{sec:cohomological} we discuss arithmetic issues of regular algebraic cuspidal automorphic representations where we borrow heavily from Clozel \cite{clozel}. In \ref{sec:periods}, 
we provide a summary of the definition of certain periods and certain relations amongst these periods as in Raghuram-Shahidi \cite{raghuram-shahidi-imrn}. The rest of 
Section~\ref{sec:l-value} is devoted to the proofs of Theorem~\ref{thm:central} and Corollary~\ref{cor:all}. We start Section \ref{sec:hilbert} 
by reviewing the basics of 
holomorphic Hilbert modular forms borrowing from Shimura's article \cite{shimura-duke}; the rest of that section gives a detailed proof of the above dictionary 
and its arithmetic properties. 

\medskip

{\small
\noindent{\it Acknowledgements:} This paper had its origins in a Viennese cafe in February 2007 where Ralf Schmidt and the first author would meet for lunch and generally talk about modular forms, automorphic representations, and new vectors. It became clear then that something was lacking in the literature on Hilbert modular forms. This article is an attempt to fill that gap. Both the authors thank Paul Garrett and Ralf Schmidt for helpful discussions concerning this project. The first author is grateful to an NSF grant (DMS-0856113) which partially supports his research.}

\section{Notations and preliminaries}
\label{sec:prelims}

\subsection{The base field}
Let $F$ denote a totally real number field of degree $n$, $\ringO=\ringO_F$ the ring of integers in $F$. The set of real embeddings of $F$ is denoted 
${\rm Hom}_{\rm field}(F, \C) = {\rm Hom}_{\rm field}(F, \R) = \{\eta_1,\dots,\eta_n\}$. We will also let $S_{\infty}$ denote this set of all real places. 
We write $F_+$ for the set of all the totally positive elements in $F$. (A totally positive element means an element $\alpha$ in $F$ such that $\eta_j(\alpha) > 0$ for all $j=1, \dots, n$.) Let $F_{\infty} = \prod_{v \in S_{\infty}} F_v = \prod_{j=1}^n F_{\eta_j} \simeq \prod_{j=1}^n \R$, and 
$F_{\infty^+}$ the subset of all $(x_1, \dots, x_n)$ in $F_{\infty}$ such that $x_j > 0$ for all $j$. Let $\A_F$ denote the ad\`ele ring of $F$, and 
$\A_{F,f}$ the finite ad\`eles; we will drop the subscript $F$ for the field $\Q$. Hence $\A_F = \A \otimes_{\Q} F$, etc. 

Let $\p$ denote a prime ideal of $\ringO$, $F_{\p}$ the completion of $F$ at $\p$, and $\ringO_{\p}$ the ring of integers of $F_{\p}$.  The unique maximal ideal of $\ringO_{\p}$ is $\p\ringO_{\p}$ and is generated by a uniformizer $\varpi_{\p}$. Let $\mathfrak{D}_F$ denote the absolute different of $F$, i.e., 
$\mathfrak{D}_F^{-1} = \{x \in F : T_{F/\Q}(x \ringO) \subset \Z\}$. Let $\mathfrak{D}_F = \prod_{\p} \p^{r_{\p}}.$ Let $\mathfrak{d}_F$ denote the 
absolute discriminant of $F$. Let ${\sf d}_F \in \ringO_F$ be such that ${\rm ord }_{\p}({\sf d}_F) = r_{\p} = {\rm ord}(\mathfrak{D}_F)$; this is possible by strong approximation.

\subsection{The narrow class group of $F$}
\label{class_gp}
 By the narrow class group, we mean the group  $F^\times\backslash\A_F^\times/F^\times_{\infty^+}\prod \ringO_\p^\times$, and the cardinality of this group, which is denoted as $h=h_F$, is called the narrow class number. The narrow class group can be also viewed as the group $J_F/P_{F_+}$ of all fractional ideals of $F$ modulo principal ideals generated by totally positive elements in $F$. The narrow class group is, in general, bigger than the class group $J_F/P_F$, and one has the following exact sequence. (See, for instance, \cite[Section VI. 1]{neukirch}.)

\[ 1 \To \ringO^\times/\ringO^\times_+ \To F_\infty^\times/F_{\infty^+}^\times \To J_F/P_{F_+} \To J_F/P_F \To 1.\]

\subsection{The groups $G \supset B \supset T \supset Z$}
Let $G = {\rm Res}_{F/\Q}(\GL_2)$ which is the Weil restriction of scalars from $F$ to $\Q$ of the algebraic group $\GL_2$ over $F$. 
Hence $G(\Q) = \GL_2(F)$ and more generally, for any $\Q$-algebra $A$ we have $G(A) = \GL_2(A \otimes_{\Q} F)$. For any finite prime $p$, 
$G(\Q_p) = \prod_{\p | p} \GL_2(F_{\p})$; similarly, $G_{\infty} := G(\R) = \prod_{j=1}^n \GL_2(F_{\eta_j}) = \prod_{j=1}^n \GL_2(\R).$ Let 
$\g_{\infty}$ be the complexified Lie algebra of $G_{\infty}$.

Fix the standard Borel subgroup $B = {\rm Res}_{F/\Q}(B_2)$, with $B_2$ being the standard Borel subgroup of $\GL_2$ of all upper triangular matrices. 
Let $T = {\rm Res}_{F/\Q}(T_2)$, where $T_2$ stands for the diagonal torus in $\GL_2$. Let $Z = {\rm Res}_{F/\Q}(Z_2)$, where 
$Z_2$ is the center of $\GL_2$ consisting of scalar matrices. For any $\Q$-algebra $A$, we can talk of $B(A)$, $T(A)$ and $Z(A)$ as we did for $G$.

\subsection{Maximal `compact' subgroup}
Let $K_{\infty}$ stand for the maximal compact subgroup of $G_{\infty}$ thickened by its center. Hence 
$$
K_{\infty} = \prod_{j=1}^n ({\rm O}_2(\R) Z_2(\R))
$$
where ${\rm O}_2(\R)$ is the usual maximal compact subgroup of $\GL_2(\R)$. 
For any Lie group $\mathcal{G}$ we will denote $\mathcal{G}^0$ the connected component of the identity, and $\pi_0(\mathcal{G}) := 
\mathcal{G}/\mathcal{G}^0$ denotes the group of connected components. 
Observe then that 
$$
K_{\infty}^0 = \prod_{j=1}^n ({\rm SO}_2(\R) Z_2(\R)^0)
$$
and that $\pi_0(G_{\infty}) = \pi_0(K_{\infty}) = K_{\infty} /K_{\infty}^0 \simeq (\Z/2\Z)^n$. We will identify the dual group $(K_{\infty} /K_{\infty}^0)^{\widehat{}}$ with 
$(\Z/2\Z)^n = \{\pm\}^n$, with the $+$ (resp., $-$) denoting the trivial (resp., nontrivial) character of ${\rm O}_2(\R)/{\rm SO}_2(\R)$.

Let $\k_{\infty}$ be the complefixed Lie algebra of $K_{\infty}$ or $K_{\infty}^0$; we will use similar `standard' notation for the complexified Lie algebras of other Lie groups.

\subsection{Finite-dimensional representations}
Any $t \in T_{\infty}$ looks like 
$t = (t_j)_j \in \prod_{j=1}^n T_2(F_{\eta_j}) = \prod_{j=1}^n T_2(\R)$. We will also write $t \in T_{\infty}$ as: 
$$
t = \left(\begin{pmatrix}x_1& 0 \\ 0 & y_1 \end{pmatrix}, \begin{pmatrix}x_2& 0 \\ 0 & y_2 \end{pmatrix}, \dots,\begin{pmatrix}x_n & 0 \\ 0 & y_n \end{pmatrix}\right). 
$$
Let $\mu = (\mu_1,\dots,\mu_n)$ be an integral weight for $T_{\infty}$, i.e., each $\mu_j = (a_j,b_j) \in \Z^2$ and we have 
$$
\mu(t) = \prod_j \mu_j(t_j) = \prod_j x_j^{a_j} y_j^{b_j}.
$$
Let $X(T_{\infty})$ stand for set of all integral weights. Let $X^+(T_{\infty})$ be the subset of dominant integral weights; dominant for the choice of Borel subgroup being $B$. A weight $\mu \in X(T_{\infty})$ as above is dominant if and only if $a_j \geq b_j$ for all $1 \leq j \leq n$.

For $\mu \in X^+(T_{\infty})$, we let $E_{\mu}$ stand for the irreducible finite-dimensional representation of $G(\C)$ of highest weight $\mu$. 
Since $G(\C) = \prod_{j=1}^n \GL_2(\C)$, it is clear that $E_{\mu} = \otimes_j E_{\mu_j}$ with $E_{\mu_j}$ being the irreducible finite-dimensional 
representation of $\GL_2(\C)$ of highest weight $\mu_j$. 
Since $\mu_j = (a_j,b_j)$ it is well-known that 
$$
E_{\mu_j} = {\rm Sym}^{a_j-b_j}(\C^2) \otimes {\rm det}^{b_j}
$$
where $\C^2$ is the standard representation of $\GL_2(\C)$.
We let $E_{\mu}^{\sf v}$ stand for the contragredient representation; $E_{\mu}^{\sf v} = E_{\mu^{\sf v}}$ where 
$\mu^{\sf v} = (\mu_1^{\sf v},\dots,\mu_n^{\sf v})$ with $\mu_j^{\sf v} = (-b_j,-a_j)$. Hence, 
$$
E_{\mu_j}^{\sf v} = {\rm Sym}^{a_j-b_j}(\C^2) \otimes {\rm det}^{-a_j}.
$$

\subsection{Automorphic representations}
Following Borel--Jacquet \cite[Section 4.6]{borel-jacquet}, we say an irreducible 
representation of $G({\mathbb A}_F)$ is automorphic if it is isomorphic to an 
irreducible subquotient of the representation of $G({\mathbb A}_F)$ on its
space of automorphic forms. We say an automorphic representation is cuspidal 
if it is a subrepresentation of the representation of $G({\mathbb A}_F)$ on 
the space of cusp forms $\mathcal{A}_{\rm cusp}(G(F)\backslash G({\mathbb A}_F))$; let 
$V_{\Pi}$ denote the representation space of $\Pi$. 
(In particular, a cuspidal representation need not be unitary.)
For an automorphic representation $\Pi$ of $G({\mathbb A}_F)$, we have
$\Pi = \Pi_{\infty} \otimes \Pi_f$, where $\Pi_{\infty} =
\otimes_{v \in S_{\infty}} \Pi_v$ is an irreducible representation of $G_{\infty}$,
and $\Pi_f = \otimes_{v \notin S_{\infty}} \Pi_v$, which is a restricted tensor product, 
is an irreducible representation of $G({\mathbb A}_f)$.

\subsection{Measures and absolute values}
The normalized absolute value for any local field $L$ is denoted $|\ |$, and occasionally we might write 
$|\ |_L$. The product of all the local absolute values gives the ad\`elic norm $|\ |$ on $\A_F^{\times}$. All the measures used will be Haar measures, or measures on quotient spaces derived from Haar measures. We will simply denote the underlying measure by $dx$ or $dg$; the measures are normalized in the usual or `obvious' way. For example, locally $\ringO_{\p}^{\times}$ has volume $1$, and similarly, so does $\GL_2(\ringO_{\p}).$ The global measures 
on $\A_F^{\times}$ and $\GL_2(\A_F)$ are the product measures of local measures, etc.

\subsection{Additive character $\psi$ and Gauss sums}
\label{sec:additive-character}
We fix, once and for all, an additive character $\psi_{\Q}$ of ${\mathbb Q} \backslash {\mathbb A}$, as in Tate's thesis, namely, 
$\psi_{\Q}(x) = e^{2\pi i \lambda(x)}$ with the $\lambda$ as defined in \cite[Section 2.2]{tate}. In particular, 
$\lambda = \sum_{p \leq \infty} \lambda_p$; $\lambda_{\infty}(t) = -t$ for any $t \in \R$; $\lambda_p(x)$ for any $x \in \Q_p$ 
is that rational number with only $p$-power denominator such that $x - \lambda_p(x) \in \Z_p$. If we write 
$\psi_{\Q} =  \psi_{\R} \otimes \otimes_p \psi_{\Q_p}$, 
then $\psi_{\R}(t) = e^{-2\pi it}$ and $\psi_{\Q_p}$ is trivial on $\Z_p$ and nontrivial on $p^{-1}\Z_p$.

Next, we define a character $\psi$ of $F\backslash \A_F$ by composing $\psi_{\Q}$ with the 
trace map from $F$ to $\Q$: $\psi = \psi_{\Q} \circ T_{F/\Q}$. If $\psi = \otimes_v \psi_v$, then the local characters are determined analogously. 
In particular, for all prime ideals $\p$,  suppose $\p^{r_{\p}}$ is the highest power of $\p$ dividing the different $\mathfrak{D}_F$,  
then the conductor of the local character $\psi_{\p}$ is $\p^{-r_{\p}}$, i.e., $\psi_{\p}$ is trivial on $\p^{-r_{\p}}$ and nontrivial on $\p^{-r_{\p}-1}.$

For a Hecke character $\xi$ of $F$, by which we mean a continuous homomorphism 
$\xi: F^{\times}\backslash {\mathbb A}_F^{\times} \to {\mathbb C}^{\times}$, 
following Weil \cite[Chapter VII, Section 7]{weil}, we define the Gauss sum of $\xi$ as follows: 
We let $\mathfrak{c}$ stand for the conductor ideal of $\xi_f$. Let 
$y = (y_{\p})_{\p} \in {\mathbb A}_{F,f}^{\times}$ be such that 
${\rm ord}_{\p}(y_{\p}) = -{\rm ord}_{\p}(\mathfrak{c}) - r_{\p}$.  
The Gauss sum of $\xi$ is defined as  $\mathcal{G}(\xi_f,\psi_f,y) = \prod_{\p} \mathcal{G}(\xi_{\p},\psi_{\p},y_{\p})$
where the local Gauss sum $\mathcal{G}(\xi_{\p},\psi_{\p},y_{\p})$ is defined as
$$
\mathcal{G}(\xi_{\p},\psi_{\p},y_{\p}) = \int_{\mathcal{O}_{\p}^{\times}} \xi_{\p}(u_{\p})^{-1}\psi_{\p}(y_{\p}u_{\p})\, du_{\p}.
$$
For almost all $\p$, where everything in sight is unramified, we have $\mathcal{G}(\xi_{\p},\psi_{\p},y_{\p}) =1$, and 
for all $\p$ we have $\mathcal{G}(\xi_{\p},\psi_{\p},y_{\p})  \neq 0$. 
Note that, unlike Weil, we do not normalize the Gauss sum to make it have absolute value one and we do not have any factor at infinity. Suppressing the dependence on $\psi$ and $y$, we denote $\mathcal{G}(\xi_f,\psi_f,y)$ simply 
by $\mathcal{G}(\xi_f)$ or even $\mathcal{G}(\xi)$.

\subsection{Whittaker models}
We will often be working with Whittaker models, and without any ado we will freely use these standard results. 
(See, for example, Bump~\cite[Chapters 3,4]{bump}.)

\begin{thm}[Local Whittaker Models]\label{local_whittaker}
For any place $v$ of $F$, let $\Pi_v$ be an irreducible admissible infinite-dimensional representation of $\GL_2(F_v)$. Then there exists a unique space 
$\Whit(\Pi_v, \psi_v)$ of smooth functions invariant under right translations by elements of  $\GL_2(F_v)$ such that for any function $W\in \Whit(\Pi_v, \psi_v)$
\[ W\left( \left( \begin{array}{ll} 1 & x \\ & 1 \end{array} \right)g \right)=\psi_v(x)W(g), \hskip 0.2in \text{for} \,\, x\in F_v \, \text{and} \,\, g\in\GL_2(F_v), \]
and the representation of $\GL_2(F_v)$ on the space $\Whit(\Pi_v, \psi_v)$  
is equivalent to the representation $\Pi_v$. This space $\Whit(\Pi_v, \psi_v)$ is called a (local) Whittaker model for $\Pi_v$. 
\end{thm}

\begin{thm}[Global Whittaker Models]\label{whittaker}
Let $\A:=\A_F$ be the ad\`ele ring of a number field $F$, and $(\Pi, V_\Pi)$ a cuspidal automorphic representation of $\GL_2(\A)$. Then there exists a unique Whittaker model $\Whit(\Pi, \psi)$ for $\Pi$ with respect to a non-trivial additive character $\psi$, 
that consists of finite linear combinations of functions given by
\[ W_\phi(g):=\int_{\A/ F} \phi\left(\left(\begin{array}{ll} 1 & x\\  & 1 \end{array}\right)g\right)\overline{\psi(x)}\,dx,\]
where $\phi\in V_\Pi$ and $g\in \GL_2(\A)$. This space decomposes as a restricted tensor product of local Whittaker models.
\end{thm}

\section{Central critical value}
\label{sec:l-value}

The purpose of this section is to give a self-contained proof of Theorem~\ref{thm:central} which is based on a cohomological interpretation of the classical Mellin transform.  
In \ref{sec:cohomological} we define regular algebraic cuspidal automorphic representations where we borrow heavily from Clozel \cite{clozel}, and also record some well-known arithmetic properties of such representations. In \ref{sec:periods} we provide a summary of the definition of periods attached to such a representation $\Pi$; these periods arise via a comparison of a rational structure on a Whittaker model of $\Pi$ with a rational structure on a cohomological realization of $\Pi$. We also record certain relations amongst these periods as in Raghuram-Shahidi \cite{raghuram-shahidi-imrn}. The rest of 
Section~\ref{sec:l-value} is devoted to the proofs of Theorem~\ref{thm:central} and Corollary~\ref{cor:all}. 

\subsection{Cohomological automorphic representations}
\label{sec:cohomological}

\subsubsection{Cuspidal cohomology} 
\label{sec:cuspidal-cohomology}
For any open-compact subgroup $K_f \subset G(\A_f)$ define the space
$$
S^G_{K_f} := G(\Q)\backslash G(\A)/ K_{\infty}^0K_f = \GL_2(F)\backslash \GL_2(\A_F)/ K_{\infty}^0K_f. 
$$
This is an example of a locally symmetric space, because such a space is a finite disjoint union of its connected components which are all of the form 
$\Gamma \backslash G(\R)^0/K_{\infty}^0$ for an arithmetic subgroup $\Gamma$ of $G(\R)^0$; locally it looks like the symmetric space 
$G(\R)^0/K_{\infty}^0$. In the literature on Hilbert modular forms, these spaces also go by the appellation Hilbert--Blumenthal varieties. (See, for example, 
Ghate~\cite[Section 2.2]{ghate}.)

Let $\mu = (\mu_1,\dots,\mu_n) \in X^+(T)$. The representation $E_{\mu}^{\sf v}$ defines a local system $\E_{\mu}^{\sf v}$ on $S^G_{K_f}$. 
(Working with the dual $E_{\mu}^{\sf v}$ instead of just $E_{\mu}$ is for convenience which will become clear later on.) We are interested in the 
sheaf cohomology groups 
$$
H^{\bullet}(S^G_{K_f} , \E_{\mu}^{\sf v}).
$$
It is convenient to pass to the limit over all open-compact subgroups $K_f$ and let 
$$
H^{\bullet}(S^G, \E_{\mu}^{\sf v}) := \varinjlim_{K_f} H^{\bullet}(S^G_{K_f} , \E_{\mu}^{\sf v}).  
$$
There is an action of $\pi_0(G_{\infty}) \times G(\A_f)$ on $H^{\bullet}(S^G, \E_{\mu}^{\sf v})$, which is usually called a Hecke-action, and 
one can always recover the cohomology of $S^G_{K_f}$ by taking invariants: 
$$
H^{\bullet}(S^G_{K_f} , \E_{\mu}^{\sf v}) = H^{\bullet}(S^G, \E_{\mu}^{\sf v})^{K_f}.
$$

We can compute the above sheaf cohomology via the de~Rham complex, and then reinterpreting the de~Rham complex in terms of the complex computing relative Lie algebra cohomology, we get the isomorphism: 
$$
H^{\bullet}(S^G, \E_{\mu}^{\sf v})  \simeq H^{\bullet}(\g_{\infty}, K_{\infty}^0; \  C^{\infty}(G(\Q)\backslash G(\A)) \otimes E_{\mu}^{\sf v}) .
$$
With level structure $K_f$ this takes the form: 
$$
H^{\bullet}(S^G_{K_f}, \E_{\mu}^{\sf v})  \simeq H^{\bullet}(\g_{\infty}, K_{\infty}^0; \  C^{\infty}(G(\Q)\backslash G(\A))^{K_f} \otimes E_{\mu}^{\sf v}) .
$$
The inclusion $C^{\infty}_{\rm cusp} (G(\Q)\backslash G(\A)) \hookrightarrow C^{\infty}(G(\Q)\backslash G(\A))$ of the space of smooth cusp forms  in the space of all smooth functions induces, via results of Borel \cite{borel-duke},  an injection in cohomology; this defines cuspidal cohomology: 
$$
\xymatrix{
H^{\bullet}(S^G, \E_{\mu}^{\sf v}) 
\ar[rr] & &
H^{\bullet}(\g_{\infty},K_{\infty}^0; C^{\infty}(G(\Q)\backslash G(\A)) \otimes E_{\mu}^{\sf v})  \\
H^{\bullet}_{\rm cusp}(S^G, \E_{\mu}^{\sf v}) \ar@{^{(}->}[u]
\ar[rr] 
& & 
H^{\bullet}(\g_{\infty},K_{\infty}^0; C^{\infty}_{\rm cusp}(G(\Q)\backslash G(\A)) \otimes E_{\mu}^{\sf v}) \ar@{^{(}->}[u]
}$$
Using the usual decomposition of the space of cusp forms into a direct sum of cuspidal automorphic representations, we get the following fundamental decomposition of $\pi_0(G_{\infty}) \times G(\A_f)$-modules: 
\begin{equation}
H^{\bullet}_{\rm cusp}(S^G, \E_{\mu}^{\sf v}) = \bigoplus_{\Pi} H^{\bullet}(\g_{\infty},K_{\infty}^0;  \Pi_{\infty} \otimes  E_{\mu}^{\sf v}) \otimes \Pi_f
\end{equation}

We say that {\it $\Pi$ contributes to the cuspidal cohomology of $G$ with coefficients in $E_{\mu}^{\sf v}$} if $\Pi$ has a nonzero contribution to the above decomposition. Equivalently, if $\Pi$ is a cuspidal automorphic representation whose representation at infinity $\Pi_{\infty}$ after twisting by $E_{\mu}^{\sf v}$
has nontrivial relative Lie algebra cohomology. In this situation, we write $\Pi \in {\rm Coh}(G, \mu^{\sf v})$. 

Whether $\Pi$ contributes to cuspidal cohomology or not is determined entirely by its infinite component $\Pi_{\infty}$. This is a very well-known and somewhat surprising fact; surprising because local representations at infinity seem to have a such a strong control over a global phenomenon. 
Further, it was observed by Clozel that this property is in fact captured purely in terms of certain exponents of characters of $\C^*$ appearing in the Langlands parameter of $\Pi_{\infty}$. We now proceed to describe these exponents, for which we need some preliminaries about the local Langlands correspondence; we refer the reader to 
Knapp \cite{knapp}.

\subsubsection{The Weil group of $\R$}
Let $W_{\mathbb R}$ be the Weil group of ${\mathbb R}$. Recall 
that as a set it is 
defined as $W_{\mathbb R} = {\mathbb C}^* \cup j{\mathbb C}^*$. The group 
structure is induced from that of ${\mathbb C}^*$ and the relations 
$jzj^{-1} = \overline{z}$ and $j^2= -1 $. There is a homomorphism 
$W_{\mathbb R} \to {\mathbb R}^*$ which sends $z \in {\mathbb C}^*$ to 
$|z|_{\C} = z\bar{z}$ and sends $j$ to $-1$. This homomorphism induces an isomorphism 
of the abelianization $W_{\mathbb R}^{\rm ab} \to {\mathbb R}^*$.

Let us recall the classification of two-dimensional semi-simple representations of $W_{\R}$. To begin, 
any (quasi-)character $\xi$ of $\C^*$ looks like $\xi_{(s,w)} : \C^* \to \C^*$ with 
$$
\xi_{(s,w)}(z) = z^s \bar{z}^w, \ \ \ {\rm or} \ \ \ \xi_{(s,w)}(r e^{i \theta}) = r^{s+w} e^{i (s-w)\theta},  
$$
where $s, w \in \C$ and $s - w \in \Z$. As alluded to above, the complex absolute value is $|z|_{\C} := z\bar{z} = \xi_{(1,1)}(z)$. 
A character $\xi_{(s,w)}$ is unitary, i.e., takes values in $\C^1 = \{z \in \C^* : |z|_{\C} = 1\}$, if and only if $w = -s$ in which case $s \in \frac12 \Z$. 
In other words, any unitary character of $\C^*$ is of the form $\xi_l$ for $l \in \Z$, where 
$$
\xi_l(z) = \left(\frac{z}{\bar{z}}\right)^{l/2} = \left(\frac{z}{\sqrt{|z|_{\C}}}\right)^l, \ \ \ {\rm or} \ \ \  \xi_l(re^{i \theta}) = e^{i l \theta}.
$$
Next, any character $\chi$ of $\R^*$ looks 
like $\chi_{(s,\epsilon)} : \R^* \to \C^*$ with 
$$
\chi_{(s,\epsilon)}(t) = |t|^s {\rm sgn}(t)^{\epsilon}
$$
where $s \in \C$ and $\epsilon$ is in $\{0,1\}$. 
Via the isomorphism $W_{\R}^{\rm ab} \to \R^*$ any character $\theta$ of $W_{\R}$ also looks like $\theta_{(s,\epsilon)} : W_{\R} \to \C^*$ with 
$$
\theta_{(s,\epsilon)}(z) = (z \bar{z})^s \ \ \ {\rm and} \ \ \ \theta_{(s,\epsilon)}(j) = (-1)^{\epsilon}.
$$ 
Henceforth, we identify the character $\theta_{(s,\epsilon)}$ of $W_{\R}$ with the character $\chi_{(s,\epsilon)}$ of $\R^*$.  
Let $\epsilon_{\R} : W_{\mathbb R} \to \{\pm 1\}$ denote the sign 
homomorphism, defined as $\epsilon_{\R}(z) = 1$ and $\epsilon_{\R}(j)=-1$, i.e., $\epsilon_{\R} = \chi_{(0,1)}$. The usual absolute value of 
a real number $t$ is denoted $|t|$ and this gives an absolute value $|\ |_{\R}$ on $W_{\R}$, defined as $\chi_{(1,0)}$. The restriction of $|\ |_{\R}$ to $\C^*$ via $\C^* \hookrightarrow W_{\R}$ gives $|\ |_{\C}$ on $\C^*$.  
Since $W_{\R}$ contains  $\C^*$ as an abelian subgroup of index two, it is an easy exercise to see that any two-dimensional semi-simple representation 
$\tau$ is one of these two-kinds:
\begin{enumerate}
\item an irreducible 2-dimensional representation; $\tau = \tau(l,t)$ parametrized by pairs $(l, t)$ with $l \geq 1$ an integer and $t \in \C$ where 
$$
\tau(l,t) = {\rm Ind}_{\C^*}^{W_\R}(\xi_l) \otimes |\ |_{\R}^t = {\rm Ind}_{\C^*}^{W_\R}(\xi_l  \otimes |\ |_{\C}^t).
$$

\item a reducible 2-dimensional semi-simple representation; $\tau = \tau(\chi_1, \chi_2)$ with characters $\chi_i = \chi_{(s_i,\epsilon_i)}$ of 
$W_{\R}$, where 
$$
\tau(\chi_1, \chi_2) = \chi_1 \oplus \chi_2.
$$
\end{enumerate}

\subsubsection{Irreducible admissible representations of $\GL_2(\R)$} 
Let us recall the Langlands classification for $\GL_2(\R)$. 
Let $\chi_1, \chi_2$ be characters of $\R^*$ such that $\chi_i = \chi_{(s_i,\epsilon_i)}$. 
Let $I(\chi_1,\chi_2)$ be the normalized parabolic induction of the character $\chi_1 \otimes \chi_2$ of the standard Borel subgroup to all of $\GL_2(\R)$. 
Suppose that $\Re(s_1) \geq \Re(s_2)$ then $I(\chi_1,\chi_2)$ has a unique irreducible quotient, called the Langlands quotient, which we denote 
as $J(\chi_1,\chi_2)$. The induced representation $I(\chi_1,\chi_2)$ is reducible if and only if $s_1 - s_2 = l \in \Z_{\geq 1}$; in this case 
the Langlands quotient is, up to a twist, the irreducible finite-dimensional sub-quotient of dimension $l$, and the other piece is a twist of the 
discrete series representation $D_l$ which we now define. (Later we will give this exact sequence precisely.) 
For any integer $l \geq 1$, let $D_l$ stand for the discrete series representation with lowest non-negative $K$-type being the character 
$\left(\begin{smallmatrix} \cos{\theta} & -\sin{\theta} \\ 
\sin{\theta}& \cos{\theta} \end{smallmatrix}\right)  \mapsto  e^{ - i (l+1) \theta}$, and central character $a \mapsto {\rm sgn}(a)^{l+1}$. 
Note the shift from $l$ to $l+1$. The representation at infinity for a holomorphic  elliptic modular cusp form of weight $k$ is $D_{k-1}$.
The Langlands classification states that any irreducible admissible representation of $\GL_2(\R)$ is, up to equivalence, one of these: 
\begin{enumerate}
\item $D_l \otimes |\ |_{\R}^t$, for an integer $l \geq 1$ and $t \in \C$; or 
\item $J(\chi_1,\chi_2)$, for characters $\chi_i = \chi_{(s_i,\epsilon_i)}$ of $\R^*$ with $\Re(s_1) \geq \Re(s_2)$.
\end{enumerate}

\subsubsection{The local Langlands correspondence for $\GL_2(\R)$} 
\label{sec:llc-gl2r}
There is a canonical bijection $\pi \leftrightarrow \tau$ between equivalence classes of irreducible admissible representations $\pi$ of $\GL_2(\R)$ and 
equivalence classes of two dimensional semi-simple representations $\tau = \tau(\pi)$ of $W_{\R}$. We call $\tau$ the {\it Langlands parameter} of 
$\pi$. From the above classifications it is clear that under this correspondence, we have 
\begin{enumerate}
\item $\pi = D_l \otimes |\ |_{\R}^t \ \ \leftrightarrow \ \ \tau = \tau(l,t)$; for an integer $l \geq 1$ and $t \in \C$; and 
\item $\pi = J(\chi_1,\chi_2) \ \ \leftrightarrow \ \ \tau = \tau(\chi_1, \chi_2)$; for characters $\chi_i = \chi_{(s_i,\epsilon_i)}$ of $\R^*$.  
\end{enumerate}
In the second case, given $\chi_1$ and $\chi_2$, if necessary we reorder them such that $\Re(s_1) \geq \Re(s_2)$  which ensures that $J(\chi_1,\chi_2)$ is defined, while noting that reordering them does not change the equivalence class of $\tau(\chi_1, \chi_2)$. This bijection is canonical in that it preserves local factors and is equivariant under twisting. The local $L$-factor of an irreducible representation $\tau$ of $W_{\mathbb R}$ is as follows. (See Knapp \cite{knapp}.)
$$
L(s,\tau) = \left\{\begin{array}{ll}
\pi^{-(s+t)/2} \, \Gamma\left(\frac{s+t}{2}\right) &
\mbox{if $\tau = |\  |_{\mathbb R}^t$,} \\
\pi^{-(s+t+1)/2} \, \Gamma\left(\frac{s+t+1}{2}\right) &
\mbox{if $\tau = \epsilon_\R \otimes |\  |_{\mathbb R}^t$,} \\
2(2\pi)^{-(s+t+l/2)}\ \, \Gamma(s+t + l/2 ) &
\mbox{if $\tau = {\rm Ind}_{\C^*}^{W_\R}(\xi_l) \otimes |\ |_{\R}^t $ with
$l \geq 1$.}
\end{array}\right.
$$

\subsubsection{Algebraic automorphic representation} (See Clozel \cite[p.89]{clozel}.) 
\label{sec:algebraic}
Let $\Pi$ be an irreducible automorphic representation of $\GL_2(\A_F)$. 
We will work over a totally real number field $F$. 
The representation at infinity $\Pi_{\infty}$ is a tensor product
$$
\Pi_{\infty} = \otimes_{\eta \in S_{\infty}} \Pi_{\eta} = \Pi_{\eta_1} \otimes \cdots \otimes \Pi_{\eta_n}, 
$$
where $\Pi_{\eta}$ is an irreducible admissible representation of $\GL_2(F_\eta) = \GL_2(\R)$. For $1 \leq j \leq n$, let $\tau_j$ be the Langlands parameter of 
$\Pi_{\eta_j}$. The restriction of $\tau_j$ to $\C^*$ is a direct sum of characters: 
$$
\tau_j|_{\C^*} = \xi_{j_1} \oplus \xi_{j_2}, 
$$
with $\xi_{j_i} = \xi_{(s_{j_i},w_{j_i})}$. We say that an irreducible automorphic representation $\Pi$ is {\it algebraic} if 
$$
s_{j_i} = \frac12 + p_{j_i}, \ \  w_{j_i} = \frac12 + q_{j_i}, \ {\rm with} \ p_{j_i}, q_{j_i} \in \Z.
$$
(In other words, a global representation $\Pi$ is algebraic if all the exponents appearing in the characters of $\C^*$ coming from the representations 
$\Pi_{\eta}$ at infinity are half plus an integer.)

Note that the data $(s_{j_1}, w_{j_1}, s_{j_2}, w_{j_2})$ depends only on two of these numbers: 
\begin{enumerate}
\item the restriction of $\tau = \tau(l,t)$ to $\C^*$ is given by
$$
{\rm Ind}_{\C^*}^{W_R}(\xi_l  \otimes |\ |_{\C}^t)|_{\C^*} = \xi_l  \otimes |\ |_{\C}^t \oplus \xi_{-l}  \otimes |\ |_{\C}^t, 
$$
which looks like $(z^s \bar{z}^w, z^w \bar{z}^s)$ with $s = l/2 + t$ and $w = -l/2 + t$. 
\item Or,  if $\tau = \tau(\chi_1,\chi_2)$ with $\chi_i = \chi_{(s_i,\epsilon_i)}$, then the restriction of $\tau$ to $\C^*$ 
is $((z\bar{z})^{s_1}, (z\bar{z})^{s_2}).$
\end{enumerate}

\subsubsection{The infinity type of an algebraic automorphic representation} (See Clozel \cite[p.106]{clozel}.)
\label{sec:infinity-type}
Let $\Pi$ be an irreducible algebraic automorphic representation of $\GL_2(\A_F)$. 
The infinity type of $\Pi$ is an element of $\prod_{j=1}^{n} (\Z^2)^{\Z/2}$, i.e., it is an $n$-tuple of unordered pairs of integers, 
and is defined as follows: Consider $\Pi' = \Pi \otimes |\ |^{-1/2}$. Since $\Pi$ is algebraic, all the exponents of the characters of $\C^*$ coming 
from the infinite components of $\Pi'$ are  integers. For each real place $\eta_j$ for $1 \leq j \leq n$, the restriction to $\C^*$ of the Langlands parameter of the representation 
$\Pi'_{\eta_j}$, as described above, looks either like  $(z^{p_j} \bar{z}^{q_j}, z^{q_j} \bar{z}^{p_j})$ or like $((z\bar{z})^{p_j}, (z\bar{z})^{q_j})$ for integers $p_j$ and $q_j$. The infinity type of $\Pi$ is then defined as: 
$$
\infty(\Pi) := (\{p_1,q_1\}, \{p_2,q_2\}, \dots, \{p_n,q_n\}). 
$$

\subsubsection{Regular algebraic cuspidal automorphic representation} (See Clozel \cite[p.111]{clozel}.) 
Let $\Pi$ be an algebraic cuspidal automorphic representation, and suppose $(\{p_1,q_1\}, \{p_2,q_2\}, \dots, \{p_n,q_n\})$
is the infinity type of $\Pi$. We say that $\Pi$ is {\it regular} if  $p_j \neq q_j$ for all $1 \leq j \leq n$. 

A fundamental observation of Clozel is that a cuspidal automorphic representation $\Pi$ is regular algebraic if and only if 
$\Pi$ is of cohomological type, i.e., contributes to the cuspidal cohomology--possibly with nontrivial coefficients--of a locally symmetric space attached to 
$\GL_2$ over $F$.

\subsubsection{Infinite components of a regular algebraic cuspidal automorphic representation}
\label{sec:infinite-components}
Let us suppose that $\Pi$ is such a representation, and let us look closely at the possible exponents of the characters of $\C^*$ for the representations at infinity. Suppose one of the representations at infinity looks like $\Pi_{\eta} = J(\chi_1,\chi_2)$. Then its Langlands parameter is 
 $\tau = \tau(\chi_1,\chi_2)$ with $\chi_i = \chi_{(s_i,\epsilon_i)}$; the restriction of $\tau$ to $\C^*$ as mentioned above looks like 
 $((z\bar{z})^{s_1}, (z\bar{z})^{s_2}).$ Since $\Pi$ is algebraic we have $s_i = \frac12 + p_i$ with $p_i \in \Z$. Then $s_1 - s_2 \in \Z$. Since the inducing data is of Langlands type, we have $s_1 - s_2 \geq 0$. Since $\Pi$ is regular, $s_1-s_2 \geq 1$. 
 But then the full induced representation  $I(\chi_1,\chi_2)$ is reducible; hence the Langlands quotient $J(\chi_1,\chi_2)$ is a finite-dimensional representation.  But a cuspidal automorphic representation is globally generic (i.e., has a global Whittaker model) and so locally generic everywhere, and so every local component has to be an infinite-dimensional representation. Hence $\Pi_{\eta}$ cannot be equivalent to $J(\chi_1,\chi_2)$, and has to be of the form $D_l \otimes |\ |_{\R}^t$. In this case the exponents of the characters are $l/2 + t$ and $-l/2 + t$; hence if $l$ is even then $t \in \frac12  \Z$, and if $l$ is odd then $t \in \Z$.  We have just proved that {\it the infinite components of a regular algebraic cuspidal automorphic representation of $\GL_2(\A_F)$ are all discrete series representations twisted by  integral or half-integral powers of absolute value.} Further, there is a compatibility with all these twists afforded by 
the fact that there is a twist of the global representation which makes it unitary; see \ref{sec:purity}.

\subsubsection{Cohomology of a discrete series representation}
\label{sec:gl2r}
We will digress for a moment to observe that discrete series representations of $\GL_2(\R)$, possibly twisted by a half-integral power of absolute value, 
have nontrivial cohomology. For brevity, let $(\g_2,K_2^0) := (\gl_2, {\rm SO}(2)Z_2(\R)^0)$. 
For a dominant integral weight $\nu = (a,b)$, with integers $a \geq b$, the basic fact here is that  
there is a non-split exact sequence of $(\g_2,K_2^0)$-modules: 
\begin{equation}
\label{eqn:exact-seq-dsr}
0 \to 
D_{a-b +1} \otimes |\ |_{\R}^{(a+b)/2}  \to 
{\rm Ind}_{B_2(\R)}^{\GL_2(\R)}(\chi_{(a,a)}|\ |^{1/2} \otimes \chi_{(b,b)}|\ |^{-1/2}) \to
E_{\nu}
\to 0.
\end{equation}
(Recall from our earlier notation that $\chi_{(a,a)}(t) = |t|^a {\rm sgn}(t)^a = t^a$ for any integer $a$.) 
In other words, in the category $\Cat(\g_2,K_2^0)$ of admissible $(\g_2,K_2^0)$-modules, one has
$$
\Ext^1_{\Cat(\g_2,K_2^0)}(E_{\nu}, D_{a-b +1} \otimes |\ |_{\R}^{(a+b)/2}) \neq 0. 
$$
But 
\begin{eqnarray*}
H^1(\g_2, K_2^0; (D_{a-b +1} \otimes |\ |_{\R}^{(a+b)/2}) \otimes E_{\nu}^{\sf v}) & = & 
\Ext^1_{\Cat(\g_2,K_2^0)}( 1\!\!1, (D_{a-b +1} \otimes |\ |_{\R}^{(a+b)/2}) \otimes E_{\nu}^{\sf v}) \\
& = & \Ext^1_{\Cat(\g_2,K_2^0)}(E_{\nu}, D_{a-b +1} \otimes |\ |_{\R}^{(a+b)/2}) \neq 0.
\end{eqnarray*}
Further, it is well-known that $H^{\bullet}(\g_2, K_2^0; (D_{a-b +1} \otimes |\ |_{\R}^{(a+b)/2}) \otimes E_{\nu}^{\sf v}) \neq 0$ if and only if ${\bullet} = 1$, and that dimension of $H^1(\g_2, K_2^0; (D_{a-b +1} \otimes |\ |_{\R}^{(a+b)/2}) \otimes E_{\nu}^{\sf v})$ is two, with both the characters of 
${\rm O}(2)/{\rm SO}(2)$ appearing exactly once. (See, for example, Waldspurger \cite[Proposition I.4]{waldspurger}.) This detail will be useful below; 
see \ref{sec:pinning-down}. 
Finally, suppose $H^q(\g_2, K_2^0; \Xi \otimes E_{\nu}^{\sf v}) \neq 0$ for some irreducible admissible infinite-dimensional representation $\Xi$ of 
$\GL_2(\R)$, then the central character restricted to $\R_{>0}$ and the infinitesimal character of $\Xi$ are the same as that of $E_{\nu}$ which can be seen from Wigner's Lemma (Borel-Wallach \cite[Theorem I.4.1]{borel-wallach}). It follows from 
Langlands classification that $\Xi \simeq D_{a-b +1} \otimes |\ |_{\R}^{(a+b)/2}$.

\subsubsection{`Regular algebraic' = `Cohomological'} 
\label{sec:reg-alg-coh}
Let $\Pi$ be a cuspidal automorphic representation of $G(\A_F)$. A point of view afforded by Clozel \cite{clozel} is that 
$$
\Pi \ \mbox{is regular and algebraic} \ \iff \ \Pi \in {\rm Coh}(G,\mu^{\sf v}) \ \mbox{for some $\mu \in X^+(T)$}. 
$$

Let $\Pi \in {\rm Coh}(G,\mu^{\sf v})$. Say, $\mu = (\mu_1,\dots,\mu_n)$, and each $\mu_j = (a_j,b_j)$. Apply the K\"unneth theorem 
(see, for example,  Borel-Wallach \cite[I.1.3]{borel-wallach}) to see that 
$$
H^{\bullet}(\g_{\infty}, K_{\infty}^0; \Pi_{\infty} \otimes E_{\mu}^{\sf v}) = 
\bigoplus_{d_1 + \dots +d_n = \bullet} 
\otimes_{j=1}^n H^{d_j}(\gl_2, {\rm SO}(2)Z_2(\R)^0; \Pi_j \otimes E_{\mu_j}^{\sf v}).
$$
From \ref{sec:gl2r} the right hand side is nonzero only for $d_j =1$ and $\Pi_{\eta_j} = D_{a_j-b_j +1} \otimes |\ |_{\R}^{(a_j+b_j)/2}$. The exponents 
in the Langlands parameter of $\Pi_{\eta_j}$ are therefore given by: 
$$
\tau(\Pi_{\eta_j})(z)  = z^{\frac12 + a_j} \bar{z}^{-\frac12 + b_j}  +  z^{-\frac12 + b_j} \bar{z}^{\frac12 + a_j}, \ \ \forall z \in \C^* \subset W_{\R}.
$$
Hence $\Pi$ is algebraic. Next, working with $\Pi' = \Pi \otimes |\ |^{-1/2}$ we see that the infinity type of $\Pi$ is: 
$$
\infty(\Pi) := (\{a_1,b_1-1\}, \{a_2,b_2-1\}, \dots, \{a_n,b_n-1\}). 
$$
Since $\mu$ is dominant, $a_j \geq b_j$; whence $a_j > b_j-1$, i.e., $\Pi$ is regular and algebraic.

Conversely, let $\Pi$ be a regular algebraic cuspidal automorphic representation of $G(\A_F)$. As in \ref{sec:infinity-type} the infinity type of $\pi$ is given by
$$
\infty(\Pi) := (\{p_1,q_1\}, \{p_2,q_2\}, \dots, \{p_n,q_n\})
$$
for integers $p_j,q_j$ and regularity says that $p_j \neq q_j$. Without loss of generality assume that $p_j > q_j$. Put $a_j = p_j$ and 
$b_j = q_j+1$. Now let $\mu_j = (a_j,b_j)$ and $\mu = (\mu_1,\dots,\mu_n)$. Then $\mu \in X^+(T)$, and it follows from \ref{sec:infinite-components} that 
$\Pi \in {\rm Coh}(G,\mu^{\sf v})$.

\subsubsection{Clozel's purity lemma} 
\label{sec:purity}
(See Clozel \cite[Lemme 4.9]{clozel}.) Let $\mu \in X^+(T)$ be a dominant integral weight as above; 
say, $\mu = (\mu_1,\dots,\mu_n)$, and each $\mu_j = (a_j,b_j)$ with integers $a_j \geq b_j$. The purity lemma says that 
if the weight $\mu$ supports nontrivial cuspidal cohomology, i.e., if ${\rm Coh}(G,\mu^{\sf v})$ is nonempty, then $\mu$ satisfies the `purity' condition: there exists ${\sf w} = {\sf w}(\mu) \in \Z$ such that $a_j + b_j = {\sf w}$ for all $j$. This integer ${\sf w}$ is called the purity weight of $\mu$, and if 
$\Pi \in {\rm Coh}(G,\mu^{\sf v})$, then we will call ${\sf w}$ the purity weight of $\Pi$ as well. (Proof: Given a cuspidal representation $\Pi$, there is a complex number ${\sf w}$ such that the twisted representation $\Pi \otimes |\ |^{\sf w}$ is unitary; if further $\Pi$ is algebraic it follows that ${\sf w}$ must be an integer.) Let us denote the set of all pure dominant integral weights  by $X^+_0(T)$. If we start with a primitive holomorphic Hilbert modular form, as will be the case in the latter part of the paper, then this condition is automatically fulfilled; however, from the perspective of cohomological automorphic representations, the purity of the weight $\mu$ is an important condition to keep in mind.

\subsubsection{Pinning down generators for the cohomology class at infinity}
\label{sec:pinning-down}
Let $\Pi \in {\rm Coh}(G,\mu^{\sf v})$. Say, $\mu = (\mu_1,\dots,\mu_n)$, and each $\mu_j = (a_j,b_j)$.  
The space $H^n(\g_{\infty}, K_{\infty}^0; \Pi_{\infty} \otimes E_{\mu}^{\sf v})$ is acted upon by $K_{\infty}/K_{\infty}^0$. It follows from the 
K\"unneth rule (Borel-Wallach \cite[I.1.3]{borel-wallach}) and \ref{sec:gl2r} that every character of $K_{\infty}/K_{\infty}^0$ appears with multiplicity one in $H^n(\g_{\infty}, K_{\infty}^0; \Pi_{\infty} \otimes E_{\mu}^{\sf v})$. Fix such a character 
$\epsilon = (\epsilon_1,\dots,\epsilon_n)$ of $K_{\infty}/K_{\infty}^0$.
The purpose of this (somewhat tedious) paragraph is to 
fix a basis $[\Pi_{\infty}]^{\epsilon}$ for the one-dimensional vector space 
$H^n(\g_{\infty}, K_{\infty}^0; \Pi_{\infty} \otimes E_{\mu}^{\sf v})(\epsilon)$. 
(See (\ref{eqn:infinity-class}) below, especially when $\epsilon = (+,\dots,+)$.) Since K\"unneth gives: 
$$
H^n(\g_{\infty}, K_{\infty}^0; \Pi_{\infty} \otimes E_{\mu}^{\sf v})(\epsilon) = 
\bigotimes_{j=1}^n H^1(\gl_2, {\rm SO}(2)Z_2(\R)^0; \Pi_j \otimes E_{\mu_j}^{\sf v})(\epsilon_j), 
$$
it suffices to fix a basis $[\Pi_j]^{\epsilon_j}$ for the one-dimensional $H^1(\gl_2, {\rm SO}(2)Z_2(\R)^0; \Pi_j \otimes E_{\mu_j}^{\sf v})(\epsilon_j)$ 
and let 
$$
[\Pi_{\infty}]^{\epsilon} = \bigotimes_{j=1}^n [\Pi_j]^{\epsilon_j}.
$$

We now proceed to fix $[\Pi_j]^{\epsilon_j}$. Since we are working with only one copy of $\GL_2(\R)$, let us omit the subscript $j$ and slightly change our notations: Let $\nu = (\nu_1,\nu_2) \in X^+(T_2)$ be a dominant integral weight for the diagonal torus $T_2(\R)$ in $\GL_2(\R)$, and 
$E_{\nu}$ the corresponding finite-dimensional irreducible representation of $\GL_2(\C)$ of highest weight $\nu$. Let 
$\Xi \isom D_{\nu_1-\nu_2 +1} \otimes |\ |^{(\nu_1+\nu_2)/2}$. For any choice of sign in $\{\pm\} := ({\rm O}(2)/{\rm SO}(2))^{\widehat{}}$, with 
$+$ or $-$ being the trivial or nontrivial character of {\rm O}(2)/{\rm SO}(2) respectively, we will fix 
a 1-cocycle $[\Xi]^{\pm}$ so that  
$$
H^1(\gl_2, {\rm SO}(2)Z_2(\R)^0; \Xi \otimes E_{\nu}^{\sf v})(\pm) \ = \ \C\, [\Xi]^{\pm}.
$$

For any integer $m \geq 1$, let
$E_m$ be the $(m-1)\th$ symmetric power of the standard
(two-dimensional) representation $\C^2$ of $\GL_2(\C)$, i.e., 
$E_m = {\rm Sym}^{m-1}(\C^2)$. The representation $E_m$ is
irreducible and has dimension $m$. Denote the standard basis of 
$\C^2$ by $\{e_1,e_2\}$, which gives the `standard' basis $\{e_2^{m-1}, e_2^{m-2}e_1, \dots, e_1^{m-1}\}$ for $E_m$. This basis will be denoted as 
$\{{\sf s}_0, {\sf s}_1,\dots, {\sf s}_{m-1}\}$, i.e., ${\sf s}_j = e_1^j e_2^{m-1-j}$. The finite-dimensional 
irreducible representation $E_{\nu}$ of $\GL_2(\C)$ with highest weight $\nu$ is $
E_{\nu} = E_{\nu_1-\nu_2 + 1} \otimes {\rm det}^{\nu_2} = {\rm Sym}^{\nu_1-\nu_2}(\C^2) \otimes {\rm det}^{\nu_2}.$
By restriction, $E_{\nu}$ is also a representation of $\GL_2(\R)$. 
The central character of $E_{\nu}$ is given by $a \mapsto \omega_{\nu}(a) = a^{\nu_1+\nu_2}$ for all $a \in \R^*$. 
The contragredient representation of $E_{\nu}$ is denoted $E_{\nu}^{\sf v}$; one has 
$E_{\nu}^{\sf v} = E_{\nu^{\sf v}}$, where $\nu^{\sf v} = (-\nu_2, -\nu_1)$ is the dual weight 
of $\nu$. Explicitly, $E_{\nu}^{\sf v} = E_{\nu_1-\nu_2 + 1} \otimes {\rm det}^{-\nu_1}. $ We will need information on the restriction of 
$E_{\nu}^{\sf v}$ to various subgroups. 
 
In either of the representations $E_{\nu}$ or $E_{\nu^{\sf v}}$, the action of the diagonal torus in ${\rm SL}_2$ on the basis vectors is given by 
$\left(\begin{smallmatrix}t & 0 \\ 0 & t^{-1} \end{smallmatrix}\right) {\sf s}_j = t^{-\nu_1+\nu_2 + 2j} {\sf s}_j$. Hence, the standard basis realizes 
the weights: 
$$
\{-(\nu_1-\nu_2), -(\nu_1-\nu_2)+2,\dots,\nu_1-\nu_2\}. 
$$
In particular, the highest weight vector of $E_{\nu^{\sf v}}$ is given by 
$e^+_{\nu^{\sf v}} := {\sf s}_{\nu_1-\nu_2} = e_1^{\nu_1-\nu_2}.$ Observe that the standard basis gives a $\Q$-structure on $E_{\nu}^{\sf v}$.

The restriction of $E_{\nu}^{\sf v}$ to $\GL_1(\R) \hookrightarrow \GL_2(\R)$ 
is described by $\left(\begin{smallmatrix}t & 0 \\ 0 & 1 \end{smallmatrix}\right) {\sf s}_j = t^{j-\nu_1} {\sf s}_j.$ From this we easily deduce the following 
lemma which will be of use later on; see \ref{sec:criticality} below.
\begin{lemma}
\label{lem:trivial-repn-gl1}
Let $1\!\!1$ denote the trivial representation of $\GL_1(\R)$. Then 
$$
{\rm Hom}_{\GL_1(\R)}(E_{\nu}^{\sf v}, 1\!\!1) \neq 0 \iff \nu_1 \geq 0 \geq \nu_2. 
$$
In this situation, ${\rm Hom}_{\GL_1(\R)}(E_{\nu}^{\sf v}, 1\!\!1)$ is one-dimensional and  a nonzero 
 $\mathcal{T} \in {\rm Hom}_{\GL_1(\R)}(E_{\nu}^{\sf v}, 1\!\!1)$ is given by projecting to the coordinate corresponding to ${\sf s}_{\nu_1}$, i.e., 
$$
\mathcal{T}(\sum_{j=0}^{\nu_1-\nu_2} c_j {\sf s}_j) = c_{\nu_1}.
$$
\end{lemma}

\begin{proof}
This is easy to verify and we omit the proof. Let us mention that this is a special case of well-known classical branching laws from $\GL_n(\C)$ to $\GL_{n-1}(\C)$;  see Goodman-Wallach \cite{goodman-wallach}. 
\end{proof}

The ${\rm SO}(2)$-types of $E_{\nu}$, as well as $E_{\nu}^{\sf v}$, are given by: 
$$
E_{\nu}^{\sf v}|_{K_2^1} = 
E_{\nu}|_{K_2^1} =  \theta_{-(\nu_1-\nu_2)} \oplus  \theta_{-(\nu_1-\nu_2)+2}
\oplus \cdots \oplus \theta_{\nu_1-\nu_2-2} \oplus \theta_{\nu_1-\nu_2}, 
$$
where, for any integer $n$, $\theta_n$ is the character of ${\rm SO}(2)$ given by 
$\theta_n(r(t)) =  e^{-i nt}$, for all 
$r(t) = \left(\begin{smallmatrix} \cos{t} & -\sin{t} \\ \sin{t} & \cos{t} \end{smallmatrix}\right)$ in 
${\rm SO}(2)$. It is necessary to fix an ordered basis giving the above decomposition. 
Let $\{{\sf w}_1, {\sf w}_2\}$ denote the basis for $\C^2$ which diagonalizes the ${\rm SO}(2)$-action: 
$$
{\sf w}_1 = e_1 + ie_2, \  {\sf w}_2 = ie_1 + e_2; 
$$
it is easily verified that
$$
r(t) {\sf w}_1= e^{-it} {\sf w}_1 = \theta_1(r(t)) {\sf w}_1, \ \ 
r(t) {\sf w}_2 = e^{it} {\sf w}_2  = \theta_{-1}(r(t)) {\sf w}_2.
$$
The ordered basis $\{ {\sf w}_2^{\nu_1 - \nu_2}, {\sf w}_2^{\nu_1-\nu_2-1} {\sf w}_1, \dots, {\sf w}_1^{\nu_1-\nu_2}\}$ of $E_{\nu}^{\sf v}$ 
realizes the above decomposition of $E_{\mu}^{\sf v}$ into its $K$-types. 
Let ${\sf w}_{\nu^{\sf v}}^+ = {\sf w}_1^{\nu_1-\nu_2}$ be the basis vector realizing the highest non-negative $K$-type in $E_{\nu}^{\sf v}$, i.e., 
$\theta_{\nu_1-\nu_2}$; similarly, the lowest $K$-type is realized by ${\sf w}_{\nu^{\sf v}}^- = {\sf w}_2^{\nu_1-\nu_2}$. In terms of the standard basis: 
\begin{equation}
\label{eqn:gl2-k-type-to-standard}
\begin{split}
{\sf w}_{\nu^{\sf v}}^+ = & \ (e_1 + ie_2)^{\nu_1-\nu_2} \ = \ 
\sum_{\alpha=0}^{\nu_1-\nu_2} \binom{\nu_1-\nu_2}{\alpha} i^{\nu_1-\nu_2-\alpha} \, {\sf s}_{\alpha} \\
{\sf w}_{\nu^{\sf v}}^- = & \ (ie_1 + e_2)^{\nu_1-\nu_2} \ = \ 
\sum_{\alpha=0}^{\nu_1-\nu_2} \binom{\nu_1-\nu_2}{\alpha} i^{\alpha} \, {\sf s}_{\alpha}. 
\end{split}
\end{equation}

For a dominant integral weight $\nu = (\nu_1,\nu_2)$, and for $\Xi \isom D_{\nu_1-\nu_2 +1} \otimes |\ |^{(\nu_1+\nu_2)/2}$, 
using the exact sequence in (\ref{eqn:exact-seq-dsr}) we deduce that the ${\rm SO}(2)$-types of $\Xi$ are 
$$
\cdots \oplus \theta_{-(\nu_1-\nu_2 +4)} \oplus  \theta_{-(\nu_1-\nu_2 +2)} \oplus 
(\mbox{nothing here})
\oplus \theta_{\nu_1-\nu_2+2} \oplus \theta_{\nu_1-\nu_2+4} \oplus \cdots
$$
The missing $K$-types in {\it (nothing here)} correspond exactly to the $K$-types of $E_{\nu}$.  Let $\phi_{\pm(\nu_1-\nu_2+2)}$ 
be vectors in $\Xi$ with $K$-types $\theta_{\pm(\nu_1-\nu_2+2)}$, respectively. 
The vectors in $\Xi = D_{\nu_1-\nu_2 +1} \otimes |\ |^{(\nu_1+\nu_2)/2}$ may be identified with vectors 
in the induced representation ${\rm Ind}_{B_2}^{G_2}(\chi_{(\nu_1,\nu_1)}|\ |^{1/2} \otimes \chi_{(\nu_2,\nu_2)}|\ |^{-1/2})$ 
which is the middle term in the exact sequence (\ref{eqn:exact-seq-dsr}). In particular,  we may and shall normalize them as 
$$
\phi_{\pm(\nu_1-\nu_2+2)}\left(\left(\begin{smallmatrix}1 & 0 \\ 0 & 1 \end{smallmatrix}\right)\right) = 1.
$$

Recall that $(\g_2,K_2^0) := (\gl_2(\C), {\rm SO}(2)Z_2(\R)^0)$ and let $K_2^1 = {\rm SO}(2)$. 
The cochain complex 
$$
C^{\bullet} := {\rm Hom}_{K_2^0}(\wedge^{\bullet} \g_2/\k_2, \Xi \otimes E_{\nu}^{\sf v})
$$ 
computes $(\g_2,K_2^0)$-cohomology of $\Xi \otimes E_{\nu}^{\sf v}$. Since the central characters of $\Xi$ and $E_{\nu}$ are equal, this complex is same as 
${\rm Hom}_{K_2^1}(\wedge^{\bullet} \g_2/\k_2, \Xi \otimes E_{\nu}^{\sf v})$. It is easy to see 
that $\g_2/\k_2 = \theta_2 \oplus \theta_{-2}$ as a $K_2^1$-module: let $\{{\bf z}_1,{\bf z}_2\}$ be the basis for $\g_2/\k_2$ given by: 
$$
{\bf z}_1 =  i \left(\begin{array}{cc} 1 & i \\ i & -1 \end{array}\right), \ \  
{\bf z}_2 = i \left(\begin{array}{cc}1 & -i \\ -i & -1 \end{array}\right);
$$
it is easily checked that   
$$
{\rm Ad}(r(t))({\bf z}_1) = e^{-2it} {\bf z}_1 = \theta_2(r(t)){\bf z}_1 \ \ {\rm and} \ \ 
{\rm Ad}(r(t))({\bf z}_2) = e^{2it} {\bf z}_2 = \theta_{-2}(r(t)){\bf z}_2. 
$$
From the description of $K$-types of $\Xi$ and $E_{\nu}$ we see that $C^q= 0$ for all $q \neq 1$, and 
$$
C^1 = {\rm Hom}_{K_2^1}(\theta_{-2} \oplus \theta_2, \Xi \otimes E_{\nu}^{\sf v}) \isom \C^2.
$$ 
Fix a basis $\{f_{-2}, f_2\}$ for this two dimensional space $C^1$ as follows: $f_{-2}$ picks up the vector $\phi_{-\nu_1+\nu_2-2} \otimes {\sf w}_{\nu^{\sf v}}^+$ 
realizing the character $\theta_{-2}$; similarly, $f_2$ picks up the vector $\phi_{\nu_1-\nu_2+2} \otimes {\sf w}_{\nu^{\sf v}}^-$ 
realizing the character $\theta_{2}$. More precisely, 
$$
\begin{array}{ll}
f_{-2}({\bf z}_1) = 0, \ &\  f_{-2}({\bf z}_2) = \phi_{-\nu_1+\nu_2-2} \otimes {\sf w}_{\nu^{\sf v}}^+, \\
f_2({\bf z}_1) = \phi_{\nu_1-\nu_2+2} \otimes {\sf w}_{\nu^{\sf v}}^-, \ & \ f_2({\bf z}_2) = 0.
\end{array}
$$
Since $C^1 = {\rm Hom}_{K_2^0}(\g_2/\k_2, \Xi \otimes E_{\nu}^{\sf v}) \simeq 
( (\g_2/\k_2)^* \otimes \Xi \otimes E_{\nu}^{\sf v})^{K_2^0} $ we can transcribe these expressions for $f_{\pm 2}$ as follows: 
Let $\{ {\bf z}_1^*, {\bf z}_2^*\}$ be the basis for $(\g_2/\k_2)^*$ that is dual to the basis $\{{\bf z}_1, {\bf z}_2\}$ for 
$\g_2/\k_2$. Then 
$$
f_{-2} = {\bf z}_2^* \otimes \phi_{-\nu_1+\nu_2-2} \otimes {\sf w}_{\nu^{\sf v}}^+, \ \ {\rm and}\ \ 
f_2 = {\bf z}_1^* \otimes \phi_{\nu_1-\nu_2+2} \otimes {\sf w}_{\nu^{\sf v}}^-.
$$
To summarize we have: 
$$
H^1(\g_2,K_2^0; \Xi \otimes E_{\nu}^{\sf v}) = 
{\rm Hom}_{K_2^0}(\wedge^1 \g_2/\k_2, \Xi \otimes E_{\nu}^{\sf v}) = ( (\g_2/\k_2)^* \otimes \Xi \otimes E_{\nu}^{\sf v})^{K_2^0}  = 
\C f_{-2} \oplus  \C f_2
$$
with explicit expressions for $f_{\pm 2}$ as relative Lie algebra cocycles.

To identify the class $[\Xi]^{\pm}$, a generator for the one-dimensional space $H^1(\g_2,K_2^0; \Xi \otimes E_{\nu}^{\sf v})(\pm)$, 
we need to know the action of the element $\delta = \left(\begin{smallmatrix}-1 & 0 \\ 0 & 1 \end{smallmatrix}\right)$ which 
represents the nontrivial element in $K_2/K_2^0$. Recall that the action of $\delta$ on any $f \in {\rm Hom}_{K_2^0}(\wedge^1 \g_2/\k_2, \Xi \otimes E_{\nu}^{\sf v})$ is given by 
$$(\delta f)({\bf z}) = (\Xi \otimes E_{\nu}^{\sf v})(\delta) (f({\rm Ad}(\delta^{-1}){\bf z})).$$

\begin{lemma}
\label{lem:delta}
The action of $\delta$ on $f_{\pm 2}$ is given by
$$
\delta f_{-2} = i^{\nu_1 - \nu_2}f_2, \ \ {\rm and} \ \ \delta f_2 = i^{-\nu_1 + \nu_2}f_{-2}.
$$
In particular, $\delta$ acts by $\pm 1$ on the cocycle $[\Xi]^{\pm} \ := \ f_2 \pm i^{-\nu_1+\nu_2}f_{-2}$.
\end{lemma}

\begin{proof}
The proof is routine; here are some useful relations: 
\begin{eqnarray*}
{\rm Ad}(\delta^{-1})({\bf z}_1) & = & {\bf z}_2, \\ 
\Xi(\delta)(\phi_{\nu_1-\nu_2+2}) & = & i^{2\nu_1} \phi_{-(\nu_1-\nu_2+2)}, \\ 
E_{\nu}^{\sf v}(\delta)({\sf w}_{\nu^{\sf v}}^+) & = & i^{-(\nu_1+\nu_2)} {\sf w}_{\nu^{\sf v}}^-. 
\end{eqnarray*}
\end{proof}

Let $\Whit(\Xi)$ denote the Whittaker model of $\Xi$ with respect to a nontrivial additive character $\psi_{\R}$ of $\R$; which we recall 
from \ref{sec:additive-character}, is taken to be $x \mapsto \psi_{\R}(x) = e^{-2\pi i x}$. For any $\phi \in \Xi$, let $\lambda = w(\phi)$ denote the corresponding Whittaker vector. The cohomology class $[\Xi]^{\pm}$ which generates $H^1(\g_2,K_2^0; \Whit(\Xi) \otimes E_{\nu}^{\sf v})(\pm)$ is explicitly given by
$$
[\Xi]^{\pm} \ = \ {\bf z}_1^* \otimes \lambda_{\nu_1-\nu_2+2} \otimes {\sf w}_{\nu^{\sf v}}^-  \, \pm \, 
i^{-\nu_1+\nu_2} {\bf z}_2^* \otimes \lambda_{-(\nu_1-\nu_2+2)} \otimes {\sf w}_{\nu^{\sf v}}^+
$$
Using (\ref{eqn:gl2-k-type-to-standard}) we can also express this class as: 
\begin{equation}
\label{eqn:gl2-class}
[\Xi]^{\pm} =  \sum_{l=1}^2 \sum_{\alpha=0}^{\nu_1-\nu_2}  {\bf z}_l^* \otimes \lambda^{\pm}_{l,\alpha} \otimes {\sf s}_{\alpha} 
\end{equation}
where 
\begin{equation}
\lambda_{1,\alpha}^{\pm} = \binom{\nu_1-\nu_2}{\alpha} i^{\alpha} \lambda_{\nu_1-\nu_2+2}, \ \ \ 
\lambda_{2,\alpha}^{\pm} =  \pm \binom{\nu_1-\nu_2}{\alpha} i^{-\alpha} \lambda_{-(\nu_1-\nu_2+2)}
\end{equation}

Let us now go back to $\Pi \in {\rm Coh}(G,\mu^{\sf v})$ and write down $[\Pi_{\infty}]^{+\!+}$ explicitly, where $+\!+$ is short for $(+,\dots,+)$. 
Since  
$$
[\Pi_{\infty}]^{+\!+} = \bigotimes_{j=1}^n [\Pi_j]^+
$$
we will tensor over $j$ the class $[\Pi_j]^+$. 

Let $\{{\sf s}_{j,0}, {\sf s}_{j,1},\dots, {\sf s}_{j, a_j-b_j}\}$ denote the standard basis for the representation $E_{\mu_j}^{\sf v}$. 
Let $\alpha = (\alpha_1,\dots,\alpha_n)$ be an $n$-tuple of integers such that $0 \leq \alpha_j \leq a_j - b_j$. Let 
$$
{\sf s}_{\alpha} = \otimes_{j=1}^n {\sf s}_{j , \alpha_j}. 
$$
Then the set $\{ {\sf s}_{\alpha}\}_{\alpha}$, as $\alpha$ runs through all $n$-tuples as above gives a basis for $E_{\mu}^{\sf v}$. 
Next, let $l = (l_1,\dots,l_n)$ be an $n$-tuple of integers such that $l_j \in \{1,2\}$. For each such $l$, put
$$
{\bf z}_l^* = \otimes_{j=1}^n {\bf z}^*_{j,l_j}, 
$$
where for each $1 \leq j \leq n$ we let 
${\bf z}_{j,1} = {\bf z}_1$ and ${\bf z}_{j,2} = {\bf z}_2$ 
as elements of $\gl_2 = {\rm Lie}(\GL_2(F_{\eta_j}))$; and as before ${\bf z}^*$ is the corresponding element in the dual basis. 
For each $1 \leq j \leq n$, and $\alpha_j$ as above, let 
$$
\lambda_{j, 1,\alpha_j} = \binom{\nu_1-\nu_2}{\alpha_j} i^{\alpha_j} \lambda_{\nu_1-\nu_2+2}, \ \ \ 
\lambda_{j, 2,\alpha_j} = \binom{\nu_1-\nu_2}{\alpha_j} i^{-\alpha_j} \lambda_{-(\nu_1-\nu_2+2)}
$$
and for any $l$ and $\alpha$ put 
$$
W_{l,\alpha,\infty} = \otimes_{j=1}^n \lambda_{j, l_j,\alpha_j} \in \Whit(\Pi_{\infty}, \psi_{\infty}).  
$$
We have the following expression 
\begin{equation}
\label{eqn:lambda-j}
[\Pi_{\infty}]^{+\!+} = \bigotimes_{j=1}^n [\Pi_j]^+ = \bigotimes_{j=1}^n \left(\sum_{l_j =1}^2 \sum_{\alpha_j = 0}^{a_j-b_j} 
{\bf z}_{j,l_j}^* \otimes \lambda_{j,l_j,\alpha_j} \otimes {\sf s}_{j,\alpha_j} \right).  
\end{equation}
Interchanging the tensor and the summations and regrouping we get: 
\begin{equation}
\label{eqn:infinity-class}
\boxed{
[\Pi_{\infty}]^{+\!+} = \sum_{l = (l_1,\dots,l_n)} \sum_{\alpha = (\alpha_1,\dots,\alpha_n)} {\bf z}_l^* \otimes W_{l, \alpha, \infty} \otimes {\sf s}_{\alpha}
}\end{equation}
which is our chosen generator of the one-dimensional space
$H^n(\g_{\infty}, K_{\infty}^0; \Pi_{\infty} \otimes E_{\mu}^{\sf v})(+,\dots,+)$ and is expressed as a $K_{\infty}$-fixed element of 
$$
(\g_{\infty}/\k_{\infty})^* \otimes \Whit(\Pi_{\infty},\psi_{\infty}) \otimes E_{\mu^{\sf v}}.
$$

\subsection{Periods and period relations}
\label{sec:periods}

\subsubsection{Action of ${\rm Aut}(\C)$ on global representations}

The following theorem  is due to Harder \cite{harder-general} and Waldspurger \cite{waldspurger} for $\GL_2$ over any number field (although we state it only for our totally real base field $F$). It was generalized to $\GL_n$ over any number field by Clozel \cite{clozel}. We have adapted the statement from Clozel's and Waldspurger's articles. In a classical context of Hilbert modular forms it is due to Shimura \cite{shimura-duke}; see also 
Garrett \cite[Theorem 6.1]{garrett}.

\begin{thm}
\label{thm:q-pi}
Let $\Pi$ be a regular algebraic cuspidal automorphic representation. 
For any $\sigma \in {\rm Aut}(\C)$, define an abstract irreducible representation ${}^{\sigma}\Pi = \otimes_v {}^{\sigma}\Pi_v$ of $\GL_2(\A_F)$ as follows: 
\begin{itemize}
\item For any finite place $v$, suppose the representation space of $\Pi_v$ is $V_v$, then pick any $\sigma$-linear isomorphism $A_v : V_v \to V_v'$, and 
define ${}^{\sigma}\Pi_v$ as the representation of $\GL_2(F_v)$ acting on $V_v'$ by ${}^{\sigma}\Pi_v(g) = A_v \circ \Pi_v(g) \circ A_v^{-1}$. The definition of ${}^{\sigma}\Pi_v$ is, up to equivalence, independent of all the choices made. 
\item For $v \in S_{\infty}$, define ${}^{\sigma}\Pi_v := \Pi_{\sigma^{-1} v}$, i.e., 
$({}^{\sigma}\Pi)_{\infty} = \otimes_{\eta} \Pi_{ \sigma^{-1} \circ \eta }$ where $\eta$ runs through the set ${\rm Hom}(F, \C)$ of all infinite places of the totally real field $F$.
\end{itemize}
Then ${}^{\sigma}\Pi$ is also a regular algebraic cuspidal automorphic representation. The rationality field $\Q(\Pi_f)$, which 
is defined as the subfield of $\C$ fixed by $\{\sigma\ : \ {}^{\sigma}(\Pi_f) \simeq \Pi_f \}$, is a number field. For any field 
$E$ containing $\Q(\Pi_f)$, the representation $\Pi_f$ of $\GL_2(\A_{F,f})$ has an $E$-structure that is unique up to homotheties. 
\end{thm}

(Note that the above action is a left-action, i.e., ${}^{\sigma \tau}\Pi = {}^{\sigma} ({}^{\tau}\Pi )$.) 
Suppose $\sigma$ fixes $\Pi_f$ for a representation $\Pi$ as in the theorem above, then by the strong multiplicity one theorem, $\sigma$ fixes $\Pi$; which justifies a change in notation: $\Q(\Pi)$ instead of $\Q(\Pi_f)$.

\subsubsection{Periods}
We now look closely at the assertion that $\Pi_f$ has an $E$-structure. On the one hand, a cuspidal automorphic representation $\Pi$ of 
$\GL_n(\A_F)$ admits a Whittaker model, and these models carry a natural rational structure. On the other hand, if $\Pi$ is regular and algebraic, then 
it contributes to cuspidal cohomology and from this arises a rational structure on a cohomological realization of $\Pi$. One defines periods by playing off these rational structures against each other. (Another word for these `periods' might be `regulators', as the definition our periods is very close in spirit to Borel's regulators \cite{borel-pisa}.)  The rest of \ref{sec:periods} is a very brief summary of Raghuram-Shahidi \cite{raghuram-shahidi-imrn}.

As a matter of definition/notation, given a ${\mathbb C}$-vector space $V$, and given a subfield $E \subset {\mathbb C}$, by an $E$-structure on $V$ we mean an $E$-subspace $V_E$ such that the canonical map $V_E \otimes_E \C \to V$ is an isomorphism. Further, if $V$ is a representation space for the action of a group $G$, then we will need $V_E$ to be $G$-stable. Fixing an $E$-structure gives an action of ${\rm Aut}({\mathbb C}/E)$ on $V$, 
by making it act on the second factor in $V = V_E \otimes_E {\mathbb C}$. Having fixed an 
$E$-structure, for any extension $E'/E$, we have a canonical $E'$-structure
by letting $V_{E'} = V_E \otimes_E E'$.

\subsubsection{Rational structures on Whittaker models}
Recall from \ref{sec:additive-character} that we have fixed a nontrivial character $\psi = \psi_{\infty}\otimes \psi_f$ of $F\backslash {\mathbb A}_F$. 
Let $\Whit(\Pi, \psi)$ be the Whittaker model of 
$\Pi$, and this factors as 
$\Whit(\Pi,\psi) = \Whit(\Pi_{\infty},\psi_{\infty}) \otimes \Whit(\Pi_f,\psi_f)$.
There is a semi-linear action of ${\rm Aut}({\mathbb C})$ on $\Whit(\Pi_f,\psi_f)$ 
which is defined as follows. (See Harder \cite[pp.79-80]{harder-general}.) Consider:
{\small
$$
\begin{array}{clclclclc}
{\rm Aut}({\mathbb C}/{\mathbb Q}) & \to & 
{\rm Gal}(\overline{\mathbb Q}/{\mathbb Q}) & \to&
{\rm Gal}({\mathbb Q}(\mu_{\infty})/{\mathbb Q}) & \to &
\widehat{{\mathbb Z}}^{\times} \simeq \prod_p {\mathbb Z}_p^{\times} & \subset & 
\prod_p \prod_{\mathfrak{p}|p} \mathcal{O}_{\mathfrak{p}}^{\times} \\
\sigma & \mapsto & \sigma |_{\overline{\mathbb Q}} & \mapsto & 
\sigma |_{{\mathbb Q}(\mu_{\infty})} & \mapsto & t_{\sigma} & \mapsto & t_{\sigma} = (t_{\sigma, \p})_{\p} 
\end{array}
$$}

\noindent
where the last inclusion is the one induced by the diagonal embedding of 
${\mathbb Z}_p^{\times}$ into $\prod_{\mathfrak{p}|p} 
\mathcal{O}_{\mathfrak{p}}^{\times}$. The element $t_{\sigma}$ at the end can
be thought of as an element of ${\mathbb A}_{F,f}^{\times}$. 
Let $[t_{\sigma}^{-1}]$ denote the diagonal matrix 
${\rm diag}(t_{\sigma}^{-1},1)$ regarded as an 
element of  ${\rm GL}_2({\mathbb A}_{F,f})$. 
For $\sigma \in {\rm Aut}({\mathbb C})$ and 
$W \in \Whit(\Pi_f,\psi_f)$, define the function ${}^{\sigma}W$ by
$$
{}^{\sigma}W(g_f) = \sigma(W([t_{\sigma}^{-1}]g_f))
$$
for all $g_f \in {\rm GL}_2({\mathbb A}_{F,f})$. Note that this action makes sense locally, by replacing
$t_{\sigma}$ by $t_{\sigma, \mathfrak{p}}$. Further, if $\Pi_{\p}$ is unramified, then a spherical vector is 
mapped to a spherical vector under $\sigma$. If we normalize the spherical vector to take the value $1$ on 
the identity, then $\sigma$ fixes this vector. This makes the local and global actions of $\sigma$ compatible.

\begin{lemma}
\label{lem:rational-whittaker}
With notation as above, $W \mapsto {}^{\sigma}W$ is a $\sigma$-linear ${\rm GL}_2({\mathbb A}_{F,f})$-equivariant
isomorphism from $\Whit(\Pi_f,\psi_f)$ onto $W({}^{\sigma}\Pi_f, \psi_f)$. For any finite extension $E/{\mathbb Q}(\Pi_f)$ we have
an $E$-structure on $\Whit(\Pi_f,\psi_f)$ by taking invariants: 
$$
\Whit(\Pi_f,\psi_f)_E = \Whit(\Pi_f,\psi_f)^{{\rm Aut}({\mathbb C}/E)}.
$$
\end{lemma}

\begin{proof}
See Raghuram-Shahidi \cite[Lemma 3.2]{raghuram-shahidi-imrn}; it amounts to saying that a normalized new-vector generates the $E$-structure obtained by taking invariants under ${\rm Aut}({\mathbb C}/E)$. (It helps to keep Waldspruger's \cite[Lemme I.1]{waldspurger} in mind.) 
Later, we will work with some carefully normalized new-vectors; see \ref{sec:normalized-newvector} below. 
\end{proof}

As a notational convenience, when we talk of Whittaker
models, we will henceforth suppress the additive character $\psi$, since
that has been fixed once and for all; for example, $\Whit(\Pi_f)$ will denote 
$\Whit(\Pi_f,\psi_f)$. Next, $\Whit(\Pi_f)_0$ will denote the $\Q(\Pi)$-rational structure on $\Whit(\Pi_f)$.

\subsubsection{Rational structures on cohomological representations}
Let $\mu \in X^+_0(T)$ and $\Pi \in {\rm Coh}(G,\mu^{\sf v})$. For any character $\epsilon$ of $\pi_0(G_{\infty})$, 
the cohomology space $H^n(\mathfrak{g}_{\infty},K_{\infty}^0; V_{\Pi}\otimes E_{\mu}^{\sf v})(\epsilon)$, which as 
a representation of the group $\pi_0(G_{\infty}) \times G(\A_{F,f})$ is isomorphic to $\epsilon \otimes \Pi_f$, has a natural $\Q(\Pi)$-structure which may be seen as follows. Consider the following diagram:
$$
\begin{array}{clcll}
H^n(\mathfrak{g}_{\infty},K_{\infty}^0; V_{\Pi}\otimes E_{\mu}^{\sf v})(\epsilon) 
& \simeq & \epsilon \otimes \Pi_f
& & \\
\downarrow
& & \downarrow & & \\
H^n(\mathfrak{g}_{\infty} ,K_{\infty}^0; 
\mathcal{A}_{\rm cusp}(G(F) \backslash G({\mathbb A})) 
\otimes E_{\mu}^{\sf v})
& \simeq & H^n_{\rm cusp}(S^G, \E_{\mu}^{\sf v})
& & \\
\downarrow & & \downarrow & & \\
H^n(\mathfrak{g}_{\infty},K_{\infty}^0; C^{\infty}(G(F) \backslash G({\mathbb A_F})) 
\otimes E_{\mu}^{\sf v})
 & \simeq &
H^n_{dR}(S^G, \E_{\mu}^{\sf v}) 
 & \simeq &
H^n_{B}(S^G, \E_{\mu}^{\sf v}) 
\end{array}
$$
where all the vertical arrows are injections induced by inclusions. Indeed, the rational structures on all
the above spaces come from a rational structure on the Betti cohomology space on which
it is easy to describe an action of ${\rm Aut}({\mathbb C})$--see Clozel \cite{clozel}. The point is that 
cuspidal cohomology admits a $\Q(\mu)$-structure which it inherits from `the' canonical $\Q(\mu)$-structure on Betti cohomology 
$H^n_{B}(S^G, \E_{\mu}^{\sf v})$. (By $\Q(\mu)$ we mean the subfield of $\C$ fixed by $\{\sigma : {}^{\sigma}\mu = \mu\}$, where the 
action of $\sigma$ on $\mu$, or any quantity indexed by the infinite places $S_{\infty}$, is via permuting these places, 
exactly as the action of ${\rm Aut}({\mathbb C})$ on $\Pi_{\infty}$ described in Theorem~\ref{thm:q-pi}.) 
Since $\epsilon \otimes \Pi_f \simeq H^n(\mathfrak{g}_{\infty},K_{\infty}^0; V_{\Pi}\otimes E_{\mu}^{\sf v})(\epsilon)$ is 
a Hecke eigenspace (i.e., is an irreducible subspace for the action of $\pi_0(G_{\infty}) \times G(\A_{F,f})$) of cuspidal cohomology,  
it follows that this eigenspace admits a $\Q(\Pi)$-rational structure.

\subsubsection{Comparing Whittaker models and cohomological representations}
\label{sec:comparison}
We have the following comparison isomorphism $\mathcal{F}^{\epsilon}_{\Pi}$, which is the composition of three isomorphisms: 
\begin{eqnarray*}
\Whit(\Pi_f) & \longrightarrow & 
\Whit(\Pi_f) \otimes 
H^n(\mathfrak{g}_{\infty},K_{\infty}^0; \Whit(\Pi_{\infty}) \otimes E_{\mu}^{\sf v})(\epsilon) \\
& \longrightarrow & 
H^n(\mathfrak{g}_{\infty},K_{\infty}^0; \Whit(\Pi) \otimes E_{\mu}^{\sf v})(\epsilon) \\
& \longrightarrow & 
H^n(\mathfrak{g}_{\infty},K_{\infty}^0; V_{\Pi} \otimes E_{\mu}^{\sf v})(\epsilon),
\end{eqnarray*}
where the first map is $W_f \mapsto W_f \otimes [\Pi_{\infty}]^{\epsilon}$ for all $W_f \in \Whit(\Pi_f)$ with 
$[\Pi_{\infty}]^{\epsilon}$ being the generator (as in \ref{sec:pinning-down}) of the one-dimensional space 
$H^n(\mathfrak{g}_{\infty},K_{\infty}^0; \Whit(\Pi_{\infty}) \otimes E_{\mu}^{\sf v})(\epsilon)$; 
the second map is the obvious one; and the third map is the map induced in cohomology 
by the inverse of the map which gives the Fourier coefficient of a cusp form in $V_{\Pi}$--the space
of functions in $\mathcal{A}_{{\rm cusp}}(G(F) \backslash G({\mathbb A}))$ which 
realizes $\Pi$.

\subsubsection{Definition of the periods}
The isomorphism $\mathcal{F}^{\epsilon}_{\Pi}$ need not preserve rational structures on either side. Each side is an irreducible representation space for the action of $\pi_0(G_{\infty}) \times G(\A_{F,f})$ and rational structures being unique up to homotheties 
(by Waldspurger \cite[Lemme I.1]{waldspurger}), we see that we can adjust the isomorphism $\mathcal{F}^{\epsilon}_{\Pi}$ by a scalar--which is the period--such that the adjusted map preserves rational structures. Let us state this more precisely: 

Let $\Pi = \Pi_f \otimes \Pi_{\infty}$ be a regular algebraic cuspidal automorphic representation of 
${\rm GL}_2({\mathbb A}_F)$. Let $\mu \in X^+_0(T)$ be such 
that $\Pi \in {\rm Coh}(G, \mu^{\vee})$. Let $\epsilon$ be a character of 
$K_{\infty}/K_{\infty}^0$. Let $[\Pi_{\infty}]^{\epsilon}$ be a generator of the one dimensional vector space 
$H^n(\mathfrak{g}_{\infty},K_{\infty}^0, \Pi_{\infty} \otimes E_{\mu}^{\sf v})(\epsilon)$. 
To such a datum $(\Pi_f, \epsilon, [\Pi_{\infty}]^{\epsilon})$, there is a nonzero complex number
$p^{\epsilon}(\Pi)$, such that the normalized map
$$
\mathcal{F}^{\epsilon}_{\Pi,0}:= 
p^{\epsilon}(\Pi)^{-1}
\mathcal{F}^{\epsilon}_{\Pi}
$$
is ${\rm Aut}({\mathbb C})$-equivariant, i.e., the following diagram commutes:
$$
\xymatrix{
\Whit(\Pi_f) \ar[rrr]^-{\mathcal{F}^{\epsilon}_{\Pi,0}} \ar[d]_{\sigma} & & &
H^n(\mathfrak{g}_{\infty},K_{\infty}^0; V_{\Pi}\otimes E_{\mu}^{\sf v})(\epsilon)
\ar[d]^{\sigma} \\
W({}^{\sigma}\Pi_f) 
\ar[rrr]^-{\mathcal{F}^{{}^{\sigma}\!\epsilon}_{{}^{\sigma}\Pi,0}} 
& & &
H^n(\mathfrak{g}_{\infty},K_{\infty}^0; V_{{}^{\sigma}\Pi}\otimes E_{{}^{\sigma}\!\mu}^{\sf v})
({}^{\sigma}\!\epsilon)
}
$$
The complex number $p^{\epsilon}(\Pi)$, called a period, is well-defined only up to multiplication by elements of ${\mathbb Q}(\Pi)^*$. 
If we change $p^{\epsilon}(\Pi)$
to $\alpha p^{\epsilon}(\Pi)$ with a $\alpha \in {\mathbb Q}(\Pi)^*$ then the period 
$p^{{}^{\sigma}\!\epsilon}({}^{\sigma}\Pi)$ changes to 
${\sigma}(\alpha)p^{{}^{\sigma}\!\epsilon}({}^{\sigma}\Pi)$. 

In terms of the un-normalized maps, we can describe the above commutative diagram by 
\begin{equation}
\label{eqn:unnormalized}
\sigma \circ \mathcal{F}_{\Pi}^{\epsilon} = 
\left(\frac{\sigma(p^{\epsilon}(\Pi))}
{p^{{}^{\sigma}\!\epsilon}({}^{\sigma}\Pi)} \right) 
\mathcal{F}_{{}^{\sigma}\!\Pi}^{{}^{\sigma}\!\epsilon} \circ \sigma.
\end{equation}

\subsubsection{Period relations} 
\label{sec:period-relations}
The following is the main result proved in Raghuram-Shahidi \cite{raghuram-shahidi-imrn}, but stated below for our context of $\GL_2$ over a totally real $F$. 
Let $\mu \in X^+_0(T)$ be such 
that $\Pi \in {\rm Coh}(G, \mu^{\vee})$.
Let $\epsilon$ be a character of 
$K_{\infty}/K_{\infty}^0$.  Let $\xi$ be an algebraic Hecke character of $F$ with signature $\epsilon_{\xi}$ which is defined as follows: 
any such $\xi$ is of the form $\xi = |\ |^m \otimes \xi^0$ for an integer $m$, and a finite order character $\xi^0$, then 
$$
\epsilon_{\xi} = (-1)^m(\xi^0_{\eta_1}(-1), \dots, \xi^0_{\eta_n}(-1)).
$$
For any $\sigma \in {\rm Aut}({\mathbb C})$ we have 
$$
\sigma\left(\frac
{p^{\epsilon \cdot \epsilon_{\xi}}(\Pi \otimes \xi)}
{\mathcal{G}(\xi) \, p^{\epsilon}(\Pi) }\right) 
=
\left(\frac
{p^{{}^{\sigma}\!\epsilon  \cdot \epsilon_{{}^{\sigma}\!\xi}}( {}^{\sigma}\Pi  \otimes {}^{\sigma}\xi )}
{\mathcal{G}({}^{\sigma}\xi) \,p^{{}^{\sigma}\!\epsilon}({}^{\sigma}\Pi)}\right).
$$
The action of ${\rm Aut}({\mathbb C})$ on $\epsilon$  is via permuting the infinite places.
Define $\Q(\xi)$ as the field obtained by adjoining the values of $\xi^0$, and let 
${\mathbb Q}(\Pi, \xi)$ be the compositum of the number fields
${\mathbb Q}(\Pi)$ and ${\mathbb Q}(\xi)$. We have
$$
p^{\epsilon \cdot \epsilon_{\xi}}(\Pi \otimes \xi)
 \, \sim_{{\mathbb Q}(\Pi, \xi)} \,
\mathcal{G}(\xi) \, p^{\epsilon}(\Pi).
$$
By $\sim_{{\mathbb Q}(\Pi, \xi)}$ we mean up to an element of ${\mathbb Q}(\Pi, \xi)$.

\subsection{Proof of Theorem~\ref{thm:central}}

\subsubsection{Normalized new vectors}
\label{sec:normalized-newvector}
We now show that local new-vectors when normalized appropriately give a very explicit element in the rational structure $\Whit(\Pi_f)_0$ of the 
(finite part of the) global Whittaker model. 

Recall from \ref{sec:additive-character} our choice of additive character $\psi$. Pick an element ${\sf d}_F \in \ringO_F$ such that 
${\rm ord }_{\p}({\sf d}_F) = r_{\p} = {\rm ord}(\mathfrak{D}_F)$; this is possible by strong approximation. 
Now define a character $\psi'$ by $\psi'(x) = \psi({\sf d}_F^{-1}x)$; then it is trivially checked that $\psi'_{\p}$ has conductor $\ringO_{\p}$ for all 
prime ideals $\p$. We have a map $\Whit(\Pi_f, \psi'_f) \to \Whit(\Pi_f, \psi_f)$, given by $W_f' \mapsto W_f$ where 
$$
W_f(g) = W_f'\left(\left(\begin{array}{cc} {\sf d}_F & 0 \\  0 & 1 \end{array}\right)g\right).
$$
This also makes sense locally: $W_{\p}(g_{\p}) = W_{\p}'\left(\left(\begin{smallmatrix} {\sf d}_{\p} & 0 \\  0 & 1 \end{smallmatrix}\right)g_{\p}\right)$, where,  
by ${\sf d}_{\p}$, we mean ${\sf d}_F$ as an element of $F_{\p}$.

A Whittaker vector $W'_{\p}$ is completely determined by the function on $F_{\p}^*$ 
$$
x_{\p} \mapsto \phi_{\p}'(x_{\p}) := W'_{\p}\left(\left(\begin{array}{cc} x_{\p} & 0 \\  0 & 1 \end{array}\right)\right), 
$$
i.e., the map $W_{\p}' \mapsto \phi_{\p}'$ is injective. (See, for example, Godement \cite[Lemma 3 on p.1.5]{godement}.) 
The set of all such functions $\kappa_{\p}'$ is the Kirillov model $\mathcal{K}(\Pi_{\p},\psi'_{\p})$. 
Working in the Kirillov model $\mathcal{K}(\Pi_{\p}, \psi'_{\p})$, we have the following explicit formulae for new-vectors taken from 
Schmidt \cite[p.141]{schmidt}. For each representation $\Pi_{\p}$ we have a very special vector $\kappa_{\p}^{\rm new} \in \mathcal{K}(\Pi_{\p},\psi'_{\p})$
that is the local new-vector in that model. Since the table consists of purely local information, we will abuse our notation by dropping the subscript $\p$. 

\begin{enumerate}
\item Principal series representation $\pi(\chi_1, \chi_2)$, with $\chi_1,\chi_2$ unramified, and $\chi_1\chi_2^{-1} \neq |\ |^{\pm 1}$. Then 
$$
\kappa_{\p}^{\rm new}(x) = |x|^{1/2}\left(\sum_{k+l = v(x)} \chi_1(\varpi^k)\chi_2(\varpi^l)\right)
{\bf 1}_{\ringO}(x). 
$$

\item Principal series representation $\pi(\chi_1, \chi_2)$ with exactly one of the characters being unramified; 
say $\chi_1$ unramified and $\chi_2$ ramified. Then 
$$
\kappa_{\p}^{\rm new}(x) =
|x|^{1/2} \chi_1(x) {\bf 1}_{\ringO}(x).
$$

\item Unramified twist of the Steinberg representation: ${\rm St} \otimes \chi$ with $\chi$ unramified. Then 
$$
\kappa_{\p}^{\rm new}(x) = |x| \chi(x) {\bf 1}_{\ringO}(x).
$$

\item In all other cases (principal series $\pi(\chi_1, \chi_2)$ with both $\chi_1, \chi_2$ ramified; ramified twist of the Steinberg representation; any supercuspidal representation) we have 
$$
\kappa_{\p}^{\rm new}(x) =
{\bf 1}_{\ringO^{\times}}(x).
$$
\end{enumerate}

Let $W_{\p}^{\rm new} \in \Whit(\Pi_{\p}, \psi'_{\p})$ correspond to $\kappa_{\p}^{\rm new}$, and finally we let $W_{\p}^{\circ} \in \Whit(\Pi_{\p}, \psi_{\p})$ correspond to 
$W_{\p}^{\rm new}$. That is we have: 
\begin{equation}
\label{eqn:newvector}
W_{\p}^{\circ} \ \leftrightarrow \ W_{\p}^{\rm new}  \ \leftrightarrow \ \kappa_{\p}^{\rm new} .
\end{equation}
We will also denote $W_{\p}^{\circ}$ by $W_{\Pi_{\p}}^{\circ}$, and observe that 
$$
W_{\Pi_{\p}}^{\circ}\left(\left(\begin{array}{cc} x_{\p} & 0 \\  0 & 1 \end{array}\right)\right) = 
W_{\Pi_{\p}}^{\rm new}\left(\left(\begin{array}{cc} {\sf d}_{\p}x_{\p} & 0 \\  0 & 1 \end{array}\right)\right) = 
\kappa_{\p}^{\rm new}({\sf d}_{\p}x_{\p}). 
$$

\begin{prop}
\label{prop:new-rational}
Let $\Pi$ be a cuspidal automorphic representation of $\GL_2(\A_F)$. For each prime ideal $\p$, let $W_{\Pi_{\p}}^{\circ}$ be the normalized new-vector 
as defined in (\ref{eqn:newvector}) of the representation $\Pi_{\p}$ which is realized in its Whittaker model $\Whit(\Pi_{\p}, \psi_{\p}).$
For any $\sigma \in  {\rm Aut}(\C)$ we have 
$$
{}^{\sigma}W_{\Pi_{\p}}^{\circ} = W_{{}^{\sigma}\!\Pi_{\p}}^{\circ}.
$$
Let $W^{\circ} = \otimes_{\p} W_{\Pi_{\p}}^{\circ} \in \Whit(\Pi_f,\psi_f).$
Then $W^\circ$ is fixed by $ {\rm Aut}(\C/\Q(\Pi))$, and hence $W^\circ  \in \Whit(\Pi_f,\psi_f)_0$.
\end{prop}

\begin{proof}
To see ${}^{\sigma}W_{\Pi_{\p}}^{\circ} = W_{{}^{\sigma}\!\Pi_{\p}}^{\circ}$, it suffices to check that both Whittaker vectors give the same vector in the Kirillov model; which is then verified using a case-by-case analysis using the above table. 

Suppose $\Pi_{\p}$ is an unramified principal series representation, and say, 
$\Pi_{\p} = \pi(\chi_{1,\p}, \chi_{2,\p})$ for characters $\chi_{j,\p} : F_{\p}^* \to \C^*$. Let us describe ${}^{\sigma}\Pi_{\p}$. For this, given any 
character $\chi$ of $F^*$, and any $\sigma\in  {\rm Aut}(\C)$, define ${}^{\sigma}\chi$ as $\sigma \circ \chi$, i.e., 
${}^{\sigma}\chi(x) = \sigma(\chi(x))$. Define a twisted action of $\sigma\in  {\rm Aut}(\C)$ on characters by: 
$$
{}^{\sigma'}\!\chi(x) = |x|^{-1/2}\sigma(\chi(x)|x|^{1/2}).
$$
As is checked in Waldspurger \cite[I.2]{waldspurger}, we have 
$$
{}^{\sigma}\pi(\chi_{1,\p}, \chi_{2,\p}) = \pi({}^{\sigma'}\!\chi_{1,\p}, {}^{\sigma'}\!\chi_{2,\p}).
$$
On the one hand, using the formula for $\kappa_{\p}^{\rm new}$ for $\Pi_{\p}  = \pi(\chi_{1,\p}, \chi_{2,\p})$ we have: 
\begin{eqnarray*}
{}^{\sigma}W_{\Pi_{\p}}^{\circ}\left(\left(\begin{array}{cc} x_{\p} & 0 \\  0 & 1 \end{array}\right) \right) 
& = & \sigma\left(W_{\Pi_{\p}}^{\circ}\left(\left(\begin{array}{cc} t_{\sigma,\p}^{-1}x_{\p} & 0 \\  0 & 1 \end{array}\right)\right)\right)  \\
& = & \sigma\left(W_{\Pi_{\p}}^{\rm new}\left(\left(\begin{array}{cc} {\sf d}_{\p}t_{\sigma,\p}^{-1}x_{\p} & 0 \\  0 & 1 \end{array}\right)\right)\right)  \\
& = & \sigma\left(|{\sf d}_{\p}t_{\sigma,\p}^{-1}x_{\p}|^{1/2}\left(\sum_{k+l = r_{\p}+v(x)} \chi_{1,\p}(\varpi^k)\chi_{2,\p}(\varpi^l)\right)
{\bf 1}_{\ringO_{\p}}({\sf d}_{\p}t_{\sigma,\p}^{-1}x_{\p})\right) \\
& = & \sigma(|{\sf d}_{\p}x_{\p}|^{1/2}) 
\left(\sum_{k+l = r_{\p}+v(x)} \sigma\left(\chi_{1,\p}(\varpi^k)\chi_{2,\p}(\varpi^l) \right) \right)
{\bf 1}_{\ringO_{\p}}({\sf d}_{\p}x_{\p})
\end{eqnarray*}
since $t_{\sigma,\p} \in \ringO_{\p}^{\times}$. 
On the other hand, using the same formula for $\kappa_{\p}^{\rm new}$, but now for the representation
${}^{\sigma}\Pi_{\p}  = \pi({}^{\sigma'}\!\chi_{1,\p}, {}^{\sigma'}\!\chi_{2,\p})$ we have:
\begin{eqnarray*}
W_{{}^{\sigma}\Pi_{\p}}^{\circ}\left(\left(\begin{array}{cc} x_{\p} & 0 \\  0 & 1 \end{array}\right) \right) 
& = & W_{{}^{\sigma}\Pi_{\p}}^{\rm new}\left(\left(\begin{array}{cc} {\sf d}_{\p}x_{\p} & 0 \\  0 & 1 \end{array}\right) \right) \\
& = & |{\sf d}_{\p}x_{\p}|^{1/2} \left(\sum_{k+l = r_{\p}+v(x)} {}^{\sigma'}\!\chi_{1,\p}(\varpi^k) {}^{\sigma'}\!\chi_{2,\p}(\varpi^l)\right)
{\bf 1}_{\ringO_{\p}}({\sf d}_{\p}x_{\p}) \\
& = & |{\sf d}_{\p}x_{\p}|^{1/2} |\varpi_{\p}|^{-\, \frac{r+v(x)}{2}}
\left(\sum_{k+l = r_{\p}+v(x)} \sigma \left(\chi_{1,\p}(\varpi^k) \chi_{2,\p}(\varpi^l) |\varpi_{\p}|^{\frac{r+v(x)}{2}} \right) \right)
{\bf 1}_{\ringO_{\p}}({\sf d}_{\p}x_{\p}) 
\end{eqnarray*}
Using $|\varpi_{\p}|^{\frac{r+v(x)}{2}} =  |{\sf d}_{\p}x_{\p}|^{1/2}$ the final expression also simplifies  to 
$$
\sigma(|{\sf d}_{\p}x_{\p}|^{1/2}) 
\left(\sum_{k+l = r_{\p}+v(x)} \sigma\left(\chi_{1,\p}(\varpi^k)\chi_{2,\p}(\varpi^l) \right) \right)
{\bf 1}_{\ringO_{\p}}({\sf d}_{\p}x_{\p}).
$$
This concludes the proof in the case of an unramified principal series representation. In all the other cases, the above calculation is much simpler. Let us note that in the case of the Steinberg representation one has ${}^{\sigma}({\rm St} \otimes \chi) = {\rm St} \otimes {}^{\sigma}\!\chi$. We omit further details. 
\end{proof}

\subsubsection{The global integral}
\label{sec:global-integral}
Let $\Pi$ be a cuspidal automorphic representation as in Theorem~\ref{thm:central}. Piece together all the normalized Whittaker vectors
$W_{\p}^{\circ}$ in \ref{sec:normalized-newvector} and let $W^{\circ} = \otimes_{\p} W_{\p}^{\circ}.$ 
For each infinite place $\eta_j$ pick any Whittaker vector $W_j \in \Whit(\Pi_{\eta_j}, \psi_{\eta_j})$, and put 
$W_{\infty} = \otimes_{j=1}^n W_j$. Now put
$$
W = W_{\infty} \otimes W^{\circ} \in \Whit(\Pi_{\infty}) \otimes \Whit(\Pi_f) = \Whit(\Pi).
$$
Let $\phi \in V_{\Pi}$ be the cusp form that corresponds to $W$ under the isomorphism $V_{\Pi} \to \Whit(\Pi)$ of taking the $\psi$-Fourier coefficient.  
For any place $v$, and any $W_v \in \Whit(\Pi_v)$, define the zeta-integral
$$
\zeta_{v}(s, W_{v}) = \int_{x \in F_v^*} 
W_{v}\left(\left(\begin{array}{cc} x & 0 \\ 0 & 1 \end{array}\right)\right) |x|^{s-\frac12} dx, \ \ \Re(s) \gg 0. 
$$
Hecke theory for $\GL_2$ (see Gelbart \cite[Section 6]{gelbart}) says that these integrals have a meromorphic continuation to all of $\C$. 
The assumption that $s=1/2$ is critical for $L(s,\Pi)$ says that $\zeta_{\eta}(\frac12, W_{\eta})$ is finite for every infinite place $\eta$. Lastly, let 
$\zeta_{\infty}(s, W_{\infty}) = \prod_{\eta \in S_{\infty}} \zeta_{\eta}(s, W_{\eta}).$

\begin{prop}
\label{prop:global-integral}
With the notations as above, 
$$
\int_{F^{\times}\backslash \A_F^{\times}} \phi\left(\left(\begin{array}{cc} x & 0 \\ 0 & 1 \end{array}\right)\right) dx 
= 
\zeta_{\infty}(1/2, W_{\infty}) L_f(1/2, \Pi).
$$
\end{prop}

\begin{proof}
The usual unfolding argument gives
$$
\int_{F^{\times}\backslash \A_F^{\times}} \phi\left(\left(\begin{array}{cc} x & 0 \\ 0 & 1 \end{array}\right)\right) |x|^{s-\frac12} dx  = 
\int_{\A_F^{\times}} W\left(\left(\begin{array}{cc} x & 0 \\ 0 & 1 \end{array}\right)\right) |x|^{s-\frac12} dx.
$$
The integral on the left converges absolutely everywhere (since $\phi$ has rapid decay). The integral on the right converges for $\Re(s) \gg 0$, and there it
is eulerian, and so factorizes as $\prod_v \zeta_{v}(s, W_v).$
For every prime ideal $\p$, we know that the zeta-integral of the local new vector gives the local $L$-function; more precisely, we have
$$
\zeta_{\p}(s, W_{\p}^{\circ}) = |{\sf d}_{\p}|^{s-1/2} \zeta_{\p}(s, W_{\p}^{\rm new}) = 
|{\sf d}_{\p}|^{s-1/2} L_{\p}(s, \Pi_{\p}). 
$$
We deduce that for $\Re(s) \gg 0$ we have 
$$
\int_{F^{\times}\backslash \A_F^{\times}} \phi\left(\left(\begin{array}{cc} x & 0 \\ 0 & 1 \end{array}\right)\right) |x|^{s-\frac12} dx  = 
\zeta_{\infty}(s, W_{\infty}) L_f(s, \Pi) |\mathfrak{d}_F|^{s-1/2}
$$
where $\mathfrak{d}_F$ is the absolute discriminant of $F$. However, the left hand side converges for all $s$, and the right hand side has a meromorphic continuation for all $s$, and so we can evaluate at $s=1/2$ to finish the proof of the proposition. 
\end{proof}

\subsubsection{The cohomology class $\vartheta_{\Pi}$ attached to $W_{\Pi}^{\circ}$}
Consider the map 
$$
\mathcal{F}^{+\!+}_{\Pi} : \Whit(\Pi_f) \to H^n(\mathfrak{g}_{\infty},K_{\infty}^0; V_{\Pi}\otimes E_{\mu}^{\sf v})(+\!+) 
$$
as in \ref{sec:comparison}, and let $\vartheta_{\Pi}$ be the image of $W_{\Pi}^{\circ}$ under this map, i.e., 
$$
\vartheta_{\Pi} =  \mathcal{F}^{+\!+}_{\Pi}(W_{\Pi}^{\circ}).
$$
Fix an open compact subgroup $K_f$ that leaves $W_{\Pi}^{\circ}$ invariant; 
an optimal one is related to the conductor of $\Pi$, but this will not play a role here. From \ref{sec:cuspidal-cohomology}, we have 
$$
\vartheta_{\Pi}  \in H^n_{\rm cusp}(S^G_{K_f}, \E_{\mu}^{\sf v}) \subset  H^n(S^G_{K_f}, \E_{\mu}^{\sf v}).
$$
It is a fundamental fact (see Clozel \cite{clozel}) that cuspidal cohomology injects into cohomology with compact supports, i.e., 
$$
H^n_{\rm cusp}(S^G_{K_f}, \E_{\mu}^{\sf v}) 
\hookrightarrow H^n_c(S_{K_f}^G, \E_{\mu}^{\sf v}). 
$$
Therefore 
$$
\vartheta_{\Pi}  \in H^n_c(S^G_{K_f}, \E_{\mu}^{\sf v}).
$$

Recall that map $ \mathcal{F}^{+\!+}_{\Pi}$ is a composition of three isomorphisms, and the first one maps $W_{\Pi}^{\circ}$ to 
$W_{\Pi}^{\circ} \otimes [\Pi_{\infty}]^{+\!+}$, where the class $[\Pi_{\infty}]^{+\!+}$ is given in (\ref{eqn:infinity-class}). Using an analogous notation, 
we may write the class $\vartheta_{\Pi}$ in terms of Lie algebra cocycles as
\begin{equation}
\label{eqn:vartheta}
\vartheta_{\Pi} = 
\sum_{l = (l_1,\dots,l_n)} \sum_{\alpha = (\alpha_1,\dots,\alpha_n)} {\bf z}_l^* \otimes \phi_{l, \alpha} \otimes {\sf s}_{\alpha}
\end{equation}
where $\phi_{l,\alpha} \in V_{\Pi}$ are cuspforms whose corresponding Whittaker functions in $\Whit(\Pi) = \Whit(\Pi_{\infty}) \otimes \Whit(\Pi_f)$ are
$$
\phi_{l,\alpha}  \leftrightarrow W_{l,\alpha} = W_{l,\alpha, \infty} \otimes W_{\Pi}^{\circ}. 
$$ 

For later use, let us record the action of ${\rm Aut}(\C)$ on $\vartheta_{\Pi}$ which is given by the following

\begin{prop}
\label{prop:vartheta-sigma}
$$
{}^{\sigma}\vartheta_{\Pi} =  
\frac{\sigma(p^{+\!+}(\Pi))}{p^{+\!+}({}^{\sigma}\Pi)}
\vartheta_{{}^{\sigma}\!\Pi}.
$$
\end{prop}

\begin{proof}
This follows from Equation (\ref{eqn:unnormalized}) and Proposition~\ref{prop:new-rational}: 
\begin{eqnarray*}
{}^{\sigma}\vartheta_{\Pi} = \sigma (\mathcal{F}_{\Pi}^{+\!+}(W_{\Pi}^{\circ})) & = & 
\left(\frac{\sigma(p^{+\!+}(\Pi))}{p^{+\!+}({}^{\sigma}\Pi)} \right) 
\mathcal{F}_{{}^{\sigma}\!\Pi}^{+\!+} ({}^{\sigma}\! W_{\Pi}^{\circ}) \\
& = &  
\left(\frac{\sigma(p^{+\!+}(\Pi))}{p^{+\!+}({}^{\sigma}\Pi)} \right) 
\mathcal{F}_{{}^{\sigma}\!\Pi}^{+\!+}(W_{{}^{\sigma}\! \Pi}^{\circ}) = 
\frac{\sigma(p^{+\!+}(\Pi))}{p^{+\!+}({}^{\sigma}\Pi)} 
\vartheta_{{}^{\sigma}\!\Pi}.
\end{eqnarray*}
\end{proof}

\subsubsection{Pulling back to get a $\GL_1$-class $\iota^*\vartheta_{\Pi}$}
\label{sec:pullback-iota}

Let $\iota : \GL_1 \to \GL_2$ be the map $x \mapsto \left(\begin{smallmatrix}x & \\ & 1\end{smallmatrix}\right)$. 
Then $\iota$ induces a map at the level of local and global groups, and between appropriate symmetric spaces 
of $\GL_1$ and $\GL_2$, all of which will also be denoted by $\iota$ again; this should cause no confusion. 
The pullback (of a subset, a function, a differential form, or a cohomology class) via $\iota$ will be denoted $\iota^*$. A little more 
precisely, $\iota$ induces an injection: 
$$
\iota : \GL_1(F)\backslash \GL_1(\A_F)/ \iota^*K_{\infty} \iota^*K_f 
\hookrightarrow \GL_2(F)\backslash \GL_2(\A_F)/ K_{\infty}^0K_f.
$$
Note that $\iota^*\K_{\infty} = \{1\}$, and let us denote $R_f := \iota^*K_f$ which is an open compact subgroup of 
$A_{F,f}^{\times}$. The above injection will be denoted $\iota : \bar{S}^{G_1}_{R_f} \hookrightarrow S^G_{K_f}$, where 
$$
\bar{S}^{G_1}_{R_f} = F^{\times}\backslash \A_F^{\times} / R_f. 
$$
As a manifold $\bar{S}^{G_1}_{R_f}$ is an oriented $n$-dimensional manifold all of whose connected components are isomorphic to 
$\prod_{j=1}^n \R_{>0}$. (Choose the obvious orientation on each connected component.)
It is a standard fact that this inclusion $\iota : \bar{S}^{G_1}_{R_f} \hookrightarrow S^G_{K_f}$ is a proper map, and hence we can pull back 
$\vartheta_{\Pi}  \in H^n_c(S^G_{K_f}, \E_{\mu}^{\sf v})$ by $\iota$, to get 
$$
\iota^* \vartheta_{\Pi}  \in H^n_c(\bar{S}^{G_1}_{R_f}, \iota^*\E_{\mu}^{\sf v})
$$
where $\iota^*\E_{\mu}^{\sf v}$ is the sheaf on $\bar{S}^{G_1}_{R_f}$ given by the restriction of the representation $E_{\mu}^{\sf v}$ to $\GL_1$ via 
$\iota$.

\subsubsection{Criticality of $s=1/2$ and the coefficient $\mu$}
\label{sec:criticality}
We now appeal to the hypothesis that $s = 1/2$ is critical to deduce that we can work with cohomology with trivial coefficients, i.e., 
in $H^n_c(\bar{S}^{G_1}_{R_f}, \C)$. For this, let us first record all the critical points for the $L$-function at hand: 
\begin{prop}
Let $\Pi \in {\rm Coh}(G,\mu^{\sf v})$, with $\mu \in X^+_0(T)$. Suppose $\mu = (\mu_1,\dots,\mu_n)$ where $\mu_j = (a_j,b_j)$ and $a_j \geq b_j$. Then 
$$
s = \frac12 + m \in \frac12 + \Z \mbox{ is critical for $L(s,\Pi)$} \iff 
-a_j \leq m \leq -b_j, \ \forall j.
$$
\end{prop}

\begin{proof}
Recall the definition (as stated, for example, in Deligne \cite{deligne}) for a point  to be critical. If we are working with an $L$-function for $\GL_n$ then the so-called motivic normalization says that critical points are in the set $\frac{n-1}{2} + \Z$. In our situation, we would say $s = s_0  \in \frac12 + \Z$ is critical for $L(s,\Pi)$ if and only if both $L_{\infty}(s, \Pi_{\infty})$ and $L_{\infty}(1-s,\Pi_{\infty}^{\sf v})$ are regular at $s = s_0$, i.e., the $L$-factors at infinity on both sides of the functional equation do not have poles at $s=s_0$. (Automorphic $L$-functions are always normalized so that the functional equation looks like 
$L(s, \Pi) = \varepsilon(s,\Pi)L(1-s,\Pi^{\sf v})$.) 

Given $\Pi \in {\rm Coh}(G,\mu^{\sf v})$, we know from \ref{sec:infinite-components} and \ref{sec:gl2r} that 
$$
\Pi_{\infty} = \otimes_j \Pi_{\eta_j} = \otimes_j (D_{a_j-b_j+1} \otimes |\ |^{(a_j+b_j)/2}).
$$
Since $D_l$ is self-dual, we also have 
$$
\Pi_{\infty}^{\sf v} =  \otimes_j (D_{a_j-b_j+1} \otimes |\ |^{-(a_j+b_j)/2}).
$$
Using the information in \ref{sec:llc-gl2r} on the local factors for $\GL_2(\R)$, and ignoring nonzero constants and exponentials (which are irrelevant to compute critical points) we have: 
$$
L_{\infty}(s,\Pi_{\infty}) \sim \prod_j \Gamma(s + \frac12 + a_j), \ \ 
L_{\infty}(1-s,\Pi_{\infty}^{\sf v}) \sim \prod_j \Gamma(1-s + \frac12 - b_j).
$$
Hence, $L_{\infty}(s, \Pi_{\infty})$ is regular at $ \frac12 + m$ if and only if $m + a_j \geq 0$; similarly, 
$L_{\infty}(1-s,\Pi_{\infty}^{\sf v})$ is regular at $ \frac12 + m$ if and only if $-m - b_j \geq 0$. 
\end{proof}

\begin{cor}
Let $\Pi \in {\rm Coh}(G,\mu^{\sf v})$, with $\mu \in X^+_0(T)$. Suppose $\mu = (\mu_1,\dots,\mu_n)$ where $\mu_j = (a_j,b_j)$ and $a_j \geq b_j$. 
The center of symmetry $s = 1/2$ is critical if and only if ${\rm Hom}_{\GL_1(F_{\infty})}(E_{\mu}^{\sf v}, 1\!\!1) \neq 0$.
\end{cor}

\begin{proof}
Follows from the above proposition and Lemma~\ref{lem:trivial-repn-gl1}.
\end{proof}

\subsubsection{The cohomology class $\mathcal{T}^* \iota^* \vartheta_{\Pi}$ with trivial coefficients}
\label{sec:pullback-t}
When $s=1/2$ is critical, let us let 
$$
\mathcal{T} \in {\rm Hom}_{\GL_1(F_{\infty})}(E_{\mu}^{\sf v}, 1\!\!1)
$$ 
be the nonzero element as prescribed by Lemma~\ref{lem:trivial-repn-gl1}. 
Since everything factors over infinite places, we can let $\mathcal{T} = \otimes_{j=1}^n \mathcal{T}_j$, with 
$\mathcal{T}_j \in {\rm Hom}_{\GL_1(F_{\eta_j})}(E_{\mu_j}^{\sf v}, 1\!\!1)$. 
The map $\mathcal{T}$ induces a morphism of sheaves on the space $\bar{S}^{G_1}_{R_f}$, and by functoriality, a homomorphism  
$$
\mathcal{T}^* : H^n_c(\bar{S}^{G_1}_{R_f}, \iota^*\E_{\mu}^{\sf v}) \to H^n_c(\bar{S}^{G_1}_{R_f}, \C). 
$$

The image of the class $\iota^*\vartheta_{\Pi}$ under $\mathcal{T}^*$,  expressed in terms of relative Lie algebra cocycles,  is given by:
\begin{equation}
\label{eqn:vartheta-c-coefficients}
\mathcal{T}^* \iota^* \vartheta_{\Pi} = 
\sum_{l = (l_1,\dots,l_n)} \sum_{\alpha = (\alpha_1,\dots,\alpha_n)} \iota^*{\bf z}_l^* \otimes \iota^*\phi_{l, \alpha} \otimes \mathcal{T}({\sf s}_{\alpha}) = 
\sum_{l = (l_1,\dots,l_n)} \iota^*{\bf z}_l^* \otimes \iota^*\phi_{l, \underline{a}}
\end{equation}
where $\underline{a} = (a_1,\dots,a_n)$. This follows from (\ref{eqn:vartheta}) and Lemma~\ref{lem:trivial-repn-gl1}. 

Since the map  $\mathcal{T}$ is defined over $\Q$, as after all the standard basis for $E_{\mu_j}^{\sf v}$ gives it a $\Q$-structure, the morphism
$\mathcal{T}^*$ is rational, i.e, for all $\sigma \in {\rm Aut}(\C)$ we have
\begin{equation}
\label{eqn:t*-rational}
\sigma \circ \mathcal{T}^* = \mathcal{T}^* \circ \sigma
\end{equation}

Observe that $\mathcal{T}^* \iota^*\vartheta_{\Pi}$ is a top-degree compactly supported cohomology class for $\bar{S}^{\GL_1}_{R_f}$.

\subsubsection{Top-degree cohomology with compact supports}
Let us recall some basic topological facts here. Let $M$ be an oriented $n$-dimensional manifold with $h$ connected components, indexed by  
$\nu$ with $1 \leq \nu \leq h$. Then Poincar\'e duality implies that 
$$
H^n_c(M,\C) \isom \oplus_{\nu = 1}^h \C.
$$
(See, for example, Harder \cite[4.8.5]{harder-book}.) The map is integration over the entire manifold with some chosen orientation; for each connected component you get a complex number. Now let us add these numbers to get a map $\vartheta \mapsto \int_M \vartheta$
$$
\int_M : H^n_c(M,\C)  \to \C. 
$$
As explained in Raghuram \cite[3.2.3]{raghuram}, such a map given by Poincar\'e duality is rational, i.e., 
\begin{equation}
\label{eqn:poincare}
\sigma\left(\int_M \vartheta \right) = \int_M {}^{\sigma}\vartheta.
\end{equation}

\subsubsection{The main identity}
Recall that $\mathcal{T}^* \iota^*\vartheta_{\Pi} \in H^n_c(\bar{S}^{G_1}_{R_f}, \C)$ is a top-degree compactly supported cohomology class. 
We can integrate it over all of $\bar{S}^{G_1}_{R_f}$. The main technical theorem needed to analyze the arithmetic properties of the 
special value $L(1/2, \Pi)$ is 
\begin{thm}
\label{thm:main}
$$
\int_{\bar{S}^{G_1}_{R_f}} \mathcal{T}^* \iota^*\vartheta_{\Pi} \ = \ 
\frac{\langle [\Pi_{\infty}]^{+\!+} \rangle }{(4i)^n{\rm vol}(R_f)} \, L_f(1/2, \Pi) 
$$
where 
$$
\langle [\Pi_{\infty}]^{+\!+} \rangle = \sum_{l = (l_1,\dots,l_n)} \zeta_{\infty}(1/2, W_{l, \underline{a},\infty}). 
$$
\end{thm}

\begin{proof}
Recall that (\ref{eqn:vartheta-c-coefficients}) gives
$$
\mathcal{T}^* \iota^*\vartheta_{\Pi} = 
\sum_{l = (l_1,\dots,l_n)} \sum_{\alpha = (\alpha_1,\dots,\alpha_n)} \iota^*{\bf z}_l^* \otimes \iota^*\phi_{l, \alpha} \otimes \mathcal{T}({\sf s}_{\alpha}) = 
\sum_{l = (l_1,\dots,l_n)} \iota^*{\bf z}_l^* \otimes \iota^*\phi_{l, \underline{a}}.
$$
We will identify the terms $\iota^*{\bf z}_l^*$. Consider just one copy of $\GL_1(\R)$ sitting inside $\GL_2(\R)$ via $\iota$. 
Let ${\bf t}_1 := 1$ be a basis for $\g_1 = \C$. (Fixing ${\bf t}_1$ is tantamount to fixing an orientation on $\R_{>0} = \GL_1(\R)^0$. 
Taking all the infinite places together, this will be fixing an orientation on each connected component of $\bar{S}^{G_1}_{R_f}$.)  
Note that 
\begin{eqnarray*}
\iota({\bf t}_1) & = & \frac{1}{4i} \left({\bf z}_1 + {\bf z}_2 + \left(\begin{smallmatrix}2 & 0 \\ 0 & 2 \end{smallmatrix}\right)\right), \ {\rm in} \ \g_2, \\
& = & \frac{1}{4i}  \left({\bf z}_1 + {\bf z}_2\right), \ {\rm in} \ \g_2/k_2, 
\end{eqnarray*}
hence 
$\iota^*{\bf z}_1^*  = \iota^*{\bf z}_2^* = \frac{1}{4i} {\bf t}_1^*$. Applying this to each infinite place, we see 
$$
 \iota^*{\bf z}_l^* = \otimes  \iota^*{\bf z}_{j, l_j}^* = \frac{1}{(4i)^n} \otimes_{j=1}^n {\bf t}_{j,1}^*,
$$ 
where ${\bf t}_{j,1}$ is the element ${\bf t}_1$ for the infinite place $\eta_j$. 

Using the fact that $\phi_{l, \alpha}$ is fixed by $K_f$ which implies that $\iota^*\phi_{l, \underline{a}}$ is fixed by $R_f = \iota^*K_f$ we get 
\begin{eqnarray*}
\int_{\bar{S}^{G_1}_{R_f}} \mathcal{T}^* \iota^*\vartheta_{\Pi}  
& = & \sum_{l = (l_1,\dots,l_n)} \frac{1}{(4i)^n} \int_{\bar{S}^{G_1}_{R_f}} \iota^*\phi_{l, \underline{a}} \\ 
& = &  \frac{1}{(4i)^n}  \sum_{l = (l_1,\dots,l_n)} \int_{F^{\times}\backslash \A_F^{\times} / R_f} 
\phi_{l, \underline{a}}\left(\left(\begin{array}{cc}x & \\ & 1\end{array}\right)\right)\, dx \\
& = &  \frac{1}{(4i)^n{\rm vol}(R_f)}  \sum_{l = (l_1,\dots,l_n)} \int_{F^{\times}\backslash \A_F^{\times} } 
\phi_{l, \underline{a}}\left(\left(\begin{array}{cc}x & \\ & 1\end{array}\right)\right)\, dx \\
& = & \frac{1}{(4i)^n{\rm vol}(R_f)}  \sum_{l = (l_1,\dots,l_n)} \left(\zeta_{\infty}(1/2, W_{l, \underline{a}, \infty}) L_f(1/2, \Pi)\right)
\end{eqnarray*}
where the last equality is due to Proposition~\ref{prop:global-integral}.
\end{proof}

\subsubsection{Archimedean computations}

\begin{prop}
\label{prop:archimedean-mess}
Given $\Pi \in {\rm Coh}(G,\mu^{\sf v})$, with $\mu = (\mu_1,\dots,\mu_n)$ and $\mu_j = (a_j,b_j)$, we have 
$$\langle [\Pi_{\infty}]^{+\!+} \rangle  \ = \ c\, (2 \pi i )^{-d_{\infty}} i^n$$
where $d_{\infty} = \sum_{j=1}^n (a_j + 1)$, and $c$ is a nonzero integer (which is made explicit in the proof).
\end{prop}

\begin{proof}
To compute $\langle [\Pi_{\infty}]^{+\!+} \rangle =  \sum_{l = (l_1,\dots,l_n)} \zeta_{\infty}(1/2, W_{l, \underline{a},\infty})$, let us begin by noting that each summand is a product over infinite places: 
$$
\zeta_{\infty}(1/2, W_{l, \underline{a},\infty}) = \prod_{j=1}^n \zeta_{\eta_j}(1/2, \lambda_{j,l_j,a_j})
$$
where the $\lambda_{j,l_j,a_j}$ are as in (\ref{eqn:lambda-j}). We can rewrite the expression for $\langle [\Pi_{\infty}]^{+\!+} \rangle $ as 
$$
 \sum_{l = (l_1,\dots,l_n)}  \prod_{j=1}^n \zeta_{\eta_j}(1/2, \lambda_{j,l_j,a_j}) = 
\prod_{j=1}^n \left( \zeta_{\eta_j}(1/2, \lambda_{j, 1 ,a_j}) + \zeta_{\eta_j}(1/2, \lambda_{j, 2 ,a_j})\right). 
$$
The $j$-th factor in the right hand side  is the value at $s=1/2$ of the sum of two zeta-integrals:
$$
\zeta_{\eta_j}(s, \lambda_{j, 1 ,a_j}) + \zeta_{\eta_j}(s, \lambda_{j, 2 ,a_j}). 
$$
Using the definitions of $\zeta_{\eta}$ and $\lambda_{j, l_j,a_j}$ we get  
{\small
$$
 \int_{x \in \R^*} \dbinom{a_j-b_j}{a_j}
i^{a_j}\lambda_{(a_j-b_j+2)}\left(\left(\begin{array}{cc} x & 0 \\ 0 & 1 \end{array}\right)\right) |x|^{s-\frac12} dx 
\ + \ 
 \int_{x \in \R^*} \dbinom{a_j-b_j}{a_j}
i^{-a_j}\lambda_{-(a_j-b_j+2)}\left(\left(\begin{array}{cc} x & 0 \\ 0 & 1 \end{array}\right)\right) |x|^{s-\frac12} dx.
$$}
Recall that the integrals converge for $\Re{s} \gg 0$.
In the second integral, using the fact that $\delta(\lambda_{a_j-b_j+2})  =  i^{2 a_j} \lambda_{-(a_j-b_j+2)}$, and changing variable from $x$ to $-x$, 
we see that it is the same as the first integral. Hence 
$$
\zeta_{\eta_j}(s, \lambda_{j, 1 ,a_j}) + \zeta_{\eta_j}(s, \lambda_{j, 2 ,a_j}) = 
2  \dbinom{a_j-b_j}{a_j} i^{a_j} 
\int_{x \in \R^*} \lambda_{(a_j-b_j+2)}\left(\left(\begin{array}{cc} x & 0 \\ 0 & 1 \end{array}\right)\right) |x|^{s-\frac12} dx.
$$
It is well-known that the zeta-integral of `the lowest weight vector' $\lambda_{(a_j-b_j+2)}$ gives the local $L$-factor $L(s, \Pi_{\eta_j})$, i.e., 
the local factor for the representation $\Pi_{\eta_j} = D_{a_j-b_j+1} \otimes |\ |^{(a_j+b_j)/2}$. (See, for example, Gelbart \cite[Proposition 6.17]{gelbart}.) 
Using the information on local factors in \ref{sec:llc-gl2r}, we get
$$
\zeta_{\eta_j}(s, \lambda_{j, 1 ,a_j}) + \zeta_{\eta_j}(s, \lambda_{j, 2 ,a_j}) = 
4  \dbinom{a_j-b_j}{a_j} i^{a_j} (2\pi)^{-(s + \frac12 + a_j)} \Gamma\left(s + \frac12  + a_j\right). 
$$
The left hand side converges for $\Re{s} \gg 0$ and has a meromorphic continuation to all $s$, and in the right hand side the $\Gamma$-function also makes sense, after continuation, to all $s$; hence we can put $s=1/2$ to get
$$
\zeta_{\eta_j}(1/2, \lambda_{j, 1 ,a_j}) + \zeta_{\eta_j}(1/2, \lambda_{j, 2 ,a_j}) = 
4  \dbinom{a_j-b_j}{a_j} \Gamma(a_j+1) i^{a_j} (2\pi)^{-(a_j+1)} = 
4 \frac{(a_j-b_j)!}{(-b_j)!} (-1)^{a_j} i (2\pi i)^{-(a_j+1)}.
$$
Hence $\langle [\Pi_{\infty}]^{+\!+} \rangle =  c\, (2 \pi i )^{-d_{\infty}} i^n$ where $c$ is the integer: 
$$
c = 4^n \prod_{j=1}^n (-1)^{a_j} \frac{(a_j-b_j)!}{(-b_j)!}.
$$
\end{proof}

Let us note that $\sigma \in {\rm Aut}(\C)$ acts on $\Pi_{\infty}$ by permuting the infinite places; in particular, 
$$
\langle [{}^{\sigma}\Pi_{\infty}]^{+\!+} \rangle  = \langle [\Pi_{\infty}]^{+\!+} \rangle.
$$

This kind of an explicit computation very quickly escalates in complexity when we go from $\GL_2$ to higher $\GL_n$. Indeed, there are many conditional theorems on special values of $L$-functions that have been proved under the assumption that a 
quantity analogous to $\langle [\Pi_{\infty}]^{+\!+} \rangle$ is nonzero. See, for example, 
Kazhdan-Mazur-Schmidt \cite{kazhdan-mazur-schmidt}, Mahnkopf \cite{mahnkopf}, or Raghuram \cite{raghuram}.

\subsubsection{Concluding part of the proof of Theorem~\ref{thm:central}}
We can now finish the proof as follows. Using Proposition~\ref{prop:archimedean-mess} in the main identity of Theorem~\ref{thm:main} we have 
\begin{equation}
\label{eqn:main-identity-cleaned}
\int_{\bar{S}^{G_1}_{R_f}} \mathcal{T}^* \iota^*\vartheta_{\Pi} \ = \ 
\frac{c}{4^n{\rm vol}(R_f)} \frac{L(1/2,\Pi)}{(2 \pi i)^{d_{\infty}}}.
\end{equation}
Apply $\sigma \in {\rm Aut}(\C)$ to both sides, while noting that $c/(4^n{\rm vol}(R_f))$ is a nonzero rational number, to get
$$
\sigma\left( \int_{\bar{S}^{G_1}_{R_f}} \mathcal{T}^* \iota^*\vartheta_{\Pi} \right) \ = \ 
\frac{c}{4^n{\rm vol}(R_f)} \ 
\sigma \left( \frac{L(1/2, \Pi)} {(2 \pi i)^{d_{\infty}}} \right).
$$
Using (\ref{eqn:poincare}), (\ref{eqn:t*-rational}), that $\sigma$ commutes with $\iota^*$--since restriction of a class to a submanifold is a rational operation,  
and using Proposition \ref{prop:vartheta-sigma},  we get that the left hand side simplifies to 
$$
\sigma\left( \int_{\bar{S}^{G_1}_{R_f}} \mathcal{T}^* \iota^*\vartheta_{\Pi} \right) \ = \ 
\frac{\sigma(p^{+\!+}(\Pi))}{p^{+\!+}({}^{\sigma}\Pi)} 
\int_{\bar{S}^{G_1}_{R_f}}\mathcal{T}^* \iota^*\vartheta_{{}^{\sigma}\!\Pi} = 
\frac{\sigma(p^{+\!+}(\Pi))}{p^{+\!+}({}^{\sigma}\Pi)}
\frac{c}{4^n{\rm vol}(R_f)} 
\frac{L(1/2, {}^{\sigma}\!\Pi)}{(2 \pi i)^{d_{\infty}}}, 
$$
where the last equality follows by applying (\ref{eqn:main-identity-cleaned}) for the representation ${}^{\sigma}\Pi$. Hence, we have 
$$
\frac{c}{4^n{\rm vol}(R_f)} \,
\sigma \left( \frac{L(1/2, \Pi)} {(2 \pi i)^{d_{\infty}}} \right)  \ = \ 
\frac{\sigma(p^{+\!+}(\Pi))}{p^{+\!+}({}^{\sigma}\Pi)}
\frac{c}{4^n{\rm vol}(R_f)}
\frac{L(1/2, {}^{\sigma}\Pi)}{(2 \pi i)^{d_{\infty}}}. 
$$
The proof of Theorem~\ref{thm:central} follows easily from this equation.

\subsubsection{Proof of Corollary~\ref{cor:all}}
Let $s = \frac12 + m \in \frac12 + \Z$ be any critical point for $L(s,\Pi)$. Let us note that 
$$
L(s + m, \Pi) = L(s, \Pi\otimes |\ |^m).
$$
Hence $1/2$ is critical for $L(s, \Pi \otimes |\ |^m)$. Now we apply Theorem~\ref{thm:central} to the representation $\Pi \otimes |\ |^m$. There is one nontrivial point to note, i.e., the coefficient system has changed. It is easy to see that 
$$
\Pi \in {\rm Coh}(G,\mu^{\sf v}) \implies \Pi \otimes |\ |^m  \in {\rm Coh}(G, (\mu+m)^{\sf v}), 
$$
where, if $\mu = (\mu_1,\dots, \mu_j)$, with $\mu_j = (a_j,b_j)$, then $\mu+m = (\mu_1+m,\dots, \mu_j+m)$, with $\mu_j = (a_j+m,b_j+m).$ Hence the integer 
$d_{\infty} = d(\Pi_{\infty})$ also changes:
$$
d(\Pi_{\infty} \otimes |\ |^m) = d(\Pi_{\infty}) + mn.
$$
Further, let us note that the main theorem of Raghuram-Shahidi \cite{raghuram-shahidi-imrn}, applied to the special case when the twisting character is 
$|\ |^m$, gives the period relation:  
$$
p^{+\!+}(\Pi \otimes  |\ |^m) = p^{(-1)^m(+\!+)}(\Pi).
$$

Note that twisting by a finite order character $\chi$ of $F^{\times}\backslash \A_F^{\times}$ does not change the set of critical points. 
Corollary~\ref{cor:all} follows by the period relations of Raghuram-Shahidi \cite{raghuram-shahidi-imrn} as recalled in 
\ref{sec:period-relations}.

\section{Hilbert modular forms}
\label{sec:hilbert}

The purpose of this section is to write down an explicit correspondence between primitive holomorphic Hilbert cusp forms and cuspidal automorphic representations of the ad\`ele group of $\GL_2$ over a totally real number field $F$. The precise definitions and basic properties of Hilbert modular forms are discussed in \ref{classic_hilbert}.  In \ref{cusp_to_rep} we attach a cuspidal automorphic representation to a holomorphic Hilbert cusp form, 
and in \ref{rep_to_cusp} we show how to retrieve a classical cusp form from a representation of appropriate type. The rest of the section is devoted to proving the arithmetic properties of this correspondence as stated in Theorem~\ref{thm:dictionary}.

\subsection{Classical holomorphic Hilbert modular forms}\label{classic_hilbert}

	\subsubsection{Some more notations regarding the base field} Let $F$ be a totally real number field of degree $n$, $\ringO=\ringO_F$ the ring of integers in $F$, and $\n$ a fixed integral ideal in $F$. The real embeddings of $F$ are denoted $\eta_j$ with $j=1, \cdots, n$, and we put $\eta=(\eta_1, \cdots, \eta_n)$ with a fixed order of $\{\eta_j\}$. With respect to this $\eta$, $F$ naturally sits inside $\R^n$, and an element $\alpha$ in $F$ will be expressed as $(\alpha_1, \cdots, \alpha_n)$ for $(\eta_1(\alpha), \cdots, \eta_n(\alpha))$ to be considered as an element of $\R^n$. We write $F_+$ for all the totally positive elements in $F$.  Let $k=(k_1, \cdots, k_n)\in\Z^n$, and $\alpha=(\alpha_1, \cdots, \alpha_n)\in\R^n$. We write $\alpha^k$ to mean 
	$\prod_{k=1}^n\alpha_j^{k_j}$.
	
Let $h=h_F$ be the narrow class number of $F$, and let $\{ t_\nu\}_{\nu=1}^{h}$ be elements of $\A_F$ whose infinity part is $1$ and that form a complete set of representatives of the narrow class group. (See Section~\ref{class_gp} for the details.) Put $x_\nu=\left(\begin{array}{ll} 1 & \\ & t_\nu \end{array}\right)$ and $x_\nu^{\iota}=\left(\begin{array}{ll} t_\nu & \\ & 1\end{array}\right)$. Here, $\iota$ denotes the involution defined as $^{\iota}{\rm A}=w_0\, {}^t\!{\rm A}w_0^{-1}$, where $t$ is transpose and $w_0=\left(\begin{array}{ll} & 1 \\ -1 & \end{array}\right)$. 

Let $F_\p$ be the completion of $F$ at a non-archimedean place $\p$, and define a subgroup $\K_\p(\n)$ of $\GL_2(F_\p)$ as
\begin{equation}\label{compact_subgp}
\K_\p(\n):=\left\{ \left( \begin{array}{cc} a & b \\ c & d \end{array}\right) \in \GL_2(F_\p)\, :\, \begin{array}{ccc} a\ringO_\p +\n_\p = \ringO_\p, & b\in\Dif_\p^{-1}, & \\ c\in\n_\p\Dif_\p, & d\in\ringO_\p, & ad-bc\in \ringO^{\times} \end{array} \right\},
\end{equation}
where $\n_\p$ and $\Dif_\p$ are $\p$-parts of $\n$ and the `different' $\Dif$ of $\ringO$, respectively, and put
\[ \K_0(\n):=\prod_{\p<\infty}\K_\p(\n).\] 
Then $\GL_2(\A_F)$ affords a decomposition given as a disjoint union,
\begin{equation}\label{GL_decomp}
 \GL_2(\A_F)=\cup_{\nu=1}^h \GL_2(F)x_\nu^{-\iota}\left(\GL_2^+(F_\infty)\K_0(\n)\right),
\end{equation}
with $\GL_2^+(F_\infty)=\GL_2^+(\R)^n$. 

Define also a congruence subgroup $\Gamma_\nu(\n)$ of $\GL_2(F)$ for each $\nu$ as
\[ \Gamma_\nu(\n)=\left\{ \left( \begin{array}{cc} a & t_\nu^{-1} b \\ t_\nu c & d \end{array}\right) \, : \, \begin{array}{lll} a\in\ringO, & b\in\Dif^{-1}, & \\ c\in\n\Dif, & d\in\ringO, & ad-bc\in\ringO^\times \end{array} \right\}.\]
We note that $\Gamma_\nu(\n)$ can be viewed as
\[ \Gamma_\nu(\n)=x_\nu\left(\GL_2^+(F_\infty)\K_0(\n)\right)x_\nu^{-1} \cap \GL_2(F).\]

	\subsubsection{Hilbert automorphic forms of holomorphic type} Let $\gamma=(\gamma_1, \cdots, \gamma_n)$ be an element of $\GL_2(\R)^n$, and write $\gamma_j=\left(\begin{array}{ll} a_j & b_j \\ c_j & d_j \end{array}\right)$ for each $j=1, \cdots, n$. Then $\gamma$ acts on $\h^n$ by
\[ \gamma.z=\left( \frac{a_1z_1+b_1}{c_1z_1+d_1}, \cdots, \frac{a_nz_n+b_n}{c_nz_n+d_n}\right),\]
with $z=(z_1, \cdots, z_n)\in \h^n$. For a holomorphic function $f$ on $\h^n$, an element $\gamma\in\GL_2(\R)^n$, and $k=(k_1, \cdots, k_n)\in\Z^n$, define
\begin{eqnarray*}
f||_k \gamma (z) &=& \det \gamma^{k/2} {\rm j}(\gamma, z)^{-k} f(\gamma z)     
\end{eqnarray*}
where ${\rm j}(\gamma, z)=cz+d$.

Fix a character $\omega$ of $\dis{\left(\ringO / \n\right)^{\times}}$, and let $\omegatil$ be a character of $\A_F^\times/F^\times$ induced from $\omega$. (See Section~\ref{character} for induced characters.) Then, we define a character of $\K_0(\n)$ by $\omegatil\left(\left(\begin{array}{ll} a & b \\ c & d \end{array}\right)\right)=\omegatil(a)$. We put $\M_k(\Gamma_\nu(\n), \omegatil)$ to be the space of Hilbert modular forms of weight $k=(k_1, \cdots, k_n)$ with respect to $\Gamma_\nu(\n)$, with a character $\omegatil$, by which we mean a space of functions $f_\nu$ that are holomorphic on $\h^n$ and at cusps, and that satisfy $f||_k\gamma=\omegatil(\gamma)f$ for all $\gamma\in\Gamma_\nu(\n)$ considered as elements of $\GL_2(\R)^n$ on the left hand side. Let us note that it makes sense to apply $\omegatil$ to $\Gamma_\nu(\n)$. A function $f_\nu$ in $\M_k(\Gamma_\nu(\n), \omegatil)$ has a Fourier expansion of the form:
\begin{equation}\label{fourier}
 f_\nu(z)=\sum_{\xi}a_\nu(\xi)e^{2\pi i \xi z},
\end{equation}
where $e^{2\pi i \xi z}=\exp\left(2\pi i \sum_{j=1}^n \xi_j z_j\right)$, and $\xi$ runs through all the totally positive elements in $t_\nu\ringO$ and $\xi=0$.
A Hilbert modular form is called a cusp form if, for all $\gamma\in\GL_2^+(F)$, the constant term of $f||_k\gamma$ in its Fourier expansion is $0$, and the space of cusp forms with respect to $\Gamma_\nu(\n)$ is denoted as $\S_k(\Gamma_\nu(\n), \omegatil)$. To have nonempty spaces of cusp forms  
$\S_k(\Gamma_\nu(\n), \omegatil)$, assume henceforth that $k_j \geq 1$. (See, for example, Garrett~\cite[Theorem 1.7]{garrett}.)

Choose a function $f_\nu\in\M_k(\Gamma_\nu(\n), \omegatil)$ for each $\nu$, and put $\f=(f_1, \cdots, f_h)$. Using the decomposition given in (\ref{GL_decomp}), let us define
\begin{equation}
\label{eqn:f-tof-nu}
\f(\gamma x_\nu g_\infty k_0)=(f_\nu||_k g_\infty )({\bf i})\omegatil_f(k_0^\iota),
\end{equation}
where $\gamma\in\GL_2(F)$, $g_\infty\in\GL_2^+(F_\infty)$, $k_0\in\K_0(\n)$, ${\bf i}=(i, \cdots, i)$, and $\omegatil_f$ is a finite part of $\omegatil$. 
The space of such functions $\f$ will be denoted as $\M_k(\n, \omegatil)$. In particular, if $f_\nu\in\S_k(\Gamma_\nu(\n), \omegatil)$ for all $\nu$, then $\f$ is called a cusp form as well, and we write as $\S_k(\n, \omegatil)$ for the space of cusp forms in the ad\`elic setting, i.e., 
$$
\S_k(\n, \omegatil) \ = \ \bigoplus_{\nu = 1}^h \S_k(\Gamma_\nu(\n), \omegatil).
$$

For any integral ideal $\m$ in $F$, there exist a unique $\nu\in \{ 1, \cdots, h\} $ and a totally positive element $\xi$ in $F$ so that $\m=\xi t_\nu^{-1}\ringO$. Put ${\rm c}(\m, \f)=a_\nu(\xi)\xi^{-k/2}$ with $a_\nu(\xi)$ being a Fourier coefficient of $f_\nu$ given in (\ref{fourier}). This is well-defined because the right hand side of the expression is invariant under the totally positive elements in $\ringO^{\times}$. For our convenience, set ${\rm c}(\m, \f)=0$ if $\m$ is not integral. 	
	
	\subsubsection{Primitive form}\label{primitive_form} Let $\f$ be a cuspform of weight $k=(k_1, \cdots, k_n)$, level $\n$, with a character $\omegatil$. For each finite place $\p$, let $\varpi_\p$ be a uniformizer for $\ringO_\p$, and let us define the Hecke operator $\T_\p$ by
\[ (\T_\p\f)(g)=\int_{\K_\p(\n)} \f\left(g k_p\left( \begin{array}{cc} \varpi_\p & \\ & 1 \end{array} \right)\right)\omegatil^{-1}(k_\p)\, dk_\p.\]

Suppose that $\p$ does not divide either $\n$ or $\Dif$, then observe that $\K_\p(\n)=\GL_2(\ringO_\p)$, $\omegatil|_{\ringO_\p^\times}\equiv 1\!\!1$, and that $\f$ is right $\K_\p(\n)$-fixed. Therefore, it follows that
\[ (\T_\p\f)(g)=\int_{K_\p(\n)\left( \begin{array}{ll} \varpi_\p & \\ & 1\end{array} \right) \K_\p(\n)} \f(gh)\, dh.  \]
Furthermore, decomposing the double coset $\K_\p(\n)\left(\begin{array}{ll} \varpi_\p & \\ & 1 \end{array}\right)\K_\p(\n)$ as a disjoint union of right cosets,

\begin{equation}\label{right_coset}
 \K_\p(\n)\left( \begin{array}{ll} \varpi_\p & \\ & 1\end{array}\right) \K_\p(\n)=\left( \begin{array}{ll} 1 & \\ & \varpi_\p \end{array} \right) \K_\p(\n) \cup \left( \cup_{u\in \ringO / \p} \left( \begin{array}{ll} \varpi_\p & u \\ & 1 \end{array} \right) \K_\p(\n) \right). 
\end{equation}
Hence, $\T_\p\f$ can be also described as the finite sum 
\begin{equation}\label{hecke_sum}
 (\T_\p\f)(g)=\f\left( g \left( \begin{array}{ll} 1 & \\ & \varpi_\p \end{array}\right) \right)+\sum_{u\in \ringO / \p} \f\left( g\left( \begin{array}{ll} \varpi_\p & u \\ & 1 \end{array} \right)\right). 
\end{equation}

Now, we shall recall the definition of new forms from Shimura \cite{shimura-duke}.  Let $\m$ be an integral ideal that divides $\n$ and is divisible by the conductor of $\omegatil$, and ${\bf g}\in \S_k(\m, \omegatil)$. Let $\mathfrak{a}$ be an integral ideal dividing $\m^{-1}\n$ that is generated by an element $\alpha\in \A_F^\times$ with $\alpha_\infty=1$. Define ${\bf g}_{\mathfrak{a}}$ by the right translation of ${\rm N}(\mathfrak{a})^{-k_0/2}{\bf g}$ by $\left(\begin{array}{cc} \alpha^{-1} & \\ & 1 \end{array}\right)$. Such ${\bf g}_\mathfrak{a}$ is an element in 
$\S_k(\mathfrak{a}\m, \omegatil)$. The space $\S_k^{{\rm old}}(\n, \omegatil)$ generated by all such ${\bf g}_\mathfrak{a}$ is called the space of old forms. The space $\S_k^{{\rm new}}(\n, \omegatil)$ of new forms is defined to be the orthogonal complement of $\S_k^{{\rm old}}(\n, \omegatil)$ with respect to an inner product:
\[ \langle \f, \, {\bf g}\rangle:=\sum_{\nu=1}^h \frac{1}{\mu\left(\Gamma_\nu\backslash \h^n\right)} \int_{\Gamma_\nu\backslash\h^n} \overline{f_\nu(z)}g_\nu(z)y^k\, d\mu(z),\]
where $d\mu(z)=\prod_{j=1}^n \frac{dx_jdy_j}{y_j^2}$.

A Hilbert cusp form $\f$ in $\S_k(\n, \omegatil)$ is said to be primitive if it is a newform, a common eigenfunction of all the Hecke operators 
$\T_\p$, and normalized so that ${\rm c}(\ringO, \f)=1$.

	\subsubsection{Remarks on primitive forms} We introduced the Hecke operators $\T_\p$ for all the prime ideals $\p$ in the previous section. Now, define more generally the Hecke operators $\T_\m$ for any integral ideal $\m$. Let $\K=\GL_2^+(F_\infty)\cdot\K_0(\n)$, where $\K_0(\n)$ is as defined earlier. We also let
\[ {\rm Y}_\p:=\left\{ \left(\begin{array}{cc} a & b \\ c & d \end{array}\right) \in \GL_2(F_\p)\, : \, \begin{array}{cc}  a\ringO_\p+\n_\p=\ringO_\p, &  b\in\Dif_\p^{-1} \\ c\in\n_\p\Dif_\p, & d\in\ringO_\p \end{array} \right\},\]
and ${\rm Y}=\left(\GL_2(F_\infty)\cdot \prod_\p {\rm Y}_\p\right)\cap \GL_2(\A_F)$. 

The Hecke operator $\T_\m$ is given by $\T_\m=\sum_y \K y\K$ where the sum is taken over all the representatives $y$ of the double cosets $\K y \K$ 
with $y\in{\rm Y}$ satisfying $(\det y)\ringO=\m$. Noting that each summand $\K y\K$ can be written as a disjoint union $\K y \K=\cup_j \K y_j$ with the infinite part of $y_j$ being $1$, we define
\[ \left(\f| \K y\K\right)(g)=\sum_j \omega^\prime(y_j)^{-1}\f(gy_j^\iota),\]
where $\omega^\prime\left(\left(\begin{array}{cc} a & b \\ c & d \end{array}\right)\right)=\omega(a_\n \mod \n)$. This definition coincides with the integral definition for all the Hecke operators $\T_\p$ with respect to prime ideals $\p$.

Miyake proved that if two newforms $\f$ and ${\bf g}$ are common eigenfunctions for $\T_\p$ and share the same eigenvalues for almost all prime $\p$, then $\f$ and ${\bf g}$ are a constant multiple of each other. In particular, if they are normalized, we have $\f={\bf g}$. Furthermore, if a newform $\f$ is normalized and a common eigenfunction for $\T_\p$ for all $\p$ not dividing $\n$, then it is an eigenfunction for all $\T_\m$ and its eigenvalue is 
${\rm N}(\m){\rm c}(\m, \f)$. (See \cite{miyake} and \cite{shimura-duke}.)

Suppose $\f = (f_1,\dots,f_h)$ is a primitive form. One may ask whether $\f$ is determined by any one of its components $f_{\nu}$. In general this is not true. For example, take $\chi$ to be a non-trivial character of the narrow class group, and put ${\bf g} = \f \otimes \chi$, i.e., for any $x \in G(\A)$, 
${\bf g}(x) = \f(x) \chi({\rm det}(x))$. Using (\ref{eqn:f-tof-nu}) it is trivial to check that $g_1 = f_1$, however, in general $\f \neq {\bf g}$.  
After we prove the correspondence $\f \leftrightarrow \Pi(\f)$, it will follow that $\Pi(\f \otimes \chi) = \Pi(\f) \otimes \chi$, and so if $\Pi(\f)$ admits a self-twist, then the twisting character must be quadratic, and $\Pi(\f)$ has to come via automorphic induction from the corresponding quadratic extension of $F$, and in general this would not be the case for a given $\f$. On a related note, one can make an interesting observation based on a refined strong multiplicity one theorem due to Ramakrishnan \cite{ramakrishnan}: suppose, $\f$ and ${\bf g}$ are primitive forms, and suppose $f_\nu = g_\nu$ for all $\nu$ 
except, say, $\nu = \nu_0$. This means that $c(\p,\f) = c(\p, {\bf g})$ for all prime ideals $\p$ whose class in the narrow class group is not 
represented by $t_{\nu_0}^{-1}$, or in other words, $c(\p,\f) = c(\p, {\bf g})$ for all prime ideals $\p$ outside a set $S$ of finite places 
with Dirichlet density $1/h$. (See, for example, Koch \cite[Theorem 1.111]{koch}.)  It follows from Ramakrishnan's theorem that if the narrow class number is sufficiently large ($h > 8$ will do) then necessarily $\f = {\bf g}$.

	\subsubsection{Some notes on the various characters}\label{character} Fix a character $\omega$ of $\left(\ringO /\n\right)^\times$. We ``lift" it to a character $\omegatil$ of $\A_F^\times/F^\times$ as follows.
	Write $\A_F^\times/F^\times$ as a disjoint union
	\[ \A_F^\times/F^\times=\bigcup_{\nu=1}^h t_\nu F_{\infty^+}^\times\prod_{\p<\infty}\ringO_\p^\times, \]
where $\{t_\nu\}$ are taken to be a set of representatives of the narrow class group, and consider the following diagram where the row is exact:

\[\renewcommand{\arraystretch}{1.3}
\begin{array}{ccc} 1 \To & \frac{F^\times(F^\times_{\infty^+}\prod\ringO_\p^\times)}{F^\times} & \hookrightarrow \frac{\A_F^\times}{F^\times} \To \frac{\A^\times_F}{F^\times(F^\times_{\infty^+}\prod\ringO_\p^\times)} \To 1 \\
& \Big\downarrow & \\
& \frac{F^\times(F^\times_{\infty^+}\prod\ringO_\p^\times)}{F^\times\left(F^\times_{\infty^+}\prod(1+\p^{f_\p}\ringO_\p)\right)} & \\
& \downarrow & \\
& \left( \ringO/\n\right)^\times & \\
& \downarrow{\scriptstyle\omega} & \\
& \C^* &
\end{array}
\]
Here $f_\p$ is the highest power of $\p$ dividing $\n$. Using the column, a character $\omega$ of $ \left(\ringO/\n\right)^\times$ can be inflated up to a 
character, also denoted $\omega$, of $F^\times(F^\times_{\infty^+}\prod\ringO_\p^\times)/F^\times$. 
Denote this latter group tentatively by $H$, and observe that it is a subgroup of finite index inside the abelian group $G := \A_F^\times/F^\times$; the index is 
the narrow class number $h$. The representation $\dis{{\rm Ind}_H^{G}(\omega)}$ is a direct sum of $h$ characters, and we can take 
$\omegatil$ to be any such character. We will say that $\omegatil$ is a character of $\A_F^\times/F^\times$ which restricts to the 
character $\omega$ of $ \left(\ringO/\n\right)^\times$.

\subsection{Attaching a cuspidal automorphic representation}
\label{cusp_to_rep} 
Let $L_0^2(\GL_2(F)\backslash \GL_2(\A_F), \omegatil)$ be the space of functions on $\GL_2(\A_F)$ such that 
\begin{eqnarray*}
\phi(\gamma g)&=& \phi(g) \hskip 0.3in \text{for all} \, \gamma\in\GL_2(F), \\
\phi(zg) &=& \omegatil(z)\phi(g) \hskip 0.3in \text{for all}\, z\in\A_F^\times,
\end{eqnarray*}
$\phi$ is square integrable modulo the center, and $\phi$ satisfies the cuspidality condition
\[\int_{F^\times\backslash\A_F^\times} \phi\left(\left( \begin{array}{cc} 1 & x \\ & 1 \end{array}\right) g\right)\,dx=0 \]
for almost all $g$ in $\GL_2(\A_F)$. The regular representation of $\GL_2(\A_F)$ on the space $L_0^2(\GL_2(F)\backslash \GL_2(\A_F), \omegatil )$ 
will be denoted $\rho_0^{\omegatil}$. 

Let $\f$ be a primitive holomorphic Hilbert cusp form of weight $k=(k_1, \cdots, k_n)$, level $\n$, with a Hecke character $\omegatil$. Let $\H(\f)$ be a space spanned by right translations of $\f$ under $\GL_2(\A_F)$. Then the resulting representation $\Pi(\f)$ on this space $\H(\f)$ occurs in the regular representation $\rho_0^{\omegatil}$ on the cusp forms. The goal of this section is to prove the following theorem.

\begin{thm}\label{attaching_rep}
With the notions above, the representation $\Pi(\f)$ on the space $\H_\f$ is irreducible. Furthermore, the local representation $\Pi_{\eta_j}$ at each archimedean place $\eta_j$ is the discrete series representation $D_{k_j-1}$ of lowest weight $k_j$.
\end{thm}

To prove the first part of this theorem, let us recall some important theorems regarding automorphic representations. 
(See, for example, Cogdell \cite{cogdell}.)

\begin{thm}[Multiplicity One Theorem]
\label{mult_one}
The representation $\rho_0^{\omegatil}$ decomposes as the direct sum of irreducible representations, each of which appear with multiplicity one.
\end{thm}

\begin{thm}[Tensor Product Theorem]\label{tensor_prod}
Let $(\Pi, V_\Pi)$ be an automorphic representation of $\GL_2(\A_F)$. Then $\Pi$ is the restricted tensor product of the local representations $\Pi_v$, where $v$ runs through all the places of $F$, and each $\Pi_v$ is an irreducible admissible representation of $\GL_2(F_v)$.
\end{thm}

\begin{thm}[Strong Multiplicity One Theorem]\label{strong_mult}
Let $(\Pi, V_\Pi)$ and $(\Pi^\prime, V_\Pi^\prime)$ be irreducible admissible constituents of the regular representation of $\GL_2$ on the cusp forms. If $\Pi_v$ is equivalent to $\Pi_v^\prime$ for almost all non-archimedean places $v$, then $\Pi\isom\Pi^\prime$.
\end{thm}

Theorem~\ref{mult_one} and Theorem~\ref{tensor_prod} guarantee that $\Pi(\f)$ can be written as 
$\dis{\Pi(\f)=\oplus_i \Pi^i}$, with each irreducible constituent $\Pi^i$ being a restricted tensor product of local representations $\Pi_v^i$. Therefore, in order to show that $\Pi(\f)$ is irreducible, it is now enough to show that $\Pi_v^i\isom\Pi_v^j$ for almost all non-archimedean places $v$ and for all $i$ and $j$ by Theorem~\ref{strong_mult}. Write $\f=\oplus_i\f^i$ with each $\f^i$ in the space of $\Pi^i$. Now consider an irreducible constituent $\Pi^i$.

Let $\p$ be a prime ideal of $F$ not dividing either $\n$ or the different $\Dif$. For such an ideal $\p$, $\Pi_\p^i$ is a spherical representation 
$\pi(\chi_{1, \p}, \chi_{2, \p})$ induced from some unramified characters $\chi_{1, \p}$ and $\chi_{2, \p}$. (We will work with only normalized parabolic induction.) Since $\f$ is an eigenfunction of $\T_\p$, so is $\f^i$, since the projection from $\Pi(\f)$ to the $i$-th coordinate in $\oplus_i \Pi^i$ is a 
Hecke-equivariant map. Furthermore, it can be seen that the eigenvalue is $q_\p^{1/2} \left( \chi_{1, \p}(\varpi_\p)+\chi_{2, \p}(\varpi_\p)\right)$, where $q_\p$ is the cardinality of the residue field $\ringO_\p / \p\ringO_\p$, and $\varpi_\p$ is a uniformizer. Indeed, applying $g=1$ in (\ref{hecke_sum}), we obtain that 

\begin{eqnarray}
(\T_\p\f^i)(1)&=& \f^i\left(\begin{array}{ll} 1 & \\ & \varpi_\p\end{array}\right)+ \sum_{u\in\ringO_\p}\f^i\left(\begin{array}{ll} \varpi_\p & u \\ & 1 \end{array}\right) \label{hecke_ope} \\
&=&  q_\p^{1/2} \left( \chi_{1, \p}(\varpi_\p)+\chi_{2, \p}(\varpi_\p)\right) \f^i(1). \nonumber
\end{eqnarray}
This shows that $\chi_{1, \p}(\varpi_\p)+\chi_{2, \p}(\varpi_\p)=q_\p^{1/2}{\rm c}(\p, \f)$, and that, together with $\chi_{1, \p}\chi_{2, \p}$ being the central character of $\Pi_\p^i$, the characters $\chi_{1, \p}$ and $\chi_{2, \p}$ are uniquely determined by $\f$ and they are independent of $i$. Hence 
$\Pi_\p^i = \Pi_\p^j$ for almost all $\p$. Strong multiplicity one implies that $\Pi^i \isom \Pi^j$, and multiplicity one will imply that $\Pi(\f)$ is irreducible, 
This completes the proof for the first part of Theorem~\ref{attaching_rep}.

For archimedean places, note that the local representation $\Pi_{\eta_j}$ at each place $\eta_j$ is a $(\mathfrak{gl}(2), \O(2))$-module, so it is enough to consider the eigenvalue $\lambda_j$ for the Casimir operator $\Delta_j=-y_j^2\left(\frac{\partial^2}{\partial x_j^2}+\frac{\partial^2}{\partial y_j^2}\right)-y_j\frac{\partial^2}{\partial x_j\partial \theta_j}$. Since $\Delta_j$ acts on $\f$ as a function on $\GL_2(F_{\eta_j})$, we only need to see the action on $(f_\nu||_k g_j)({\bf i})$ for each $\nu$. Writing $g_j= \left( \begin{array}{cc} y_j^{1/2} & x_j y_j^{-1/2} \\ & y_j^{-1/2} \end{array} \right) \left( \begin{array}{cc} \cos \theta_j & -\sin\theta_j \\ \sin\theta_j & \cos\theta_j \end{array} \right)$, a direct computation shows that $\lambda_j=\frac{k_j}{2}\left(1-\frac{k_j}{2}\right)$ for any $\nu$. 
An irreducible admissible infinite-dimensional representation of $\GL_2(\R)$, with infinitesimal character determined by 
$\frac{k_j}{2}\left(1-\frac{k_j}{2}\right)$ and central character trivial on $\R_{>0}$ has to be the discrete series representation $D_{k_j-1}$.  
This says that $\Pi_{\eta_j} = D_{k_j-1}$. (Infinite-dimensionality of $\Pi_{\eta_j}$ is guaranteed by the existence of  Whittaker models.)

\subsection{Retrieving a Hilbert modular form from a representation}
\label{rep_to_cusp} Let $(\Pi, V_\Pi)$ be a cuspidal automorphic representation with the central character $\omega$ that is trivial on $F_{\infty+}^{\times}$, 
and such that the representation at infinity is equivalent to $\dis{\otimes_{j=1}^n D_{k_j-1}}$, where $D_{k_j-1}$ is a discrete series representation of the lowest weight $k_j$. Let the conductor of $\Pi$ be $\n$. We note that, for any non-archimedean place $v$ not dividing $\n$, the local representation $\Pi_v$ is equivalent to a spherical representation induced from some unramified character $\chi_{1, v}\otimes \chi_{2, v}$. In order to retrieve a primitive holomorphic Hilbert cusp form from this representation, it is quite useful to consider a Whittaker model of $\Pi$.

Recall from \ref{sec:additive-character} our non-trivial additive character $\psi$ of $\A /F$, and write $\psi_v$ for $v$-component of this character. The isomorphism between the representation space $V_\Pi$ and the Whittaker space $\Whit(\Pi, \psi)$ allows us to determine a unique holomorphic Hilbert cusp form that corresponds to $\Pi$ by choosing a suitable element from each local Whittaker model $\Whit(\Pi_v, \psi_v)$. For almost all $v$, 
$W_v\in \Whit(\Pi_v, \psi_v)$ is a spherical element, and is normalized so that $W_v(k_v)=1$ for all $k_v\in\GL_2(\ringO_v)$. 
The choices for the local vectors should be made in the following manner.

For each archimedean place $v=\eta_j$, let $\W_{\eta_j}$ be the lowest weight vector in $\Whit(\Pi_{\eta_j}, \psi_{\eta_j})$. 
By the lowest weight vector, we shall mean the element given as follows:
\[ 
\W_{\eta_j}\left(\left(\begin{array}{cc} a & \\ & a \end{array}\right)
\left(\begin{array}{cc} 1 & x \\ & 1 \end{array}\right) 
\left(\begin{array}{cc} \cos\theta & -\sin\theta \\ \sin\theta & \cos\theta \end{array}\right)
\left(\begin{array}{cc} y & \\ & 1 \end{array}\right)\right)=\omega_{\eta_j}(a)\psi_{\eta_j}(x)e^{-ik_j\theta}e^{-2\pi y}. 
\]

For a non-archimedean place $\p$, a suitably normalized $\K_\p(\n)$-fixed vector needs to be chosen, 
where $K_{\p}(\n)$ is an open compact subgroup of $\GL_2(F_\p)$ defined in (\ref{compact_subgp}). 
For this purpose, let $f$ and $r$ be the highest powers of $\p$ that divide $\n$ and the different $\Dif_F$, 
respectively. It is clear that $\K_\p(\n)$ can be written as
\[ \K_\p(\n)=\left(\begin{array}{ll} \varpi_\p^{-r} & \\ & 1\end{array}\right)\Gamma_0(\p^f)\left(\begin{array}{ll} \varpi_\p^r & \\ & 1 \end{array}\right), \]
where $\Gamma_0(\p^f)=\left\{\left(\begin{array}{ll} a & b \\ c& d \end{array}\right)\in\GL_2(\ringO_\p) \, : \, c\equiv 0 \,\, \text{mod} \,\, \p^f \right\}$. 
Let $W_\p^{\rm new}$ be the new vector in $\Whit(\Pi_\p, \psi_\p)$, i.e., $W_\p^{\rm new}$ is an element such that 
$\left(\begin{array}{ll} a & b \\ c & d \end{array}\right) \cdot W_\p^{\rm new}=\omega_\p(d)W_\p^{\rm new}$ 
for all $\left(\begin{array}{ll} a & b \\ c& d\end{array}\right)\in\Gamma_0(\p^f)$, and normalized in a way specified below. 
Define $\W_\p= \left(\begin{array}{ll} \varpi_\p^{-r} & \\ & 1\end{array}\right) \cdot W_\p^{\rm new}$. 
Then $\W_\p$ is an ``almost" $\K_\p(\n)$-fixed vector. (Note: this $\W_\p$ is slightly different from the $\W_\p$ of \ref{sec:normalized-newvector}.)
In particular, 
\begin{equation}\label{newvector}
\W_\p\left(\begin{array}{ll} 1 & \\ & 1 \end{array}\right)=W_\p^{\rm new} \left(\begin{array}{ll} \varpi_\p^{-r} & \\ & 1\end{array}\right). 
\end{equation}

We claim that the right hand side of the above expression is not zero for any $\p$, and hence $\W_\p$ can be normalized so that $\W_\p(1)=1$.
In order to show that our claim is true, we first assume that $r=0$, and hence the conductor of $\psi_\p$ is $\ringO_\p$. 
Pass $\W_\p=W_\p$ to the new vector $\kappa_\p$ in the Kirillov model $\mathcal{K}(\Pi_\p, \psi_\p)$ with respect to 
the same additive character $\psi_\p$, by defining $\kappa_\p(x)=\W_\p\left(\begin{array}{ll} x & \\ & 1 \end{array}\right)$, 
and observe $\kappa_\p(1)\neq 0$. (\cite[Section 2.4]{schmidt}) The isomorphism between the Whittaker model 
$\Whit(\Pi_\p, \psi_\p)$ and the Kirillov model $\mathcal{K}(\Pi_\p, \psi_\p)$ 
guarantees that $\W_\p\left(\begin{array}{ll} 1 & \\ & 1 \end{array}\right)\neq 0$. Normalize this vector, (and call it $\W_\p$ again), so that $\W_\p(1)=1$.

Next, let $r>0$. Let $\psi_{\p, \varpi_\p^{-r}}$ be an additive character defined by $\psi_{\p, \varpi_\p^{-r}}(x):=\psi_\p(\varpi_\p^{-r}x)$. Then since the conductor of $\psi_{\p, \varpi_\p^{-r}}$ is $\ringO_\p$, the same argument as above applies to show that 
$W_{\p, \varpi_\p^{-r}}(1)\neq 0$ where $W_{\p, \varpi_\p^{-r}}$ is the new vector in a Whittaker model with respect to $\psi_{\p, \varpi_\p^{-r}}$. Observing that
\[ W_{\p, \varpi_\p^{-r}}\left(\begin{array}{ll} 1 & \\ & 1 \end{array}\right)=W_\p\left(\begin{array}{ll} \varpi_\p^{-r} & \\ & 1 \end{array}\right),\]
the same normalization can be done in this case as well.

Let $\W=\otimes_v \W_v$, which is an element of $\Whit(\Pi, \psi)$. Then there is a corresponding element $\f$ in $V_\Pi$ by the 
usual isomorphism $V_{\Pi} \to \Whit(\Pi,\psi)$. The vectors $\f$ and $\W$ are related by
\begin{equation}
\label{fourier_expansion}
\f(g)=\sum_{\alpha\in F^{\times}} \W \left(\left(\begin{array}{ll} \alpha & \\ & 1 \end{array} \right)g \right).
\end{equation}
Furthermore, we claim that, in the above expression, $\alpha$ only runs through totally positive elements in $F$. To see this, put $g=\left(\begin{array}{ll} y_\infty & \\ & 1 \end{array}\right)$ where $y_\infty$ is an element of $\A$ whose finite part is $1$. Then it is easy to see that $\W\left(\begin{array}{ll} y_\infty \alpha & \\ & 1 \end{array}\right)=e^{-2\pi y_\infty\alpha}\prod_{v<\infty} \W_v\left(\begin{array}{ll} \alpha & \\ & 1\end{array}\right)$ must be zero unless $\alpha$ is totally positive for the summation to be bounded. Hence the Fourier expansion of $\f$ simplifies to:

\begin{equation}\label{W_expansion}
\f(g)=\sum_{\alpha\in F_+^\times}\W \left(\left(\begin{array}{ll} \alpha & \\ & 1 \end{array} \right)g \right).
\end{equation}

The rest of the section will be devoted to show that $\f$ is the desired Hilbert modular form. 

\begin{thm}
Let $\mathcal{A}_0(k, \n, \omegatil)$ be a subspace of $\mathcal{A}_{{\rm cusp}}(\GL_2(F)\backslash\GL_2(\A), \omegatil)$ that consists of elements satisfying the following properties.
	\begin{enumerate}
	\item $\phi(gr(\theta))=e^{-ik\theta}\phi(g)$ where $r(\theta):=\left\{\left(\begin{array}{lr} \cos\theta_j & -\sin\theta_j \\ \sin\theta_j & \cos\theta_j \end{array}\right)\right\}_j \in \SO(2)^n$,
	\item $\phi(gk_0)=\omegatil_f(k_0^\iota)\phi(g)$, where $\omegatil_f$ is the finite part of $\omegatil$ and $k_0\in\K_0(\n)$, and
	\item $\phi$ is an eigenfunction of the Casimir element $\Delta:=(\Delta_1, \cdots, \Delta_n)$ as a function of $\GL_2(\R)^n$, with its eigenvalue $\lambda=\prod_{j=1}^n \frac{k_j}{2}\left(1-\frac{k_j}{2}\right)$.
	\end{enumerate}
Then $\mathcal{A}_0(k, \n, \omegatil)$ is isomorphic to $\S_k(\n, \omegatil)$.
\end{thm}

\begin{proof} These are essentially the same spaces defined from different points of view. To see this, observe first that any holomorphic 
Hilbert cusp form $\f$ is in $\mathcal{A}_0(k, \n, \omegatil)$. So it remains to recover a holomorphic Hilbert cusp form from any element 
$\phi$ in $\mathcal{A}_0(k, \n, \omegatil)$. For each element $g=\gamma x_\nu^{-\iota}g_\infty k_0$ of $\GL_2(\A_F)$, put
\[ f_\nu(z)=\phi(x_\nu^{-\iota}g_\infty)\det g_\infty^{-k/2}{\rm j}(g_\infty, \, {\bf i})^k, \] 
where $g_\infty({\bf i})=z$. Holomorphy of $f_\nu$ can be shown by checking that it is annihilated by the first order differential opeator
$\frac{\partial}{\partial x}+i\frac{\partial}{\partial y}$. (As mentioned in Gelbart \cite[Proof of Proposition 2.1]{gelbart}, details of this argument appear in 
Gelfand--Graev--Piatetski-Shapiro\cite[Chapter 1, Section 4]{gelfand-graev-shapiro}.) A direct computation shows that $\f:=(f_1, \cdots, f_h)\in\S_k(\n, \omegatil)$ and $\f=\phi$.
\end{proof}

Going back to the $\f$ that corresponds to the global Whittaker vector $\W$, it is clear that $\f$ belongs to $\mathcal{A}_0(k, \n, \omegatil)$. 
Indeed the first two conditions follow from (\ref{W_expansion}) immediately, and the third condition holds because 
$\Pi_\infty=\otimes D_{k_j-1}$ and it follows that $\Delta\W_\infty=\lambda\W_\infty$ with $\lambda$ given in the theorem. 
Therefore, it now only remains to show that $\f$ is primitive.

To prove that $\f$ is a newform, suppose that there exists an integral ideal $\mathfrak{m}$ that divides $\n$ and 
such that $\f\in\S_k(\mathfrak{m}, \omegatil)$. Writing $\f =(f_1, \cdots, f_h)$ with $f_\nu\in \S_k(\Gamma_\nu(\mathfrak{m}), \omega)$ 
for $\nu=1, \cdots, h$, it shows that $f_\nu ||_k\gamma =\omegatil(\gamma)f_\nu$ for all $\gamma\in\Gamma_\nu(\mathfrak{m})$ which 
contradicts the fact that the conductor of $\Pi$ is $\n$. 

Next, it needs to be proven that $\f$ is a common eigenfunction of the Hecke operators $\T_\p$ for almost all prime ideals $\p$, namely $\p$ not dividing neither $\n$ nor the different $\Dif$. Recall that for such an ideal $\p$, $\K_\p(\n)=\GL_2(\ringO_\p)$, $\f$ is right $\K_\p(\n)$-fixed, and the local representatoin $\Pi_\p$ is equivalent to a spherical representation, $\pi(\chi_{1, \p}, \chi_{2, \p})$, induced from some unramified characters $\chi_{1, \p}$ and $\chi_{2, \p}$. Let $\f_\p^\circ$ be the normalized spherical vector in the induced model. Then $\f_\p^{\circ}$ is an 
eigenfunction of $\T_\p$ with eigenvalue is $q_\p^{1/2} \left( \chi_{1, \p}(\varpi_\p)+\chi_{2, \p}(\varpi_\p)\right)$. (See (\ref{hecke_ope}).) Hence $W_\p^{\circ}$ is an eigenfunction for 
$\T_\p$ with same eigenvalue. It follows from (\ref{fourier_expansion}) that $\f$ is also an eigenfunction for 
$\T_\p$ with same eigenvalue.

Finally, we will prove that $\f$ is normalized, i.e., $c(\ringO, \f)=1$. Note that $\f$ has an expansion, 
\[ \f\left(\left( \begin{array}{ll} y & x \\ & 1 \end{array}\right)\right)=\sum_{\xi\in F_+^{\times}} c(\xi y\ringO, \f)(\xi y_\infty)^{k/2} e^{-2\pi \xi y_\infty} \mu(\xi x), \] 
where $y\in\A_F^\times$ with $y_\infty\in F_{\infty^+}$, $x\in\A_F$, and $\mu$ is some additive character of $\A_F/F$. 
(See, for example, Garrett \cite{garrett} or Shimura \cite{shimura-duke}.) In particular,
\begin{equation*}
 \f\left(\left( \begin{array}{ll} 1 & \\ & 1 \end{array}\right)\right)= \sum_{\xi\in F_+^{\times}} c( \xi\ringO, \f)\xi^{k/2} e^{-2\pi\xi}.
\end{equation*}
On the other hand, by (\ref{W_expansion}), 
\begin{eqnarray*}
 \f\left(\left( \begin{array}{ll} 1 & \\ & 1 \end{array} \right)\right) &=& \sum_{\alpha\in F_+^{\times}}\W\left( \begin{array}{ll} \alpha & \\ & 1 \end{array} \right) \\
 &=& \sum_{0 \ll \alpha \in F^{\times}} e^{-2\pi\alpha}\prod_{v<\infty}\W_v \left( \begin{array}{ll} \alpha & \\ & 1 \end{array} \right).
\end{eqnarray*}
Comparing these expressions, we obtain ${\rm c}(\ringO, \f)=\prod_{v<\infty}\W_v\left(\begin{array}{ll} 1 & \\ & 1 \end{array}\right)=1$ as desired.

This result, together with the argument in the previous section, completes the proof of the correspondence between 
primitive holomorphic Hilbert modular forms in $\S_k(\n, \omegatil)$
and cuspidal automorphic representations of $\GL_2(\A_F)$ over a totally reall number field $F$ 
satisfying the following conditions: the local representations at infinite places are the discrete series representations $D_{k_j-1}$ of the lowest weight $k_j$ for each $j=1, \dots, n$, the conductor is $\n$, and the central character is trivial on the totally positive elements $F_{\infty^+}$ in $\R^n$.

\subsection{L-functions}\label{L_function}

	\subsubsection{L-function attached to $\f$}\label{L_function_f} 
	Let $\f=(f_1, \cdots, f_h)$ be a primitive holomorphic Hilbert modular form of weight $k$, level $\n$, 
	and with a character $\omegatil$. Recall that ${\rm c}(\m, \f)$ is defined to be $a_\nu(\xi)\xi^{-k/2}$ for any integral ideal $\m=\xi t_\nu^{-1}\ringO$, and it is $0$ for $\m$ not integral. The (finite) $L$-function attached to $\f$ is 
\[ L_f(s, \f)=\sum_{\m} \frac{{\rm c}(\m, \f)}{{\rm N}(\m)^{s-k_0/2}},\]
where $\m$ runs through all the integral ideals of $F$. Let $\omega^*$ be a character of the group of ideals prime to $\n$ defined by 
$\omega^*(\p)=\omegatil(\varpi_\p)$ for all prime ideals $\p$ that do not divide $\n$; and let $\omega^*(\p) = 0$ if $\p$ divides $\n$. 
Then, the $L$-function has an Euler product,
\[ L_f(s, \f)=\prod_{\p}\left(1-{\rm c}(\p, \f){\rm N}(\p)^{-s+k_0/2}+\omega^*(\p){\rm N}(\p)^{k_0-1-2s}\right)^{-1}.\]
The product is taken over all the prime ideals $\p$. Define the local factors at infinite places by
\[ L_{\eta_j}(s, \f)=(2\pi)^{-\left(s-\frac{k_0-k_j}{2}\right)}\Gamma\left(s-\frac{k_0-k_j}{2}\right), \] 
and for convenience write
\[ L_\infty(s, \f)=\prod_{j=1}^n L_{\eta_j}(s, \f)=(2\pi)^{-\left(s-\frac{k_0-k}{2}\right)}\Gamma\left(s-\frac{k_0-k}{2}\right). \]
Define the completed $L$-function by 
\[ L(s, \f)=L_f(s, \f)L_\infty(s, \f).\]
The above definitions are all for $\Re(s) \gg 0$. It is part of standard `Hecke Theory' for Hilbert modular forms that $L(s, \f)$ has an analytic continuation to 
all of $\C$ and satisfies a functional equation of the expected kind.

	\subsubsection{L-functions attached to $\Pi$}\label{L_function_pi} Recall the definition of  the $L$-function attached to a cuspidal automorphic representation $\Pi$. First recall the $\GL_1$-theory. For a Hecke character $\chi=\otimes_v \chi_v$, the local $L$-factors at the finite places are given by
\[ \begin{array}{llll}
	L_v(s, \chi_v) & = & (1-\chi_v(\varpi_v)q_v^{-s})^{-1} \hskip 0.3in &  \text{if} \hskip 0.1in \chi_v \hskip 0.1in 
\text{is unramified, and} \\
	L_v(s, \chi_v) & = & 1 & \text{if} \hskip 0.1in \chi_v  \hskip 0.1in \text{is ramified.}
 \end{array} \]

Define the local $L$-factors for $\GL_2$ as follows: if the local representation $\Pi_\p$ at a place $\p$ is equivalent to a principal series representation $\pi(\chi_{1, \p}, \chi_{2, \p})$, then put
\[ L_\p(s, \Pi_\p)=L_\p(s, \chi_{1, \p})L_\p(s, \chi_{2, \p}).\]
Note that both factors are non-trivial if and only if $\Pi_\p$ is spherical. For the other places, define $L_\p(s, \Pi_\p)=1$ for a supercuspidal representation $\Pi_\p$, and 
\[ L_\p(s, \Pi_\p)=L_\p(s + 1/2, \chi_\p) \]
for $\Pi_\p = {\rm St}_{{\rm GL}_2(F_\p)} \otimes \chi_\p$, the twist of the Steinberg representation ${\rm St}_{{\rm GL}_2(F_\p)}$ by $\chi_\p$. 
(See, for example, Kudla \cite[Section 3]{kudla}.) 
At the infinite places, the factors are 
\[ L_{\eta_j}(s, \Pi_{\eta_j})=(2\pi)^{-\left(s+\frac{k_j-1}{2}\right)}\Gamma\left(s+\frac{k_j-1}{2}\right). \]
Again, we use a multi-index notation and write 
\[ L_\infty(s, \Pi_\infty)=(2\pi)^{-\left(s+\frac{k-1}{2}\right)}\Gamma\left(s+\frac{k-1}{2}\right) \]
to mean the product of all $L_{\eta_j}(s, \Pi_{\eta_j})$. The global $L$-function attached to $\Pi$ is
\[ L(s, \Pi)=\otimes_v L_v(s, \Pi_v).\]
The above definitions are all for $\Re(s) \gg 0$. It is part of standard `Hecke Theory', due to Jacquet and Langlands \cite{jacquet-langlands}, for cuspidal representations of $\GL_2(\A)$ 
that $L(s, \Pi)$ has an analytic continuation to 
all of $\C$ and satisfies a functional equation of the expected kind.

	\subsubsection{Relation between $L(s, \f)$ and $L(s, \Pi)$} Having $L$-functions attached to a Hilbert cusp form $\f$ and to a cuspidal automorphic representation $\Pi$, a natural question to ask is how $L(s, \f)$ and $L(s, \Pi(\f))$ relate to each other where $\Pi(\f)$ is a representation attached to a primitive cusp form $\f$. The main theorem of this section is:

\begin{thm}
Let $\f\in\S_k(\n, \omegatil)$ be primitive, and $\Pi(\f)$ a cuspidal automorphic representation attached to $\f$.
Then the completed $L$-functions attached to $\f$ and attached to $\Pi(\f)$ satisfy the following relation:
\[ L\left(s, \Pi(\f)\right)=L\left(s+\frac{k_0-1}{2}, \f\right), \]
where $k_0={\rm Max}\{k_1, \cdots, k_n\}$. The same relation holds between the finite and infinite parts of the two $L$-functions. 
\end{thm}

\begin{proof}
Let $\Re(s) \gg 0$. For any place $v$ of $F$, and any vector $W_v$ in a local Whittaker model $\Whit(\Pi(\f)_v, \psi_v)$, define a local $\zeta$-integral by
\[ \zeta_v(s, W_v)=\int_{F_v^\times} W_v\left(\begin{array}{cc} \alpha & \\ & 1 \end{array}\right)|\alpha|^{s-1/2} \, d^\times\alpha.\]
A global $\zeta$-integral is similarly defined for $W \in \Whit(\Pi,\psi)$ as
\[ \zeta(s, W) =\int_{\A_F^\times}W\left(\begin{array}{cc} \alpha & \\ & 1 \end{array}\right)|\alpha|^{s-1/2} \, d^\times\alpha.\]
This integral is eulerian, i.e., if the global Whittaker vector $W$ factorizes as $W = \otimes W_v$ into local Whittaker vectors then
$$
\zeta(s, W) =\prod_{v\leq\infty}\zeta_v(s, W_v).
$$

In particular, take $\psi$ to be the additive character that has been fixed, and $W_v$ the normalized new vector $\W_v$ as in 
Section~\ref{rep_to_cusp}. Then, one can show that $L_v(s, \Pi(\f)_v)=\zeta_v(s, \W_v)$. (See, for example, Gelbart~\cite[Proposition 6.17]{gelbart}.)
Therefore $L(s, \Pi(\f))=\prod_v L_v(s, \Pi(\f))=\prod_v \zeta_v(s, \W_v)=\zeta(s, \W)$. On the other hand, we have
\begin{eqnarray*}
\int_{\A_F^\times/F^\times} \f\left(\begin{array}{cc} y & \\ & 1 \end{array}\right) |y|^{s-1/2}\, d^\times y &=& \int_{\A_F^\times/F^\times} \sum_{\alpha\in F^\times} \W\left(\begin{array}{cc} \alpha y & \\ & 1 \end{array}\right)|\alpha y|^{s-1/2}\, d^\times\alpha \\
&=& \int_{\A_F^\times}\W\left(\begin{array}{cc} \alpha & \\ & 1 \end{array}\right)|\alpha|^{s-1/2}\, d^\times\alpha.
\end{eqnarray*}

We recall that $\A_F^\times/F^\times=\cup_{\nu=1}^h t_\nu^{-1}F^\times_{\infty^+}\prod\ringO_\p^\times$ (a disjoint union), and it follows that, for any $y\in\A_F^\times$, 
\[ \f\left(\begin{array}{cc} y & \\ & 1 \end{array}\right)|y|^{s-1/2}=y_\infty^{k/2}f_\nu({\bf i}y_\infty)\left(|t_\nu|^{-1}|y_\infty|\right)^{s-1/2},\]
with a unique $\nu$ where $y_\infty$ is the infinite part of $y$ and ${\bf i}=(i,\dots, i)$. Hence
\begin{eqnarray*}
L(s, \Pi(\f)) &=& \int_{\A_F^\times/F^\times} \f\left(\begin{array}{cc} y & \\ & 1 \end{array}\right) |y|^{s-1/2}\, d^\times y \\
&=& \sum_{\nu=1}^h\int_{F_{\infty^+}} f_\nu({\bf i}y)y^{s+\frac{k-1}{2}}|t_\nu|^{-(s-1/2)}\, \frac{dy}{y}. 
\end{eqnarray*}
Applying the Fourier expansion $f_\nu(z)=\sum_{\xi} a_\nu(\xi)\exp(2\pi i \xi z)$, the proof can be completed as follows.
\begin{eqnarray*}
L(s, \Pi(\f)) &=& \sum_{\nu, \xi} \frac{a_\nu(\xi)\xi^{-k/2}}{(2\pi \xi)^{s-1/2}|t_\nu|^{-(s-1/2)}}\int_{F_{\infty^+}} e^{-y}y^{s+\frac{k-1}{2}}\, \frac{dy}{y} \\
&=& (2\pi)^{-\left(s+\frac{k-1}{2}\right)}\Gamma\left(s+\frac{k-1}{2}\right) \sum_{\m} \frac{{\rm c}(\m, \f)}{{\rm N}(\m)^{s-1/2}} \\
&=& L\left(s+\frac{k_0-1}{2}, \f\right).
\end{eqnarray*}
The equality is for $\Re(s) \gg 0$. Both sides have analytic continuation to all of $\C$ and hence we have equality everywhere.

From the definitions of the infinite parts of the two $L$-functions, we see that 
$$
L_{\infty}(s, \Pi(\f)_{\infty}) = L_{\infty}\left(s+\frac{k_0-1}{2}, \f\right). 
$$
It follows that the same relations hold for the finite part since $L_f(s,\Pi) = L(s,\Pi)/L_{\infty}(s,\Pi_{\infty})$, and similarly, 
$L_f(s,\f) = L(s,\f)/L_{\infty}(s,\f).$
\end{proof}

\subsection{The action of ${\rm Aut}(\C)$ and rationality fields}

	\subsubsection{The action of ${\rm Aut}(\C)$ on Hilbert modular forms} Let $\sigma$ be an automorphism of $\C$, and
	 let it act on $\R^n = \prod_{j=1}^n F_{\eta_j}$ by permuting the coordinates. 
	Then $\sigma\circ\eta$ gives another embedding of $F$ into $\R^n$. Considering $\sigma(\xi^k)=\prod_{j=1}^n \sigma(\eta_j(\xi))^{k_j}$ for $\xi\in F$ and $k\in\Z^n$, we can view $\sigma$ as a permutation of $\{k_j\}$. We will use this identification from now on, and denote it as $k^\sigma$. 

Let $f$ be a Hilbert modular form of weight $k$, level $\n$, with a character $\omegatil$, and write its Fourier expansion as $\dis{f(z)=\sum_{\xi}a_\nu(\xi)\exp(2\pi i\xi z)}$. We define $f^\sigma$ to be 
\[ f^\sigma (z)=\sum_\xi a_\nu^\sigma(\xi)\exp(2\pi i \xi z_j),\]
with $a_\nu^\sigma(\xi)=\sigma(a_\nu(\xi))$. We have the following

\begin{prop}[Shimura, \cite{shimura-duke}]
Let $\sigma \in {\rm Aut}(\C)$. If $f\in\M_k(\Gamma, \omega)$, then $f^\sigma\in\M_{k^\sigma}(\Gamma, \omega^\sigma)$, where $\omega^\sigma=\sigma\circ\omega$.
\end{prop}
	
In order to attain a similar result in the ad\`elic setting, we normalize $f^\sigma$ as follows: For $f_\nu\in\M_k(\Gamma_\nu, \omega)$ with $\Gamma_\nu$ defined in Section~\ref{classic_hilbert}, put
\[ f_\nu^{[\sigma]}=f_\nu^\sigma\cdot\left({\rm N}(t_\nu)^{k_0/2}\right)^\sigma{\rm N}(t_\nu)^{-k_0/2}, \] 
where $k_0={\rm Max}\{k_1, \cdots, k_n\}$. If $\f$ is a holomorphic Hilbert modular form given as $\f=(f_1, \cdots, f_h)$, we define $\f^\sigma$ to be $\f^\sigma=(f_1^{[\sigma]}, \cdots, f_h^{[\sigma]})$.
	
\begin{prop}[Shimura, \cite{shimura-duke}]\label{f_sigma}
Let $\f=(f_1, \cdots, f_h)$ be in $\M_k(\n, \omegatil)$, and $\sigma  \in {\rm Aut}(\C)$. Then $\f^\sigma\in\M_{k^\sigma}(\n, \omegatil^\sigma)$, and 
$N(\m)^{k_0/2}{\rm c}(\m, \f^\sigma)=(N(\m)^{k_0/2}{\rm c}(m, \f))^\sigma$. Furthermore, $\f^\sigma$ is primitive whenever $\f$ is.
\end{prop}

	\subsubsection{${\rm Aut}(\C)$-equivariance of the dictionary} Proposition~\ref{f_sigma} guarantees that $\f^\sigma$ is a primitive holomorphic Hilbert modular form if $\f$ is. Therefore, by the bijection discussed in Section~\ref{cusp_to_rep} and Section~\ref{rep_to_cusp}, there exists a cuspidal automorphic representation $\Pi(\f^\sigma)$ of a certain type that corresponds to $\f^\sigma$. Now, the question is: how one can compare the ${\rm Aut}(\C)$-action on the space of Hilbert modular forms with the ${\rm Aut}(\C)$-action on the space of cuspidal automorphic representations? The obvious guess that  
	$\Pi(\f^\sigma) = \Pi(\f)^\sigma$ is not quite correct; indeed, $\Pi(\f)^\sigma$ may not even be an automorphic representation. The following theorem answers our question; the proof also includes verifications of some of the arithmetic properties stated in Theorem~\ref{thm:dictionary}.
	
\begin{thm}\label{c_equiv}
Let $\f\in\S_k(\n, \omegatil)$ be primitive, with $k=(k_1, \cdots, k_n)$. Assume that $k_1\equiv\cdots\equiv k_n \mod \, 2$. Then the map $\f\mapsto \Pi(\f)\otimes|\ |^{k_0/2}$ is ${\rm Aut}(\C)$-equivariant, where $k_0={\rm Max}\{k_1, \cdots, k_n\}$.
\end{thm}

\begin{proof}
First, let us note that $\Pi(\f)$ is algebraic if $k_0\equiv 0 \mod \, 2$, and $\Pi(\f) \otimes |\ |^{1/2}$
is  algebraic when $k_0\equiv 1 \mod \, 2$; these follow easily from \ref{sec:algebraic}. These cases may be uniformized by considering the twist 
$\Pi(\f)$ by $|\ |^{k_0/2}$ to say that $\Pi(\f)\otimes |\ |^{k_0/2}$ is an algebraic cuspidal automorphic representation for all $k$ that satisfy the parity condition in the hypothesis. Further, if $k_j \geq 2$ for all $j$, then $\Pi(\f)\otimes|\ |^{k_0/2}$ is a regular algebraic cuspidal automorphic representation; this can be seen immediately from \ref{sec:reg-alg-coh} after one notes that 
$\Pi(\f)\otimes|\ |^{k_0/2} \in {\rm Coh}(G, \mu^{\sf v})$, where the weight $\mu = (\mu_1,\dots,\mu_n)$ is given by: 
$$
\mu_j = \left( \frac{k_0+k_j-2}{2}, \frac{k_0-k_j+2}{2} \right).
$$
(Let us add a comment about $k_j \geq 2$. Even in the elliptic modular case, a weight $1$ form is not of motivic type; another facet of the same phenomenon is that the associated representation after twisting by $|\ |^{1/2}$ is algebraic but not regular; or that the associated $L$-function has no critical points.) 

By Theorem~\ref{thm:q-pi}, the representation $(\Pi(\f)\otimes|\ |^{k_0/2})^\sigma$ is also a regular algebraic cuspidal automorphic representation.
Let us note that this representation, however, does not have an appropriate central character to apply the ``dictionary." In order to modify the central character, twist it by $|\ |^{-k_0/2}$ and consider the representation $\Pi^\prime:=(\Pi(\f)\otimes|\ |^{k_0/2})^\sigma\otimes|\ |^{-k_0/2}$. This representation is cuspidal and automorphic, whose local representations at infinity places are 
$\dis{\otimes_\eta \Pi(\f)_{\sigma^{-1}\eta}}$, i.e., the permutation of the discrete series representations $\{D_{k_j-1}\}$, and such that the conductor is $\n$, and that the central character is trivial on $F_{\infty^+}$. Therefore, by Section~\ref{rep_to_cusp}, there is a primitive holomorphic Hilbert modular form of weight $k^\sigma$ and level $\n$. It remains to show that this cusp form is actually $\f^\sigma$, and that the central character of 
$\Pi^\prime$ is $\omegatil^\sigma$. 

By Theorem~\ref{strong_mult}, it is enough to show that $\Pi^\prime_\p$ coincides with $\Pi(\f^\sigma)_\p$ for almost all finite places $\p$. In particular, let $\p$ be a place of $F$ that does not divide $\n$. Then, the local representation $\Pi^\prime_\p$ at $\p$ is a spherical representation, say induced from $\chi_{1, \p}^\prime$ and $\chi_{2, \p}^\prime$, and write it as $\Pi^\prime_\p=\pi\left( \chi_{1, \p}^\prime, \, \chi_{2, \p}^\prime\right)$. We use the following lemma to see these characters more carefully.

\begin{lemma}[Waldspurger, \cite{waldspurger}]
Let $\Pi=\pi(\chi_1, \chi_2)$ be a spherical representation induced from $\chi_1$ and $\chi_2$, then $^\sigma\Pi$ is also spherical, and it is induced from characters defined as $|\,\,|^{-1/2}\sigma(\chi_i\cdot| \,\,|^{1/2})$, where $i=1$, $2$.
\end{lemma}

Let $\Pi(\f)_\p=\pi(\chi_{1, \p},\,\chi_{2, \p})$. By the lemma above, the characters $\chi_{i, \p}^\prime$ can be described as  
\[ \chi_{i, \p}^\prime=|\,\,|^{-\frac{k_0+1}{2}}_\p\sigma\left(\chi_{i, \p}\cdot|\,\,|^{\frac{k_0+1}{2}}_\p\right). \] 
Therefore, a direct computation shows that
\begin{eqnarray*}
\left( \chi_{1, \p}^\prime +\chi_{2, \p}^\prime \right)( \varpi_\p ) &=& q_\p^{-\frac{k_0+1}{2}}\left(q_\p^{\frac{k_0+1}{2}}\right)^\sigma \left(q_\p^{\frac12}{\rm c}(\p, \f)\right)^\sigma  \\
&=& q_\p^{-\frac{k_0-1}{2}}\left(q_\p^{\frac{k_0}{2}}\right)^\sigma{\rm c}(\p, \f)^\sigma \\
&=& q_\p^{1/2}{\rm c}(\p, \f^\sigma ) \hskip 0.5in \text{by (\ref{f_sigma})},
\end{eqnarray*}
and
\[ \chi_{1, \p}^\prime \cdot\chi_{2, \p}^\prime = \sigma(\omegatil).\]

This says that $q_\p^{1/2}\left(\chi_{1, \p}^\prime +\chi_{2, \p}^\prime \right)(\varpi_\p)$ gives the eigenvalue for the Hecke operator $\T_\p$ applied to $\f^\sigma$, which can be seen by the same computation done in (\ref{hecke_ope}), and that $\chi_{1, \p}^\prime \cdot\chi_{2, \p}^\prime$ is the central character $\omegatil^\sigma$ of $\Pi(\f^\sigma)$. This completes the proof of Theorem~\ref{c_equiv}.
\end{proof}

\subsubsection{Rationality fields}
Let $\f$ be a primitive form in $\S_k(\n, \omegatil).$ Define the rationality field of $\f$ as
\begin{equation}
\label{eqn:q-f}
\Q(\f) := \Q(\{N(\m)^{k_0/2}{\rm c}(\m, \f) : \mbox{for all integral ideals $\m$} \}).
\end{equation}
This is the field generated over $\Q$ by the normalized eigenvalues of $\f$ for all the Hecke operators $T_\m$. 
Shimura \cite[Proposition 2.8]{shimura-duke} proves that this field is a number field which is in fact generated by $N(\p)^{k_0/2}{\rm c}(p, \f)$ for almost all prime ideals $\p$. Further, this number field is either totally real, or a totally imaginary quadratic extension of a totally real number field. 

Let $\Pi = \Pi(\f)$, and we have checked that $\Pi(\f)\otimes | \ |^{k_0/2}$ is a regular cuspidal automorphic representation. Define the rationality field
$\Q(\Pi)$ to be
\begin{equation}
\label{eqn:q-pi}
\Q(\Pi) := \C^{ \{ \sigma \in {\rm Aut}(\C) : {}^{\sigma} \Pi_f = \Pi_f \} }.
\end{equation}
That is the subfield of $\C$ fixed by the group $\{ \sigma \in {\rm Aut}(\C) : {}^{\sigma} \Pi_f = \Pi_f \}$ of all $\C$-automorphisms which fix $\Pi_f$. 
By strong multiplicity one, ${}^{\sigma} \Pi_f = \Pi_f$ if and only if ${}^{\sigma} \Pi_\p = \Pi_\p$ for almost all prime ideals $\p$. 

Using Proposition~\ref{f_sigma} it is clear that 
$\Q(\f)$ is the   subfield of $\C$ fixed by the group $\{ \sigma \in {\rm Aut}(\C) :  \sigma(N(\m)^{k_0/2}{\rm c}(\m, \f)) = N(\m)^{k_0/2}{\rm c}(\m, \f)) \}$. It follows that 
$\Q(\f) = \Q(\Pi(\f))$.

\subsection{Proof of Theorem~\ref{thm:shimura} and period relations}
The proof Theorem~\ref{thm:shimura} is a totally formal consequence of Theorem~\ref{thm:central} plus Corollary~\ref{cor:all}, together with properties of the dictionary as in Theorem~\ref{thm:dictionary}; we leave the details to the reader after observing, as mentioned in the proof of Theorem~\ref{c_equiv} above,
that if $\f \in \S_k(\n,\omegatil)$ and suppose for convenience that all the weights $k_j$ are even, then 
$\Pi = \Pi(\f) \in {\rm Coh}(G,\mu^{\sf v})$ with the highest weight $\mu = (\mu_1,\dots,\mu_n)$ being given by:
$$
\mu_j = ((k_j-2)/2, -(k_j-2)/2) =: (a_j,b_j).
$$

Note that Shimura's periods $u(r, {\bf f})$ and our periods $p^{\epsilon}(\Pi)$ have different definitions. With this in mind, it is interesting to 
see the formal consequences of the fact that these periods 
appear in the critical values of the `same' $L$-function. Using the notation as in Theorem~\ref{thm:dictionary}, when all the weights $k_j$ are even integers, 
then $k_0/2$ is a critical point for $L(s, {\bf f})$, which corresponds to the central critical point $1/2$ for $L(s, \Pi({\bf f}))$. 
We have the following consequence: 
$$
(2\pi i)^{\frac{nk_0}{2}}  u(+\!+, {\bf f}) \ \sim \ 
(2\pi i)^{d_{\infty}} p^{+\!+}(\Pi({\bf f})).
$$
where $\sim$ means up to an element of $\Q({\bf f})^* = \Q(\Pi({\bf f}))^*$.
Note that $d_{\infty} = \sum_j (a_j+1) = (\sum_j k_j)/2$. Hence 
\begin{equation}
\label{eqn:period-relations}
p^{+\!+}(\Pi({\bf f})) \ \sim \ 
(2\pi i)^{\sum_j (k_0-k_j)/2} \, u(+\!+, {\bf f}). 
\end{equation}
Twisting by a quadratic character of prescribed signature, one can deduce a similar relation between $u(\epsilon, {\bf f})$ and $p^{\epsilon}(\Pi)$ for 
any $\epsilon \in (\Z/2\Z)^n$. Similarly, one can deduce period relations when all the weights $k_j$ are odd.


\bigskip

\bigskip

Department of Mathematics,

Oklahoma State University,

401 Mathematical Sciences,

Stillwater, OK 74078, USA. 

\bigskip

{\tt araghur@math.okstate.edu} 

\smallskip

{\tt ntanabe@math.okstate.edu}

\end{document}